\documentclass[a4paper,reqno]{amsart}
\usepackage{amsfonts}

\theoremstyle{plain}
\begingroup
\newtheorem{theorem}{Theorem}[section]
\newtheorem{lemma}[theorem]{Lemma}
\newtheorem{proposition}[theorem]{Proposition}
\newtheorem{corollary}[theorem]{Corollary}
\endgroup

\theoremstyle{definition}
\begingroup
\newtheorem{definition}[theorem]{Definition}
\newtheorem{remark}[theorem]{Remark}

\endgroup

\theoremstyle{remark}
\begingroup

\endgroup

\mathchardef\emptyset="001F

\numberwithin{equation}{section}

\newcommand{\e}{\varepsilon}
\newcommand{\Om}{\Omega}
\newcommand{\Ga}{\Gamma}
\newcommand{\C}{{\mathbb C}}

\newcommand{\R}{{\mathbb R}}
\newcommand{\Rn}{{\R}^2}

\newcommand{\N}{\mathbb N}

\newcommand{\Z}{{\mathbb Z}}
\newcommand{\Mnn}{{\mathbb M}^{2{\times}2}_{sym}}
\newcommand{\MD}{{\mathbb M}^{2{\times}2}_D}
\newcommand{\wto}{\rightharpoonup}

\newcommand{\setmeno}{\!\setminus\!}

\renewcommand{\div}{{\rm div}}
\renewcommand{\hom}{{hom}}
\newcommand{\supp}{{\rm supp}}
\newcommand{\tr}{{\rm tr}}
\newcommand{\hn}{{\mathcal H}^1}
\newcommand{\Ln}{{\mathcal L}^2}
\newcommand{\Lone}{{\mathcal L}^1}

\newcommand{\QQ}{{\mathcal Q}}
\newcommand{\D}{{\mathcal D}}

\newcommand{\G}{{\mathcal G}}
\newcommand{\HH}{{\mathcal H}}

\newcommand{\bary}{{\rm bar}}

\newcommand{\muu}{{\boldsymbol\mu}}
\newcommand{\nuu}{{\boldsymbol\nu}}
\newcommand{\sigmaa}{{\boldsymbol\sigma}}

\newcommand{\lambdaa}{{\boldsymbol\lambda}}
\newcommand{\alphaa}{{\boldsymbol\alpha}}
\newcommand{\ee}{{\boldsymbol e}}
\newcommand{\pp}{{\boldsymbol p}}
\newcommand{\uu}{{\boldsymbol u}}

\newcommand{\ww}{{\boldsymbol w}}
\newcommand{\zz}{{\boldsymbol z}}
\newcommand{\V}{{\mathcal V}}

\newcommand{\ol}{\overline}
\newcommand{\tki}{t_k^i}
\newcommand{\tkim}{t_k^{i-1}}
\newcommand{\wki}{w_k^i}
\newcommand{\uki}{u_k^i}
\newcommand{\eki}{e_k^i}
\newcommand{\pki}{p_k^i}
\newcommand{\zki}{z_k^i}

\newcommand{\Interior}{\mathaccent'27}
\newcommand{\T}{\mathcal T}

\title[Globally stable quasistatic evolution in
 plasticity with softening]
{Globally stable quasistatic evolution in
 plasticity with softening}
\author{G.\ Dal Maso}
\author{A.\ DeSimone}
\author{M.G.\ Mora}
\author{M.\ Morini}
\address[G.~Dal Maso, A.~DeSimone, M.G.~Mora, and M.\ Morini]{SISSA, Via Beirut 2-4, 
34014 Trieste, Italy}
\email[Gianni Dal Maso]{dalmaso@sissa.it}
\email[Antonio DeSimone]{desimone@sissa.it}
\email[Maria Giovanna Mora]{mora@sissa.it}
\email[Massimiliano Morini]{morini@sissa.it}

\begin{document}
\begin{abstract}
We study a relaxed formulation  of the  quasistatic evolution problem in the context of small strain associative elastoplasticity with softening. The relaxation takes place in spaces of generalized Young measures. The notion of solution is characterized by the following properties: global stability at each time and energy balance on each time interval. 
An example developed in detail compares the solutions obtained by this method with the ones provided by a vanishing viscosity approximation, and shows that only the latter capture a decreasing branch in the stress-strain response.

\end{abstract}
\maketitle

{\small

\bigskip
\keywords{\noindent {\bf Keywords:} 
quasistatic evolution, rate independent processes,  Prandtl-Reuss plasticity, plasticity with softening, shear bands, incremental problems, relaxation, Young measures}

\subjclass{\noindent {\bf 2000 Mathematics Subject Classification:} 74C05 (28A33, 74G65, 49J45, 35Q72)}
}
\tableofcontents
\bigskip
\bigskip

\begin{section}{Introduction}

In the study of quasistatic evolution problems for rate independent systems
a classical approach is to approximate the continuous time solution by discrete 
time solutions obtained by solving incremental minimum problems 
(see the review paper \cite{Mie-review} and the references therein).

In this paper we apply this method to the study of a plasticity problem
with softening, where the new feature is given by the presence of some nonconvex energy terms. 
For a general introduction to the mathematical theory of plasticity
we refer to \cite{Han-Red}, \cite{Hill}, \cite{Kachanov}, \cite{Lub}, and \cite{Mar}.
To focus on the new difficulty, due to the lack of convexity, we consider the simplest relevant model, namely
small strain associative elastoplasticity with no applied forces, where the
evolution is driven by a time-dependent boundary condition $\ww(t)$, prescribed on a portion $\Ga_0$ of the boundary of the reference configuration~$\Om\subset\Rn$.

The unknowns of the problem are the displacement $u\colon\Om\to\Rn$, the 
elastic strain $e\colon\Om\to\Mnn$ (the set of symmetric $2{\times}2$ matrices),
the plastic strain $p\colon\Om\to\MD$ (the set of trace free symmetric $2{\times}2$ matrices), and the internal variable $z\colon\Om\to\R$. For every given time $t\in[0,T]$ they are related by the kinematic admissibility conditions: $Eu=e+p$ in $\Om$ (additive decomposition) and $u=\ww(t)$ on $\Ga_0$. The stress depends only on the elastic part $e$ through the usual linear relation $\sigma:=\C e$, where $\C$ is the elasticity tensor.

Given a sequence of subdivisions of a time interval $[0,T]$
$$
0=t_k^0<t_k^1<\dots<t_k^{k-1}<t_k^{k}=T\,,
$$
we assume that an approximate solution $(u^{i-1}_k,e^{i-1}_k,p^{i-1}_k,z^{i-1}_k)$
is known at time $t^{i-1}_k$. The approximate solution $(\uki,\eki,\pki,\zki)$
at time $\tki$ is defined as a solution of the following incremental minimum problem:
\begin{equation}\label{minintro}
\inf_{(u,e,p,z)\in {\mathcal A}(\ww(t^i_k))} \{\QQ(e)+ \HH(p-p^{i-1}_k,z-z^{i-1}_k) +\V(z)
 \}\,,
\end{equation}
where $\QQ$ is the stored elastic energy, $\HH$ is the plastic dissipation rate,
$\V$ is the softening potential, while ${\mathcal A}(\ww(t^i_k))$ is the set of 
functions $(u,e,p,z)$ such that $Eu=e+p$ in $\Om$, $u=\ww(t^i_k)$ on $\Ga_0$, 
and $z\in L^1(\Om)$.

The details of the definition of $\QQ$, $\HH$, $\V$, together with the technical assumptions which are needed for our analysis, are given in Section~\ref{notpre}. For the present discussion it is sufficient to know that $\QQ$ is a quadratic form, $\HH$ is positively homogeneous of degree one, and $\V$ is strictly concave with linear growth.

Due to the nonconvexity of the functional the infimum in \eqref{minintro}
is not attained, in general. To overcome this difficulty,
in this paper we consider a relaxed formulation of this approach (see Proposition~\ref{relax}).
To preserve the continuity of the energy terms it is convenient to cast 
the relaxed problem in the language of Young measures. 
An additional difficulty is due to the linear growth of $\HH$ and $\V$, which may cause concentration effects. 
For this reason we formulate the problem in a suitable space of generalized Young measures
(see \cite[Section~3]{DM-DeS-Mor-Mor-1}).

The next step in our analysis is the study of the convergence of the relaxed approximate solutions
as the time step $t^i_k-t^{i-1}_k\to 0$ as $k\to\infty$ (uniformly with respect to $i$).
We prove that,  up to a subsequence, these solutions converge to a solution of a quasistatic  evolution problem formulated in the framework of generalized Young measures. This is characterized by the usual conditions considered in the variational approach to rate independent evolution problems, namely global stability and energy balance (see Definition~\ref{maindef}), suitably  phrased  in the language of Young measures. The notion  of dissipation required for this purpose is quite delicate and relies on the theory developed in~\cite{DM-DeS-Mor-Mor-1}.  

We also prove that the barycentres of these Young measure solutions define a function $(\uu(t),\ee(t),\pp(t), \zz(t))$, where $(\uu(t),\ee(t),\pp(t))$ is a quasistatic evolution of a perfect plasticity problem (see \cite{DM-DeS-Mor}) corresponding to a relaxed dissipation function, denoted $p\mapsto\HH_{\rm eff}(p,0)$, which can be computed explicitly in terms of $\HH$ and $\V$. Some other qualitative properties of the solutions are investigated 
at the end of Section~\ref{sec:4}.

This result allows to compare the  globally stable solutions obtained in this paper with the solutions delivered by 
the vanishing viscosity approach of~\cite{DM-DeS-Mor-Mor-2}. In particular, we study in Section~\ref{exa} the globally stable evolution corresponding to the same data considered in~\cite[Section 7]{DM-DeS-Mor-Mor-2}. The main differences are the following. While the globally stable solution involves
generalized Young measures,  the  vanishing viscosity evolution takes place in spaces of affine functions, since the
data in the example are spatially homogeneous.  The stress 
$\sigmaa(t)$ corresponding to the vanishing viscosity solution exhibits a decreasing branch, which accounts for the softening phenomenon. On the contrary, the stress  of the globally stable solution is nondecreasing and, after a critical time, it  becomes constantly equal to the  asymptotic value of the stress of the viscosity solution.
\end{section}

\begin{section}{Notation and preliminary results}\label{notpre}

\subsection{Mathematical preliminaries}
We refer to \cite{DM-DeS-Mor-Mor-2} for the standard notation about measures, matrices and functions with bounded deformation. In particular, for every measure $\mu$
the symbols $\mu^a$ and $\mu^s$ always denote the absolutely continuous and the singular part with respect to Lebesgue measure. The former is always identified with its density. The symbol  $\|\cdot\|_2$ denotes norm in $L^2$, while  $\|\cdot\|_1$ denotes the norm in $L^1$, as well as in the space
$M_b$ of bounded Radon measures. The symbol $\langle\cdot,\cdot\rangle$ denotes a
duality pairing depending on the context.

\medskip
\noindent
{\bf Generalized Young measures.}
We refer to \cite{DM-DeS-Mor-Mor-1} for the definition and properties of generalized Young measures and of time dependent systems of generalized Young measures. 
The underlying measure $\lambda$ will always be the two-dimensional Lebesgue
measure $\Ln$.
In particular we refer to \cite[Section~6]{DM-DeS-Mor-Mor-1} for the definition of barycentre of a  generalized Young measure, and to  \cite[Section~3]{DM-DeS-Mor-Mor-1} for the notion of weak$^*$ convergence on the space 
$GY(\ol U;\Xi)$ of generalized Young measures on the closure of an open subset $U$
of $\Rn$ with values in a 
finite dimensional Hilbert space~$\Xi$.

Given $\nu\in M_b^+(\ol U)$, $p\in L^1_\nu(\ol U;\Xi)$, and
$\alpha\in L^\infty_\nu(\ol U)$,
let $\alpha\omega_p^\nu$ be the element of $M_*(\ol U{\times}\Xi{\times}\R)$
defined by
\begin{equation}\label{omega}
\langle f,\alpha\omega_p^\nu\rangle =
\int_{\ol U} \alpha(x)f(x,p(x),0)\, d\nu(x)
\end{equation}
for every $f\in C^{hom}(\ol U{\times}\Xi{\times}\R)$. Note that $\alpha\omega_p^\nu$ does not belong to $GY(\ol U;\Xi)$ since it does not satisfy the projection property (3.3) of \cite{DM-DeS-Mor-Mor-1}.

Given $p\in L^1( U;\Xi)$, let $\T_p\colon  U{\times}\Xi{\times}\R \to   U{\times}\Xi{\times}\R $ be the map defined by $\T_p(x,\xi,\eta):=(x,\xi+\eta p(x),\eta)$. 
The translation of $\mu\in GY(\ol U;\Xi)$ by $p$ is the image $\T_p(\mu)$ of $\mu$ under $\T_p$, that is,
\begin{equation}\label{Tp}
\langle f, \T_p(\mu)\rangle =\langle f(x,\xi+\eta p(x),\eta), \mu(x,\xi,\eta)\rangle
\end{equation}
for every $f\in C^{hom}(\ol U{\times}\Xi{\times}\R)$.

\begin{lemma}\label{trlem}
Let $\mu_k,\mu\in GY(\ol U;\Xi)$. Assume that $\mu_k\wto\mu$ 
weakly$^*$ in $GY(\ol U;\Xi)$. Then
\begin{equation}\label{lm:fm}
\T_p(\mu_k) \wto \T_p(\mu)
\quad \text{weakly}^* \text{ in } GY(\ol U;\Xi)
\end{equation}
for every $p\in L^1( U;\Xi)$.
\end{lemma}

\begin{proof}
For every $\e>0$ there exists $p_\e\in C^0_0( U;\Xi)$ such that $\|p_\e-p \|_1<\e$. 
By the definition of weak$^*$ convergence
$$
\langle f(x, \xi+\eta p_\e(x),\eta),\mu_k(x,\xi,\eta)\rangle \to
\langle f(x,\xi+\eta p_\e(x),\eta),\mu(x,\xi,\eta)\rangle\,.
$$
for every $f\in C^{hom}(\ol U{\times}\Xi{\times}\R)$.
By the projection property (3.3) of \cite{DM-DeS-Mor-Mor-1} we have
$$
\begin{array}{c}
|\langle f(x, \xi+\eta p_\e(x),\eta),\mu_k(x,\xi,\eta) \rangle
-\langle f(x, \xi+\eta p(x),\eta),\mu_k(x,\xi,\eta)\rangle |
\leq
\smallskip
\\
\displaystyle
\leq \langle  a \eta |p_\e(x)-p(x)|, \mu_k(x,\xi,\eta)\rangle
= \int_U a |p_\e(x)-p(x)|\, dx \leq a \e\,,
\end{array}
$$
whenever $|f(x,\xi_1,\eta_1)-f(x,\xi_2,\eta_2)|\leq a(|\xi_1-\xi_2|+|\eta_1-\eta_2|)$,
and the same inequality holds for $\mu$. 
Since $\e$ is arbitrary, under the same hypothesis on $f$ we obtain
$$
\langle f(x, \xi+\eta p(x),\eta),\mu_k(x,\xi,\eta)\rangle \to
\langle f(x,\xi+\eta p(x),\eta),\mu(x,\xi,\eta)\rangle\,.
$$
The conclusion follows from the density result proved in \cite[Lemma~2.4]{DM-DeS-Mor-Mor-1}.
\end{proof}

\subsection{Mechanical preliminaries}
We now introduce the mechanical notions used in the paper.

\medskip
\noindent
{\bf The reference configuration.}
Throughout the paper $\Om$ is a {\it bounded connected open set\/} in $\Rn$ with 
$C^2$ {\it boundary\/}. Let $\Ga_0$ be a nonempty relatively open subset of
$\partial\Om$ with a finite number of connected components, and let
$\Ga_1:={\partial\Om\setmeno\ol\Ga_0}$.

On $\Ga_0$ we will prescribe a Dirichlet boundary condition.
This will be done by assigning a function  $w\in H^{1/2}(\Ga_0;\Rn)$, or, equivalently, a function $w\in H^1(\Om;\Rn)$, whose trace on $\Ga_0$
(also denoted by $w$) is the prescribed boundary value. 
The set $\Ga_1$ will be the traction free part of the boundary. 

\medskip

\noindent {\bf Admissible stresses and dissipation.}
Let $K$ be a closed strictly convex set in $\MD{\times}\R$ with $C^1$ boundary. For every value of the internal variable $\zeta\in \R$,
the set 
\begin{equation}\label{Kzeta}
K(\zeta):=\{\sigma\in \MD:(\sigma,\zeta)\in K\}
\end{equation}
is interpreted as the {\it elastic domain\/} and its boundary as the {\it yield surface\/} corresponding to~$\zeta$. We assume that there exist two constants $A$ and $B$, with $0<A\le B<\infty$, such that
\begin{equation}\label{rk}
\{(\sigma,\zeta)\in\MD{\times}\R: |\sigma|^2+|\zeta|^2\le A^2\}\subset
 K\subset \{(\sigma,\zeta)\in\MD{\times}\R: |\sigma|^2+|\zeta|^2\le B^2\}\,.
\end{equation}
We assume in addition that 
\begin{eqnarray}
&\label{0zeta}
(\sigma,\zeta)\in K \quad \Longrightarrow \quad (0,\zeta)\in K\,,
\\
&\label{symK}
(\sigma,\zeta)\in K \quad \Longrightarrow \quad (\sigma,-\zeta)\in K
\,.
\end{eqnarray}
Together with convexity, \eqref{symK} yields
\begin{equation}\label{sigma0}
(\sigma,\zeta)\in K \quad \Longrightarrow \quad (\sigma,0)\in K
\quad \Longleftrightarrow \quad \sigma\in K(0)\,.
\end{equation}
Let $\pi_\R \colon \MD {\times} \R \to \R$ be the projection onto $\R$.
The hypotheses on $K$ imply that there exists a constant $a_K>0$  such that 
\begin{equation}\label{prK}
\pi_\R(K) = [-a_K,a_K]\,.
\end{equation}


The {\it support function\/} 
$H\colon\MD{\times}\R\to {[0,+\infty)}$ of $K$, defined by
\begin{equation}\label{HD}
H(\xi,\theta):=\sup_{(\sigma,\zeta)\in K} \{\sigma{\,:\,}\xi+\zeta\,\theta\} \,,
\end{equation}
will play the role of the {\it dissipation density\/}.
It turns out that $H$ is convex and positively homogeneous of degree one on
$\MD{\times}\R$. In particular it satisfies the triangle inequality
\begin{equation}\label{triangle0}
H(\xi_1+\xi_2,\theta_1+\theta_2)\le H(\xi_1,\theta_1)+H(\xi_2,\theta_2)\,.
\end{equation}

Let $\Phi$ be the gauge function of $K$ according to  \cite[Section~4]{Roc}. Since $\Phi^2$ is strictly convex and differentiable, and
$\frac12H^2=(\frac12\Phi^2)^*$, by \cite[Theorem~26.3]{Roc} the function $H^2$ is strictly convex and differentiable, so that the
 set $\{(\xi,\theta)\in \MD{\times}\R: H(\xi,\theta)\le 1\}$ is strictly convex with $C^1$ boundary. The same property holds for the sets
\begin{equation}\label{strconv}
\{(\xi,\theta)\in \MD{\times}\R: H(\xi,\theta)+c\theta\le 1\}
\end{equation}
for every $c\in \R$.

{}From (\ref{rk}) it follows that
\begin{equation}\label{boundsH}
A\sqrt{|\xi|^2+\theta^2}\le H(\xi,\theta)\le B\sqrt{|\xi|^2+\theta^2}\,,
\end{equation}
 from (\ref{0zeta}) and \eqref{prK} we obtain
\begin{equation}\label{H0zeta}
H(\xi,\theta)\geq H(0,\theta) = a_K |\theta|\,,
\end{equation}
while \eqref{symK} implies
\begin{equation}\label{symH}
H(\xi,\theta)=H(\xi,-\theta)\,.
\end{equation}

It follows from \eqref{sigma0} and \eqref{HD} that
\begin{equation}\label{HD0}
H(\xi,0)=\sup_{\sigma\in K(0)} \sigma{\,:\,}\xi \,,
\end{equation}
so that $H(\cdot,0)$ is the support function of $K(0)$ in $\MD$.

Using the theory of convex functions of measures developed in \cite{Gof-Ser},
we introduce the functional $\HH\colon M_b(\ol\Om;\MD){\times}M_b(\ol\Om)\to\R$
defined by
\begin{equation}\label{calHD}
\HH(p,z):=\int_{\ol\Om} 
H({\textstyle\frac{dp}{d\lambda}(x),\frac{dz}{d\lambda}}(x))\,d\lambda(x) \,,
\end{equation}
where $\lambda\in M_b^+(\ol\Om)$ is any measure such that 
$p<<\lambda$ and
$z<<\lambda$ (the homogeneity of $H$ implies that the integral does not depend
on~$\lambda$).
Using \cite[Theorem~4]{Gof-Ser} and  \cite[Chapter~II, Lemma~5.2]{Tem} we can see that $\HH(p,z)$ coincides with the integral over $\ol\Om$ of the measure studied in 
\cite[Chapter~II, Section~4]{Tem}, hence
$\HH$ is lower semicontinuous on $M_b(\ol\Om;\MD){\times}M_b(\ol\Om)$ with respect to weak$^*$ convergence of measures.
It follows from the properties of $H$ that $\HH$ satisfies the triangle inequality, i.e., 
\begin{equation}\label{triangle}
\HH(p_1+p_2,z_1+z_2)\le \HH(p_1,z_1)+\HH(p_2,z_2)
\end{equation}
for every $p_1,p_2\in M_b(\ol\Om;\MD)$ and every $z_1,z_2\in M_b(\ol\Om)$.



\medskip

\noindent {\bf The elasticity tensor.}
Let $\C$ be the {\it elasticity tensor\/}, considered as a symmetric 
positive definite linear operator $\C\colon\Mnn\to\Mnn$. 
We assume that the orthogonal subspaces $\MD$ and $\R I$ are invariant under~$\C$.
This is equivalent to saying that there exist a 
symmetric positive definite linear 
operator $\C_D\colon\MD\to\MD$ and a constant
$\kappa>0$, called {\it modulus of compression\/}, such that
\begin{equation}
\C\xi:= \C_D\xi_D+\kappa (\tr\,\xi) I 
\end{equation}
for every $\xi\in\Mnn$. Note that when $\C$ is isotropic, we have
\begin{equation}\label{iso}
\C\xi=2\mu \xi_D + \kappa (\tr\xi)I\,,
\end{equation}
where $\mu>0$ is the {\it shear modulus\/}, so that our assumptions are satisfied.  

Let $Q\colon\Mnn\to{[0,+\infty)}$ be the quadratic form associated with $\C$, defined by
\begin{equation}\label{W}
\textstyle Q(\xi):=\frac12 \C\xi{\,:\,}\xi=
\frac12\C_D\xi_D{\,:\,}\xi_D+\frac{\kappa}2(\tr\,\xi)^2\,.
\end{equation}
It turns out that there exist two constants $\alpha_\C$ and $\beta_\C$, 
with $0<\alpha_\C\le\beta_\C<+\infty$, such that
\begin{equation}\label{boundsC}
\alpha_\C|\xi|^2\le Q(\xi)\le 
\beta_\C|\xi|^2
\end{equation}
for every $\xi\in\Mnn$.
These inequalities imply 
\begin{equation}\label{normC}
|\C\xi|\le 2\beta_\C| \xi| \,.
\end{equation}


\medskip

\noindent {\bf The softening potential.}
Let $V\colon\R\to\R$ be a function of class $C^2$, which will control the evolution of the internal variable $\zeta$, and consequently of the set $K(\zeta)$ of admissible stresses. We assume that there exist two constants $b_V>0$ and $M_V> 0$ such that for every $\theta\in \R$ and $\tilde\theta\in\R\setminus\{0\}$
\begin{eqnarray}
&
\displaystyle
\vphantom{\lim_{\theta\to-\infty}}
 -M_V\le V''(\theta)\le 0\,, 
\label{2der}
\\
&
\displaystyle
 \lim_{\theta\to-\infty}V'(\theta)=-\lim_{\theta\to+\infty}V'(\theta)=b_V\,,\label{Vsymm20}
\\
&
\displaystyle
\vphantom{\lim_{\theta\to-\infty}}
 0<b_V <a_K\,,\label{V'infty}
\\
&
\displaystyle
\vphantom{\lim_{\theta\to-\infty}}
 V^\infty(\tilde\theta)<V(\theta+\tilde\theta)-V(\theta)\,,
\label{VinftyV}
\end{eqnarray}
where
$a_K$ is the constant in~(\ref{prK}), and 
$V^\infty$
denotes the
 {\it recession function\/} of $V$, defined by
\begin{equation}\label{Vinf}
V^\infty(\theta):=\lim_{t\to+\infty}\frac{V(t\theta)}{t}=-b_V |\theta|\,.
\end{equation}
Note that \eqref{Vsymm20} is satisfied when $V$ is even, while (\ref{VinftyV}) is satisfied when $V$ is strictly concave.
{}From \eqref{VinftyV} it follows that for every $R>0$ there exists a constant $c_R>0$
such that
\begin{equation}\label{defcR}
V'(\theta)\tilde\theta -V^\infty(\tilde\theta)\geq c_R|\tilde\theta|
\end{equation}
for every $\theta,\tilde \theta\in\R$ with $|\theta|\leq R$.

{}From \eqref{H0zeta} and \eqref{V'infty} it follows that
there exists a constant $  C^K_V>0$ such that
\begin{equation}\label{gammaM}
H(\xi_2-\xi_1,\theta_2-\theta_1)+V(\theta_2)-V(\theta_1)\ge  C^K_V |\xi_2-\xi_1| +  C^K_V |\theta_2-\theta_1|
\end{equation}
for every $\xi_1,\xi_2\in \MD$ and every $\theta_1,\theta_2\in\R$ (see 
\cite[Subsection~2.2]{DM-DeS-Mor-Mor-2}).

It is convenient to introduce the function $\V^\infty\colon M_b(\ol\Om)\to\R$ defined by
$$
\V^\infty(z):=\int_{\ol\Om}V^\infty(\tfrac{dz}{d\lambda})\,d\lambda\,,
$$
where $\lambda\in M_b^+(\ol\Om)$ is any measure such that $z<<\lambda$, and
the function $\V\colon L^1(\Om)\to\R$ defined by
$$
\V(z):=\int_\Om V(z(x))\,dx\,.
$$
The definition is extended to $M_b(\ol\Om)$ by setting 
$$
\V(z):=\V(z^a)+\V^\infty(z^s)
$$
for every $z\in M_b(\ol\Om)$.
\medskip

\noindent {\bf The prescribed boundary displacements.}
For every $t\in[0,+\infty)$ we prescribe a {\it boundary displacement\/} $\ww(t)$ in the
space $H^1(\Om;\Rn)$. This choice is motivated by the fact that we do not want to
impose ``discontinuous'' boundary data, so that, if the displacement develops sharp
discontinuities, this is due to energy minimization.

We assume also that $\ww\in AC_{loc}([0,+\infty); H^1(\Om;\Rn))$, which means, by definition, that for every $T>0$ the function $\ww$ belongs to the space 
$AC([0,T]; H^1(\Om;\Rn))$ of absolutely continuous functions on $[0,T]$ with values in
$H^1(\Om;\Rn)$, so that the time derivative $\dot \ww$ belongs to 
$L^1([0,T]; H^1(\Om;\Rn))$ and its strain $E\dot \ww$ belongs to 
$L^1([0,T];L^2(\Om;\Mnn))$. For the main properties of absolutely continuous functions with values in reflexive Banach spaces we refer to \cite[Appendix]{Bre}.

\medskip
\noindent
{\bf Elastic and plastic strains.}
Given a displacement $u\in BD(\Om)$ and a 
boundary datum $w\in H^1(\Om;\Rn)$, the {\it elastic strain\/}
$e\in L^2(\Om;\Mnn)$ and the {\it plastic strain\/} $p\in M_b(\ol\Om;\MD)$
satisfy the {\it weak kinematic admissibility conditions\/}
\begin{eqnarray}
& Eu=e+p \quad \hbox{in }\Om\,, \label{900}
\\
& p=(w-u){\,\odot\,}n\,\hn \quad \hbox{on } \Ga_0\,,  \label{901}
\end{eqnarray}
where $n$ is the outward unit normal, $\odot$ denotes the symmetrized 
tensor product, and $\hn$ is the one-dimensional Hausdorff measure. 
The condition on $\Ga_0$ shows, in particular, that the prescribed boundary condition $w$ is not attained on $\Ga_0$ whenever a plastic slip occurs at the boundary.
It follows from \eqref{900} and \eqref{901} that $e=E^au-p^a$ a.e.\ in $\Om$ and $p^s=E^su$ in $\Om$. Since $\tr\,p=0$, 
it follows from (\ref{900}) that $\div\,u=\tr\, e\in L^2(\Om)$
and from (\ref{901}) that $(w-u){\,\cdot\,}n=0$ $\hn$-a.e.\ on~$\Ga_0$, where the dot denotes the scalar product in $\R^2$. 

Given $w\in H^1(\Om;\Rn)$, the set $A(w)$ of {\it admissible displacements and strains\/} for the boundary datum $w$ on $\Ga_0$ is defined by
 \begin{equation}\label{Adef}
A(w):=\{ (u,e,p)\in BD(\Om){\times}
L^2(\Om;\Mnn){\times}M_b(\ol\Om;\MD): \text{(\ref{900}), (\ref{901})
 hold}\}\,.
\end{equation}
The set $A_{reg}(w)$ of {\it regular admissible displacements and strains\/}  is defined as
 \begin{equation}\label{Areg}
A_{reg}(w):=A(w)\cap \big(W^{1,1}_{loc}(\Om;\Rn){\times}
L^2(\Om;\Mnn){\times}L^1(\Om;\MD)\big)\,.
\end{equation}
Equivalently, $(u,e,p)\in A_{reg}(w)$ if and only if $u\in W^{1,1}_{loc}(\Om;\Rn)\cap 
BD(\Om)$, $e\in L^2(\Om;\Mnn)$, $p\in L^1(\Om;\MD)$, $Eu=e+p$ a.e.\ on $\Om$, 
and $u=w$ $\hn$-a.e.\ on~$\Gamma_0$.
\medskip

\noindent{\bf The stress.}
The {\it stress\/} $\sigma\in L^2(\Om;\Mnn)$ is given by
\begin{equation}\label{sigma}
\sigma:=\C e= \C_D e_D+ \kappa \,(\tr\, e)\,I\,,
\end{equation}
and the {\it stored elastic energy\/} by
\begin{equation}\label{elenergy}
\QQ(e)=\int_{\Om}Q(e(x))\, dx={\textstyle\frac12}\langle \sigma, e\rangle \,.
\end{equation}
It is well known that $\QQ$ is lower semicontinuous on 
$L^2(\Om;\Mnn)$ with respect to weak convergence.

If $\sigma\in L^2(\Om;\Mnn)$ and $\div\,\sigma\in L^2(\Om;\Rn)$, then the trace of the normal component of $\sigma$ on $\partial\Om$, denoted by $[\sigma n]$, is  defined as the distribution 
 on $\partial\Om$ such that
\begin{equation}\label{sigmanu}
\langle [\sigma n],\psi\rangle_{\partial\Om}:=\langle \div\, \sigma,\psi 
\rangle + \langle \sigma,E\psi\rangle \qquad 
\end{equation}
for every $\psi\in H^1(\Om;\Rn)$.
It turns out that $[\sigma n]\in H^{-1/2}(\partial\Om;\Rn)$ (see, e.g., 
\cite[Theorem~1.2, Chapter~I]{Tem}). We say that $[\sigma n]=0$ on $\Ga_1$ if 
$\langle [\sigma n],\psi\rangle_{\partial\Om}=0$ for every  $\psi\in H^1(\Om;\Rn)$ with $\psi=0$ $\hn$-a.e.\ on $\Ga_0$. 

\end{section}

\begin{section}{Relaxation of the incremental problems}\label{sec:6}

In this section we study different forms of relaxation of the incremental minimum problems.

\subsection{Convex envelope of the nonelastic part}
In this subsection, given $(\xi_0,\theta_0)\in\MD{\times}\R$, 
we study the convex envelope of the function
\begin{equation}\label{HV}
F(\xi,\theta):= H(\xi-\xi_0,\theta-\theta_0)+V(\theta)\,.
\end{equation}

Setting $\tilde\xi=\xi-\xi_0$ and $\tilde\theta=\theta-\theta_0$ and subtracting the constant $V(\theta_0)$, it is enough to study the convex envelope of 
\begin{equation}\label{defG}
G(\tilde\xi,\tilde\theta):=H(\tilde\xi,\tilde\theta)+V(\tilde\theta+\theta_0)-V(\theta_0)\,.
\end{equation}
Let $G^\infty$ be the 
recession function of $G$, defined by
$$
 G^\infty(\tilde\xi,\tilde\theta):=\lim_{t\to+\infty}\frac{ G(t\tilde\xi,t\tilde\theta)}{t}\,.
$$
By the homogeneity of $H$ it follows that
\begin{equation}\label{Ginfty}
G^\infty(\tilde\xi,\tilde\theta)=H(\tilde\xi,\tilde\theta)+V^\infty(\tilde\theta)\,.
\end{equation}

\begin{lemma}\label{coGcoGinfty} For every $(\tilde\xi,\tilde\theta)\in\MD{\times}\R$
we have 
\begin{equation}\label{coF3}
{\rm co}\, G(\tilde\xi,\tilde\theta)={\rm co}\, G^\infty(\tilde\xi,\tilde\theta)\,,
\end{equation}
where ${\rm co}$ denotes the convex envelope in $\MD{\times}\R$.
In particular, ${\rm co}\, G$ does not depend on $\theta_0$ and is positively homogeneous of degree $1$ in $(\tilde\xi,\tilde\theta)$.
\end{lemma}

\begin{proof} As $V$ is concave, we have 
$V(\tilde\theta+\theta_0)-V(\theta_0)\ge V^\infty(\tilde\theta)$, 
which gives
\begin{equation}\label{coF1}
G\ge G^\infty\qquad\hbox{and}\qquad
{\rm co}\, G\ge{\rm co}\, G^\infty\,.
\end{equation}
Since $G\ge 0$ by (\ref{gammaM}) and $G(0,0)=0$, we have ${\rm co}\, G(0,0)=0$, so that by convexity 
\begin{equation}\label{coF2}
{\rm co}\, G\le ({\rm co}\, G)^\infty\,,
\end{equation}
where $({\rm co}\, G)^\infty$ is the recession function of ${\rm co}\, G$, defined by
$$
({\rm co}\, G)^\infty(\tilde\xi,\tilde\theta):=\lim_{t\to+\infty}\frac{{\rm co}\, G(t\tilde\xi,t\tilde\theta)}{t}\,.
$$
On the other hand, since ${\rm co}\, G\le G$, we have $({\rm co}\, G)^\infty \le G^\infty$, which implies $({\rm co}\, G)^\infty \le {\rm co}\,G^\infty$. Therefore, (\ref{coF2}) gives
${\rm co}\, G\le {\rm co}\, G^\infty$, which, together with (\ref{coF1}), yields~\eqref{coF3}.
\end{proof}

Let us define $H_{\rm eff}:\MD{\times}\R\to\R$ by 
\begin{equation}\label{Heff}
H_{\rm eff}:={\rm co}\, G^\infty\,.
\end{equation}
By the previous lemma the convex envelope  ${\rm co}\, F$
of the function $F$ introduced in (\ref{HV}) is given by
\begin{equation}\label{coFHeff}
{\rm co}\, F(\xi,\theta)= H_{\rm eff}(\xi-\xi_0,\theta-\theta_0)+V(\theta_0)\,.
\end{equation} 

As $H_{\rm eff}$ is convex and positively homogeneous of degree $1$, it can be written in the form
\begin{equation}\label{supeff}
H_{\rm eff}(\xi,\theta)=\sup_{(\sigma,\zeta)\in K_{\rm eff}} \{ \sigma{\,:\,}\xi +\zeta\, \theta \}\,,
\end{equation}
where $K_{\rm eff}=\{(\sigma,\zeta)\in\MD{\times}\R : H_{\rm eff}^*(\sigma,\zeta)\le 0 \}$
(see, e.g., \cite[Theorem~13.2]{Roc}), and
$H_{\rm eff}^*=\chi_{K_{\rm eff}}$, where for every set $E\subset\MD{\times}\R$ the indicator function $\chi_E$ is defined by $\chi_E(\sigma,\zeta)=0$ if $(\sigma,\zeta)\in E$, $\chi_E(\sigma,\zeta)=+\infty$ otherwise.
Since $H_{\rm eff}={\rm co}\, G^\infty$, we have that $H_{\rm eff}^*=(G^\infty)^*$, so that
$$
K_{\rm eff}=\{(\sigma,\zeta)\in \MD{\times}\R: (G^\infty)^*(\sigma,\zeta)\le 0 \}\,.
$$

Since $V^\infty(\theta)=\min\{ b_V \theta, -b_V \theta\}$ 
by (\ref{Vinf}),  the function $G^\infty$ can be
expressed as the minimum of two convex functions, namely
\begin{equation}\label{co11}
G^\infty(\xi,\theta)= \min\{ H(\xi,\theta)+b_V\theta, H(\xi,\theta)-b_V\theta\}\,.
\end{equation}
Therefore
$$
(G^\infty)^*(\sigma,\zeta)= \max\{ H^*(\sigma,\zeta-b_V), H^*(\sigma,\zeta+b_V)\}\,.
$$
Since $H^*=\chi_K$, we obtain
$$
\chi_{K_{\rm eff}}= (G^\infty)^*= \max\{ \chi_{K+(0,b_V)}, \chi_{K-(0,b_V)}\}\,,
$$
which implies
\begin{equation}\label{Keff}
K_{\rm eff}= (K+(0,b_V))\cap (K-(0,b_V))\,.
\end{equation}
Using (\ref{V'infty}), \eqref{Keff}, and the strict convexity of $K$, 
it is easy to check that $K_{\rm eff}$ is a bounded closed convex set and that
\begin{equation}\label{kefint}
(0,0)\in\Interior K_{\rm eff}\subset K_{\rm eff}\subset \Interior K\,.
\end{equation}

\begin{lemma}\label{HeffG}
For every $(\xi,\theta)\in \MD{\times}\R$ we have
\begin{eqnarray}
\label{He=G}
H_{\rm eff}(\xi,\theta)=G(\xi,\theta) &
\Longleftrightarrow & (\xi,\theta)=(0,0)\,,
\\
\label{He=Ginf}
H_{\rm eff}(\xi,0)=G^\infty(\xi,0) &
\Longleftrightarrow & \xi=0\,.
\end{eqnarray}
\end{lemma}

\begin{proof}
By \eqref{HD}, \eqref{supeff}, and \eqref{kefint} we have
\begin{equation}\label{He<H}
H_{\rm eff}(\xi,\theta)< H(\xi,\theta) \qquad \text{for every }
(\xi,\theta)\neq (0,0)\,.
\end{equation}
If $H_{\rm eff}(\xi,\theta)=G(\xi,\theta)$, 
by \eqref{coF1} and \eqref{Heff} we have 
$G(\xi,\theta)=G^\infty(\xi,\theta)$, which gives $V^\infty(\theta)=
V(\theta+\theta_0)-V(\theta_0)$ thanks to \eqref{defG} and \eqref{Ginfty}.
By \eqref{VinftyV} this implies $\theta=0$,
so that $H_{\rm eff}(\xi,0)=G(\xi,0)=H(\xi,0)$. By \eqref{He<H} we deduce
that $\xi=0$.
This concludes the proof of \eqref{He=G}.

On the other hand, if $H_{\rm eff}(\xi,0)=G^\infty(\xi,0)$, from
\eqref{Ginfty} 
we obtain $H_{\rm eff}(\xi,0)=H(\xi,0)$. By \eqref{He<H} we deduce $\xi=0$,
which concludes the proof of \eqref{He=Ginf}.
\end{proof}

\begin{lemma}\label{cothetaG}
For every $(\xi,\theta)\in \MD{\times}\R$ we have
\begin{equation}\label{cof}
{\rm co}\, G^\infty(\xi,\theta)={\rm co}_\theta G^\infty(\xi,\theta)\,,
\end{equation}
where ${\rm co}_\theta$ denotes the convex envelope with respect to~$\theta$.
Moreover there exist $\hat\theta\in\R$ 
and $\alpha\in[\frac12,1]$, such that
\begin{eqnarray}
&\theta=\alpha\hat\theta+(1-\alpha)(-\hat\theta)\,,\label{f1}
\\
&H_{\rm eff}(\xi,\theta)=G^\infty(\xi,\hat\theta)\,.
\label{f2}
\end{eqnarray}
\end{lemma}

\begin{proof}
Let 
$$
\begin{array}{c}
A^\ominus:=\{ (\xi,\theta)\in\MD{\times}\R: \ \theta\le 0, \ H(\xi,\theta)+b_V\theta\le 1\}\,,
\smallskip
\\
A^\oplus:=\{ (\xi,\theta)\in\MD{\times}\R: \ \theta\ge 0, \ H(\xi,\theta)-b_V\theta\le 1\}\,,
\end{array}
$$ 
and $A:=A^\ominus\cup A^\oplus=\{ G^\infty\le 1\}$ (see \eqref{co11}).
By (\ref{symH}) we have
\begin{equation}\label{Apm}
(\xi,\theta)\in A^\ominus \quad \Longleftrightarrow \quad (\xi,-\theta)\in A^\oplus\,.
\end{equation}
Since $G^\infty$ is positively homogeneous of degree $1$, we have ${\rm co}\,
A=\{ {\rm co}\, G^\infty\le 1 \}$ and ${\rm co}_\theta A\subset
\{ {\rm co}_\theta G^\infty \le 1 \}$, where ${\rm co}_\theta A$ is the smallest set
containing $A$, which is convex with respect to $\theta$, i.e., its intersections with all
lines $\{\xi=const.\}$ are convex. To prove that ${\rm co}\, A={\rm co}_\theta A$, it is
enough to show that ${\rm co}_\theta A$ is convex.
By (\ref{Apm}) we have that $(\xi,\theta)\in {\rm co}_\theta A$ if and only if there exists 
$\theta^\oplus\in\R$ such that $|\theta|\le\theta^\oplus$ and $(\xi,\theta^\oplus)\in A^\oplus$. 
Since  $A^\oplus$ is convex,
from this property it is easy to deduce that ${\rm co}_\theta A$ is convex, hence 
${\rm co}\, A={\rm co}_\theta A$. It follows that
$$
{\rm co}_\theta A\subset \{ {\rm co}_\theta G^\infty \le 1 \}\subset 
\{ {\rm co}\, G^\infty\le 1 \}= {\rm co}\, A\,.
$$
This implies that $\{ {\rm co}\, G^\infty\le 1 \}=\{ {\rm co}_\theta G^\infty \le 1 \}$. 
Since both functions ${\rm co}\, G^\infty$ and ${\rm co}_\theta G^\infty$ are positively homogeneous of degree $1$, we conclude that 
${\rm co}\, G^\infty={\rm co}_\theta G^\infty$.

By homogeneity, to prove \eqref{f1} and \eqref{f2} it is not restrictive to assume that
$H_{\rm eff}(\xi,\theta)=1$, so that $(\xi,\theta)\in {\rm co}\, A$. {}From the previous discussion it follows that  there exists 
$\theta^\oplus\in\R$ such that $|\theta|\le\theta^\oplus$ and 
$(\xi,\theta^\oplus)\in A^\oplus$. In particular we
have $H_{\rm eff}(\xi,-\theta^\oplus)\le G^\infty(\xi,-\theta^\oplus)\le 1$ and 
$H_{\rm eff}(\xi,\theta^\oplus)\le G^\infty(\xi,\theta^\oplus)\le 1$. By convexity we have
$H_{\rm eff}(\xi,-\theta^\oplus)=H_{\rm eff}(\xi,\theta^\oplus)=1$, which implies
$G^\infty(\xi,-\theta^\oplus)= G^\infty(\xi,\theta^\oplus)=1$.  To conclude the proof of 
\eqref{f1} and \eqref{f2} it is enough to take
$\hat\theta=\theta^\oplus$ if $\theta\ge0$, and $\hat\theta=-\theta^\oplus$ if 
$\theta<0$.
\end{proof}

\begin{lemma}\label{lemma00}
Let $(\xi_0,\theta_0)\in\MD{\times}\R$ with $(\xi_0,\theta_0)\neq(0,0)$.
Assume that $H_{\rm eff}(\xi_0,\theta_0)=G^\infty(\xi_0,\theta_0)$. 
Then $\theta_0\neq 0$ and the common tangent hyperplane to the graphs of 
$H_{\rm eff}$ and $G^\infty$ at the point $(\xi_0,\theta_0,G^\infty(\xi_0,\theta_0))$ 
is the graph of the linear function
$$
L(\xi,\theta):=\partial_\xi G^\infty(\xi_0,\theta_0)\xi +\partial_\theta 
G^\infty(\xi_0,\theta_0)\theta\,.
$$
Let $C:=\{(\xi,\theta)\in\MD{\times}\R: L(\xi,\theta)=G^\infty(\xi,\theta)\}$.
Then either $C=\{ (\lambda\xi_0,\lambda\theta_0): \lambda\geq0\}$
or $C=\{ (\lambda\xi_0,\lambda\theta_0): \lambda\geq0\}\cup \{ (\lambda\xi_0,-\lambda\theta_0): \lambda\geq0\}$.
\end{lemma}

\begin{proof}
The inequality $\theta_0\neq0$ follows from \eqref{He=Ginf}. Therefore
$G^\infty$ is differentiable at $(\xi_0,\theta_0)$. Using the convexity of $H_{\rm eff}$
and the inequality $H_{\rm eff}\leq G^\infty$, we deduce that $H_{\rm eff}$ 
is differentiable at $(\xi_0,\theta_0)$ and its partial derivatives coincide with those of
$G^\infty$. The formula for the tangent hyperplane follows easily
from the Euler identity.

By \eqref{symH} and \eqref{Vinf} we may suppose $\theta_0>0$.
By the homogeneity of the problem it is not restrictive to assume that
$H_{\rm eff}(\xi_0,\theta_0)=G^\infty(\xi_0,\theta_0)=1$. 
Then the set $\{L=1\}$ is the common tangent hyperplane to the hypersurfaces 
$\{H_{\rm eff}=1\}$ and $\{G^\infty=1\}$ at the point $(\xi_0,\theta_0)$.
As $G^\infty(\xi,\theta)=H(\xi,\theta)-b_V\theta$ for $\theta\geq0$ and
the set $\{(\xi,\theta)\in\MD{\times}\R: H(\xi,\theta)-b_V\theta\leq1\}$ is
strictly convex by \eqref{strconv}, we deduce that $\{L=1\}\cap\{G^\infty=1\}\cap\{\theta\geq0\}=
\{(\xi_0,\theta_0)\}$. If the set $\{L=1\}\cap\{G^\infty=1\}\cap\{\theta<0\}$ is empty, then
$C=\{ (\lambda\xi_0,\lambda\theta_0): \lambda\geq0\}$ by homogeneity.

Suppose $\{L=1\}\cap\{G^\infty=1\}\cap\{\theta<0\}\neq\emptyset$.
Since $L\leq H_{\rm eff}$ by convexity,
if $(\xi_1,\theta_1)\in\{L=1\}\cap\{G^\infty=1\}\cap\{\theta<0\}$ we have 
$1=L(\xi_1,\theta_1)\leq H_{\rm eff}(\xi_1,\theta_1)\leq G^\infty(\xi_1,\theta_1)=1$.
Therefore, the same argument used for $(\xi_0,\theta_0)$ shows that
\begin{equation}\label{eq89}
\{L=1\}\cap\{G^\infty=1\}\cap\{\theta<0\}=\{(\xi_1,\theta_1)\}\,. 
\end{equation}
This implies
\begin{equation}\label{eq88}
\{L=1\}\cap\{G^\infty=1\}=\{(\xi_0,\theta_0),(\xi_1,\theta_1)\}\,.
\end{equation}

Let us prove that $\xi_1=\xi_0$ and $\theta_1=-\theta_0$.
Let $S$ be the open segment with endpoints $(\xi_0,\theta_0)$ and $(\xi_1,\theta_1)$.
As $L=1$ on the endpoints, it is $L=1$ on $S$. As $H_{\rm eff}=1$ on the endpoints,
by convexity we have $H_{\rm eff}\leq1$ on $S$. On the other hand, since
the graph of $L$ is tangent to the graph of $H_{\rm eff}$, by convexity we also have
$L\leq H_{\rm eff}$. Therefore $H_{\rm eff}=1$ on $S$.
By \eqref{eq88} we have $G^\infty\neq 1$ on $S$. As $H_{\rm eff}\leq G^\infty$,
we conclude that $H_{\rm eff}=1<G^\infty$ on~$S$.

Let us fix $(\xi,\theta)\in S$. Then $(\xi,\theta)\in\{H_{\rm eff}\leq1\}={\rm co}\,
\{G^\infty\leq1\}$. As $G^\infty(\xi,\theta)>1$,
by the previous lemma there exist $\theta^\ominus$ and 
$\theta^\oplus$ with $\theta^\ominus\leq0\leq\theta^\oplus$, such that
$\theta^\ominus<\theta<\theta^\oplus$, $(\xi,\theta^\ominus)\in \{G^\infty\leq1\}$,
and $(\xi,\theta^\oplus)\in \{G^\infty\leq1\}$. As $L\leq H_{\rm eff}$ by convexity and
$H_{\rm eff}\leq G^\infty$ by definition, we have $L(\xi,\theta^\ominus)\leq1$
and $L(\xi,\theta^\oplus)\leq1$. Since $\theta^\ominus<\theta<\theta^\oplus$ and
$L(\xi,\theta)=1$, we deduce from the linearity of $L$ that $L(\xi,\theta^\ominus)=
L(\xi,\theta^\oplus)=1$. Using again the inequality $L\leq G^\infty$, we find 
$G^\infty(\xi,\theta^\ominus)\geq1$ and $G^\infty(\xi,\theta^\oplus)\geq1$.
Since the opposite inequality follows from the definition of $\theta^\ominus$ and
$\theta^\oplus$, we also obtain $G^\infty(\xi,\theta^\ominus)=
G^\infty(\xi,\theta^\oplus)=1$. Therefore, \eqref{eq88} yields $\xi=\xi_0=\xi_1$,
$\theta^\ominus=\theta_1$, and  $\theta^\oplus=\theta_0$.
This implies that the straight line $\{(\xi_0,\theta):\theta\in\R\}$ belongs to the
hyperplane $\{L=1\}$. Since by \eqref{symH} and \eqref{Vinf} the point 
$(\xi_0,-\theta_0)$ belongs to $\{G^\infty=1\}$, we deduce that $\theta_1=-\theta_0$
by \eqref{eq89}. This concludes the proof of the equality
$\{L=1\}\cap\{G^\infty=1\}=\{(\xi_0,\theta_0),(\xi_0,-\theta_0)\}$, which implies that
$C=\{ (\lambda\xi_0,\lambda\theta_0): \lambda\geq0\}\cup \{ (\lambda\xi_0,-\lambda\theta_0): \lambda\geq0\}$ by homogeneity.
\end{proof}

\begin{lemma}\label{lemma01}
Let $(\xi_0,\theta_0)\in\MD{\times}\R$ with $(\xi_0,\theta_0)\neq(0,0)$.
Assume that there exist $\ol\theta\geq|\theta_0|$ such that
$H_{\rm eff}(\xi_0,\theta_0)=G^\infty(\xi_0,\ol\theta)$. 
Then $\ol\theta>0$.
Let $L\colon\MD{\times}\R\to\R$ be a linear function such that $L\leq H_{\rm eff}$
and $L(\xi_0,\theta_0)=H_{\rm eff}(\xi_0,\theta_0)$, and
let $C:=\{(\xi,\theta)\in\MD{\times}\R: L(\xi,\theta)=G^\infty(\xi,\theta)\}$.
Then 
$C\subset\{ (\lambda\xi_0,\lambda\ol\theta): \lambda\geq0\}\cup \{ (\lambda\xi_0,-\lambda\ol\theta): \lambda\geq0\}$.
\end{lemma}

\begin{proof}
If $\ol\theta=|\theta_0|$, the result follows from the previous lemma.
If $\ol\theta>|\theta_0|$, 
the affine function $\theta\mapsto L(\xi_0,\theta)$ is bounded from above by 
$G^\infty(\xi_0,\ol\theta)$ at the endpoints of the interval $[-\ol\theta,\ol\theta]$ 
(recall that $L\leq H_{\rm eff}\leq G^\infty$) 
and coincides with $G^\infty(\xi_0,\ol\theta)$ at the interior
point $\theta_0$.
Therefore, $L(\xi_0,\theta)=H_{\rm eff}(\xi_0,\theta)=G^\infty(\xi_0,\ol\theta)$ 
for every $\theta\in[-\ol\theta,\ol\theta]$.
The result follows by the previous lemma with $\theta_0$ replaced by $\ol\theta$.
\end{proof}

\subsection{Relaxation with respect to weak convergence}
We begin with a result that can be easily deduced from \cite{Anz-Gia}: every 
$(u,e,p)$ of the admissible set $A(w)$ introduced in \eqref{Adef} can be
approximated by triples $(u_k,e_k,p_k)$ in the set $A_{reg}(w)$ introduced in
\eqref{Areg}, so that $u_k$ satisfies the
boundary condition $u_k=w$ $\hn$-a.e.\ on $\Ga_0$.

\begin{theorem}\label{thm:d-app}
Let $w\in H^1(\Om;\Rn)$ and let $(u,e,p)\in A(w)$. Then there exists a sequence
$(u_k,e_k,p_k)\in A_{reg}(w)$ such that $u_k\wto u$ weakly$^*$ in $BD(\Om)$, $e_k\to e$ strongly in $L^2(\Om;\Mnn)$, $p_k\wto p$ weakly$^*$ in $M_b(\ol\Om;\MD)$,  $\|p_k\|_1\to \|p\|_1$, and $\|p_k-p^a\|_1\to\|p^s\|_1$.
\end{theorem}

\begin{proof}
By \cite[Theorem~5.2]{Anz-Gia} for every $k$ there exists a function $\psi_k\in W^{1,1}(\Om;\Rn)$ such that $\|\psi_k\|_1\le\frac1k$, $\psi_k=w-u$ $\hn$-a.e.\ on $\Ga_0$, $\|\div\, \psi_k\|_2\le\frac1k$, and 
$$
\textstyle
\|E\psi_k\|_1 \le \frac{1}{\sqrt2}\|w-u\|_{1,\Ga_0} +\frac1k
=\|p^s\|_{1,\Ga_0}+\frac1k\,,
$$
where $\|\cdot \|_{1,\Ga_0}$ denotes the norms in $L^1_{\hn}(\Ga_0;\Rn)$
and in $M_b(\Ga_0;\MD)$.
We define $v_k:=u+\psi_k$ and we note that $v_k=w$ $\hn$-a.e.\ on $\Ga_0$.
By \cite[Theorem~5.1]{Anz-Gia} there exists a sequence $v_k^m$ in $BD(\Om)\cap W^{1,1}_{loc}(\Om;\Rn)$, with $v_k^m=v_k=w$ $\hn$-a.e.\ on $\Ga_0$,
such that $v_k^m\to v_k$ strongly in $L^1(\Om;\Rn)$, $\div\, v_k^m\to \div\, v_k$ strongly in $L^2(\Om)$, $Ev_k^m\wto Ev_k$ weakly$^*$ in $M_b(\ol\Om;\Mnn)$, and
$$
\lim_{m\to\infty}\|Ev_k^m-E^a u -E\psi_k\|_1 =
\lim_{m\to\infty}\|Ev_k^m-E^a v_k\|_1= \|E^sv_k\|_{1,\Om}=\|p^s\|_{1,\Om}\,,
$$
where $\|\cdot \|_{1,\Om}$ denotes the norm in $M_b(\Om;\MD)$.
By approximation it is clear that we can find a sequence $m_k\to\infty$ such that, setting $u_k:=v_k^{m_k}$, we have $u_k\in BD(\Om)\cap W^{1,1}_{loc}(\Om;\Rn)$, $u_k=w$ $\hn$-a.e.\ on $\Ga_0$, $u_k\wto u$ weakly$^*$ in $BD(\Om)$, $\div\, u_k\to \div\, u$ strongly in $L^2(\Om)$, and
\begin{equation}\label{ganz}
\limsup_{k\to\infty}\|Eu_k-E^a u\|_1 \le \|p^s\|_1\,.
\end{equation}
Setting $e_k:=e_D+\frac12\,\div\,u_k\, I$ and $p_k:=Eu_k-e_k$, we clearly have that $e_k\to e$ strongly in $L^2(\Om;\Mnn)$ and $p_k\wto p$ weakly$^*$ in $M_b(\ol\Om;\MD)$. Since $Eu_k-E^a u=\frac12\, (\div\, u_k-\div\, u)\, I +p_k-p^a$, from (\ref{ganz}) it follows that
$$
\limsup_{k\to\infty}\|p_k-p^a\|_1 \le \|p^s\|_1\,.
$$
By lower semicontinuity this implies that $\|p_k-p^a\|_1\to \|p^s\|_1$ and $\|p_k\|_1\to \|p\|_1$.
\end{proof}

To deal with the inner variable $z$ we need a technical lemma concerning the
approximation of measures on product spaces.

\begin{lemma}\label{lm:comp}
Let $\Xi_1$ and $\Xi_2$ be finite dimensional Hilbert spaces. Let $p_i\in M_b(\ol\Om;\Xi_i)$ for $i=1,2$ and let $p_1^k$ be a sequence in $L^1(\Om;\Xi_1)$ such that $p_1^k\wto p_1$ weakly$^*$ in $M_b(\ol\Om;\Xi_1)$ and $\|p_1^k\|_1\to \|p_1\|_1$. Then, there exists  a sequence $p_2^k$ in $L^1(\Om;\Xi_2)$ such that $p_2^k\wto p_2$ weakly$^*$ in $M_b(\ol\Om;\Xi_2)$ and $\|(p_1^k,p_2^k)\|_1\to  \|(p_1,p_2)\|_1$, where
the norms are computed in the product Hilbert structure of ${\Xi_1{\times}\Xi_2}$.
\end{lemma}

\begin{proof}
First of all we observe that $|p_1^k|\wto |p_1|$ weakly$^*$ in $M_b(\ol\Om)$.
We decompose $p_2$ as
$$
p_2=p_{21}+p_{22}\,,
$$ 
with $p_{21},p_{22}\in M_b(\ol\Om;\Xi_2)$, $|p_{21}|<<|p_1|$, and 
$|p_{22}|\perp |p_1|$.

Let us construct a sequence $p^k_{21}$ in 
$L^1(\Om;\Xi_2)$ such that 
$p^k_{21}\wto p_{21}$ weakly$^*$ in $M_b(\ol\Om;\Xi_2)$ and 
$\|(p_1^k,p_{21}^k)\|_1\to  \|(p_1,p_{21})\|_1$.
As $|p_{21}|<<|p_1|$, we have $p_{21}=\psi |p_1|$ for a suitable density 
$\psi\in L^1_{|p_1|}(\ol\Om;\Xi_2)$. 
Let $\psi_m$ be a sequence in $C(\ol\Om;\Xi_2)$ which converges to $\psi$ in 
$L^1_{|p_1|}(\ol\Om;\Xi_2)$, so that
$\langle\sqrt{1+|\psi_m|^2}, |p_1|\rangle\to \langle\sqrt{1+|\psi|^2}, |p_1|\rangle
=\|(p_1,p_{21})\|_1$, as $m\to\infty$.
For every $m$ let $p_{21}^{km}:=\psi_m |p_1^k|$, so that $p_{21}^{km}\wto \psi_m |p_1|$ weakly$^*$ in $M_b(\ol\Om;\Xi_2)$ and
$\|(p_1^k,p_{21}^{km})\|_1 \to \langle\sqrt{1+|\psi_m|^2}, |p_1|\rangle$, as $k\to\infty$. 
Let ${\mathcal B}_R:=\{p\in M_b(\ol\Om;\Xi_2): \|p\|_1\le R\}$, with  $R>\|p_2\|_1$.
Since $\psi_m$ converges to $\psi$ in $L^1_{|p_1|}(\ol\Om;\Xi_2)$ we have
$\psi_m|p_1|\in {\mathcal B}_R$ for $m$ large enough.
As  $|p_1^k|\wto |p_1|$ weakly$^*$ in $M_b(\ol\Om)$, for these values of $m$ we 
also have
$p_{21}^{km}\in {\mathcal B}_R$ for $k$ large enough.
Since the weak$^*$ convergence is metrizable on ${\mathcal B}_R$, we can construct a sequence $m_k\to\infty$ such that $p^k_{21}:=p_{21}^{km_k}$ satisfies the required properties. 

Using convolutions it is easy to construct a sequence $p_{22}^k$ in $L^1(\Om;\Xi_2)$ such that $p_{22}^k\wto p_{22}$ weakly$^*$ in $M_b(\ol\Om;\Xi_2)$ and $\|p_{22}^k\|_1 \to \|p_{22}\|_1$.

Let $p_2^k:=p_{21}^k+p_{22}^k$. Then $p_2^k\wto p_2$ weakly$^*$ in $M_b(\ol\Om;\Xi_2)$. It remains to prove that 
\begin{equation}\label{Nconv}
\limsup_{k\to\infty} \|(p_1^k,p_2^k)\|_1\le \|(p_1,p_2)\|_1\,.
\end{equation}
By the triangle inequality and by the properties of $p_{21}^k$ and $p_{22}^k$, we have
\begin{eqnarray*}
& \displaystyle
\limsup_{k\to\infty}\|(p_1^k,p_2^k)\|_1\le
\lim_{k\to\infty} \|(p_1^k,p_{21}^k)\|_1 + \lim_{k\to\infty} \|(0,p_{22}^k)\|_1
=
\smallskip
\\
& = \|(p_1,p_{21})\|_1  + \|(0,p_{22})\|_1
= \|(p_1,p_2)\|_1\,,
\end{eqnarray*}
where the last equality follows from the fact that the measures $(p_1,p_{21})$ and 
$(0,p_{22})$ are mutually singular.
\end{proof}

Let $\HH_{\rm eff}: M_b(\ol\Om;\MD){\times} M_b(\ol\Om)\to \R$ be the functional defined by (\ref{calHD}) with $H$ replaced by $H_{\rm eff}$.

\begin{theorem}\label{thm:1app}
Let $e_0\in L^2(\Om;\Mnn)$, let $z_0\in M_b(\ol\Om)$, let $w\in H^1(\Om;\Rn)$, 
let $(u,e,p)\in A(w)$, and let $z\in M_b(\ol\Om)$. Then for every
$e_k\wto e$ weakly in $L^2(\Om;\Mnn)$, $p_k\wto p$ weakly$^*$ in 
$M_b(\ol\Om;\MD)$, $z_k\wto z$ weakly$^*$ in $M_b(\ol\Om)$,
we have
\begin{equation}\label{0app}
\QQ(e_0+e)+\HH_{\rm eff}(p, z)+\V(z_0)\le \liminf_{k\to\infty}
\{\QQ(e_0+e_k)+\HH(p_k,z_k) +\V(z_0+z_k)\}
\,.
\end{equation}
Moreover, there exist a sequence $(u_k,e_k,p_k)\in A_{reg}(w)$ and a sequence 
$z_k\in L^1(\Om)$ such that 
$u_k\wto u$ weakly$^*$ in $BD(\Om)$, $e_k\to e$ strongly in $L^2(\Om;\Mnn)$, $p_k\wto p$ weakly$^*$ in $M_b(\ol\Om;\MD)$, 
$z_k\wto z$ weakly$^*$ in $M_b(\ol\Om)$, and
\begin{equation}\label{1app}
\QQ(e_0+e)+\HH_{\rm eff}(p, z)+\V(z_0)\geq\limsup_{k\to\infty}
\{\QQ(e_0+e_k)+\HH(p_k,z_k) +\V(z_0+z_k)\}
\,.
\end{equation}
\end{theorem}

\begin{proof}
Owing to  the lower semicontinuity of $\QQ$ and 
$\HH_{\rm eff}$ (see the comments after \eqref{calHD} and \eqref{elenergy}), 
inequality \eqref{0app} follows from the inequality
$\HH_{\rm eff}(p_k, z_k)\le \HH(p_k,z_k) +\V(z_0+z_k)-\V(z_0)$, which is
a consequence of \eqref{coF3} and~\eqref{Heff}.

We observe that it is enough to prove \eqref{1app} when
$z_0$ belongs to $L^1(\Om)$ and is piecewise constant on a suitable
triangulation. Indeed, there exists a sequence $z_0^n$ of piecewise constant functions
which converge to $z_0^a$ strongly in $L^1(\Om)$. 
For every $n$ let $(u_k^n, e_k^n,p_k^n,z_k^n)$ be a sequence satisfying the second statement of the theorem as $k\to\infty$, with $z_0$ replaced by $z_0^n$. Then
\begin{equation}\label{1appV}
\QQ(e_0+e)+\HH_{\rm eff}(p, z)=\lim_{k\to\infty}
\{\QQ(e_0+e_k^n)+\HH(p_k^n,z_k^n) +\V(z_0^n+z_k^n)-\V(z_0^n)\}
\,.
\end{equation}
By \eqref{Vsymm20} and by the definition of $\V$ we have
$$
\V(z_k^n+z_0^n)-\V(z_0^n)-(\V(z_k^n+z_0)-\V(z_0))\leq
2b_V \|z_0^n-z_0^a\|_1\,.
$$
Therefore, for every $n$
\begin{equation}\label{1appleq}
\begin{array}{c}
\displaystyle
\limsup_{k\to\infty}
\{\QQ(e_0+e_k^n)+\HH(p_k^n,z_k^n) +\V(z_0+z_k^n)-\V(z_0)\}
\leq
\smallskip\\
\leq
\QQ(e_0+e)+\HH_{\rm eff}(p, z)+2b_V \|z_0^n-z_0^a\|_1
\,.
\end{array}
\end{equation}
By a standard double limit procedure it is then easy to construct a sequence $(u_k, e_k,p_k,z_k)$ satisfying the second statement of the theorem. 

Moreover, we may also assume that $(u,e,p)\in A_{reg}(w)$ and $z\in L^1(\Om)$.
Indeed, in the general case, combining Theorem~\ref{thm:d-app} with Lemma~\ref{lm:comp} we can construct a sequence $(u_m,e_m,p_m)\in A_{reg}(w)$ and a sequence $z_m\in L^1(\Om)$ such that $u_m\wto u$ weakly$^*$ in $BD(\Om)$, $e_m\to e$ strongly in $L^2(\Om;\Mnn)$, $p_m\wto p$ weakly$^*$ in $M_b(\ol\Om;\MD)$, $z_m\wto z$ weakly$^*$ in $M_b(\ol\Om)$, and $\|(p_m,z_m)\|_1\to \|(p,z)\|_1$. 
By \cite[Theorem~3]{Res} (see also \cite[Appendix]{Luc-Mod}) 
these properties imply that
$$
\QQ(e_0+e_m)+\HH_{\rm eff}(p_m,z_m) \quad \longrightarrow \quad
\QQ(e_0+e)+\HH_{\rm eff}(p, z)
$$
and the conclusion of the theorem can be obtained by a standard double limit procedure.

Let us fix a piecewise constant function $z_0\in L^1(\Om)$. Let
$$
G_0(x,\xi,\theta):=H(\xi,\theta)+V(\theta+z_0(x))-V(z_0(x))\,,
$$
let
$$
G_1(x,\xi,\theta):=\inf_{ (\lambda,\xi_1,\xi_2,\theta_1,\theta_2)\in
\Lambda}\{ \lambda G_0(x,\xi+\xi_1,\theta+\theta_1) +
(1-\lambda)G_0(x,\xi+\xi_2,\theta+\theta_2)\}\,,
$$
and let 
$$
G_2(x,\xi,\theta):=\inf_{ (\lambda,\xi_1,\xi_2,\theta_1,\theta_2)\in
\Lambda}\{ \lambda G_1(x,\xi+\xi_1,\theta+\theta_1) +
(1-\lambda)G_1(x,\xi+\xi_2,\theta+\theta_2)\}\,,
$$
where $\Lambda$ is the set of vectors $(\lambda,\xi_1,\xi_2,\theta_1,\theta_2)$
with $0\le\lambda\le1$, $\xi_1,\xi_2\in\MD$, $\theta_1,\theta_2\in\R$, 
$\lambda\xi_1+(1-\lambda)\xi_2=0$, and $\lambda\theta_1+(1-\lambda)\theta_2=0$.
As $G_0$ is globally Lipschitz continuous in $(\xi,\theta)$, uniformly with
respect to $x$, it follows that $G_1$ and $G_2$ satisfy the same property.
Moreover, $G_1$ and $G_2$ are piecewise constant in $x$, uniformly with respect to
$(\xi,\theta)$. It is easy to see that
$$
{\rm co}_{(\xi,\theta)} G_0(x,\xi,\theta)\leq G_2(x,\xi,\theta)\leq\sum_{i=1}^4\lambda_iG_0(x,\xi_i,\theta_i)
$$
whenever $(\xi,\theta)=\sum_{i=1}^4\lambda_i(\xi_i,\theta_i)$ with $\lambda_i\geq0$
and $\sum_{i=1}^4\lambda_i=1$. 
By the Carath\'eodory Theorem we conclude that $G_2={\rm co}_{(\xi,\theta)} G_0=H_{\rm eff}$.

To conclude the proof, using a standard double limit procedure, it is enough to show that for every $i=1,2$, $(u,e,p)\in A_{reg}(w)$, $z\in L^1(\Om)$, and $\eta>0$ there exist a sequence $(u_k,e_k,p_k)\in A_{reg}(w)$ and a sequence 
$z_k\in L^1(\Om)$ satisfying the properties of the second statement of the theorem
and such that 
\begin{equation}\label{1appv}
\G_i(p, z)+\eta \geq\limsup_{k\to\infty}
\G_{i-1}(p_k,z_k)
\,,
\end{equation}
where 
$$
\G_i(p,z):=\int_\Om G_i(x,p(x),z(x))\,dx
$$
for $i=0,1,2$.

Using the approximation argument introduced in \cite{MeySer} we can also assume
$z\in C^{\infty}(\Om)\cap L^1(\Om)$, $u\in C^{\infty}(\Om;\Rn)\cap BD(\Om)$,
$p\in C^{\infty}(\Om; \MD)\cap L^1(\Om;\MD)$.
Using the Lagrange interpolation on a locally finite grid composed by isosceles right triangles which becomes finer and finer near the boundary, we can replace these functions by new functions $u$, $e$, $p$, and $z$, with $(u,e,p)\in A_{reg}(w)$, such that $u$ is piecewise affine on this triangulation ${\mathcal T}$, while $e$, $p$, and $z$ are piecewise constant. 
Since $z_0$ is piecewise constant, it is not restrictive to assume that $G_i(\cdot,\xi,\theta)$ is piecewise constant on ${\mathcal T}$, so that $G_i(x,\xi,\theta)=G_{i,T}(\xi,\theta)$ for every $x\in T$ and every $T\in{\mathcal T}$. 
We may assume that every triangle $T$ of the triangulation ${\mathcal T}$ is relatively
compact in $\Om$.

Let us fix $i=1,2$ and $T\in{\mathcal T}$. Then
$$
u(x)=\xi_Tx+c_T \quad \text{for every } x\in T\,,
$$
where $\xi_T$ is a $2{\times}2$-matrix and $c_T\in\R^2$.
Moreover, we have
$$
e(x)=e_T\,, \quad p(x)=p_T\,, \quad z(x)=z_T
 \quad \text{for every } x\in T\,,
$$
where $e_T\in\Mnn$, $p_T\in\MD$, and $z_T\in\R$. Then we have $\xi_T=e_T+p_T+\omega_T$, where $\omega_T$ is a skew symmetric $2{\times}2$-matrix.

For every $\e>0$ there exists $(\lambda_T,p_T^1,p_T^2,z_T^1,z_T^2)\in\Lambda$
such that 
\begin{equation}\label{Gi-1}
G_{i,T}(p_T,z_T)+\e> \lambda_T G_{i-1,T}(p_T+p_T^1,z_T+z_T^1)+
(1- \lambda_T) G_{i-1,T}(p_T+p_T^2,z_T+z_T^2)\,.
\end{equation}
By an algebraic property of $\MD$ there exist $a_T,b_T\in\R^2$ such that
$p_T^2-p_T^1=a_T\otimes b_T+q_T$ with $q_T$ a skew symmetric $2{\times}2$-matrix.
Note that this is the only point where the dimension two is crucial.
By a standard lamination procedure with interfaces orthogonal to $b_T$ we can construct two sequences $v_T^k\in W^{1,\infty}_{loc}(\Rn;\Rn)$ and 
$z_T^k\in L^\infty_{loc}(\Rn)$ such that 
$v_T^k(0)=0$, $v_T^k\wto p_Tx$
weakly$^*$ in $W^{1,\infty}_{loc}(\Rn;\Rn)$, 
$z_T^k\wto z_T$
weakly$^*$ in $L^\infty_{loc}(\Rn)$,
$Ev_T^k=p_T+p_T^1$ and $z_T^k=z_T+z_T^1$ on $A_T^k$, 
$Ev_T^k=p_T+p_T^2$ and $z_T^k=z_T+z_T^2$ on $\Rn\setmeno A_T^k$, and $1_{A_T^k}\wto\lambda_T$
weakly$^*$ in $L^\infty_{loc}(\Rn)$.
Let us define $u_T^k(x):=
e_Tx+v_T^k(x)+\omega_Tx+c_T$. Recalling our definitions we find that $u_T^k\wto u$
weakly$^*$ in $W^{1,\infty}(T;\Rn)$ and $z_T^k\wto z$
weakly$^*$ in $L^\infty(T)$. 

For every $T\in \mathcal{T}$ and every $\delta>0$ let $T_\delta$ be the  triangle
similar to $T$ with the same centre and similarity ratio $1-\delta$, and let  $\varphi^\delta_T\in C^\infty_c(T)$ a cut-off function such that $\varphi^\delta_T=1$ on $T_\delta$  and $0\leq\varphi^\delta_T\leq1$ on $T$.
Let us fix a finite subset ${\mathcal T}'\subset{\mathcal T}$, let
\begin{eqnarray*}
&\displaystyle
\Om':=\bigcup_{T\in \mathcal {T}'}T\,, \qquad \Om'_\delta=\bigcup_{T\in \mathcal {T}'}T_\delta\,, 
\\
&\displaystyle
u_k:=\sum_{T\in{\mathcal T}'} \varphi_T^\delta u_T^k +\big(1 - \sum_{T\in{\mathcal T}'} \varphi_T^\delta )u\,,
\qquad
z_k:=\sum_{T\in{\mathcal T}'} \varphi_T^\delta z_T^k +\big(1 - \sum_{T\in{\mathcal T}'} \varphi_T^\delta )z\,.
\end{eqnarray*}
It is clear that $u_k\wto u$ weakly$^*$ in $BD(\Om)$ and $u_k=w$ $\hn$-a.e.\ on $\Gamma_0$.
We set 
\begin{eqnarray*}
&\displaystyle
p_k:=\sum_{T\in{\mathcal T}'} \varphi_T^\delta Ev^k_T+
\big(1 - \sum_{T\in{\mathcal T}'} \varphi_T^\delta )p\,,
\\
&\displaystyle
 e_k:=Eu_k-p_k=e+\sum_{T\in{\mathcal T}'}  
\nabla\varphi_T^\delta \odot(u^k_T-u)\,.
\end{eqnarray*}
It follows that $p_k\wto p$ weakly$^*$
 in $L^\infty(\Om;\MD)$ and $e_k\to e$ strongly in 
$L^{\infty}(\Om;\Mnn)$.
For every $T\in\T'$ we have $e_k=e_T$ a.e.\ on $T_\delta$, 
$p_k=p_T+p^1_T$ a.e.\ on $T_\delta \cap A_T^k$, $z_k=z_T+z^1_T$ 
a.e.\ on $T_\delta\cap A_T^k$,
$p_k=p_T+p^2_T$ a.e.\ on $T_\delta\setmeno A_T^k$, and $z_k=z_T+z^2_T$ a.e.\ on $T_\delta\setmeno A_T^k$. Therefore
\begin{eqnarray*}
&
\displaystyle \G_{i-1}(p_k,z_k)\leq \sum_{T\in \mathcal{T}'}
G_{i-1,T}(p_T+p^1_T, z_T+z_T^1)\Ln(T_\delta\cap A_T^k)+{}
\\
&
\displaystyle
{}+\sum_{T\in \mathcal{T}'}
G_{i-1,T}(p_T+p^2_T, z_T+z_T^2)\Ln(T_\delta\setmeno A_T^k)+{}
\\
&\displaystyle+\int_{\Om'\setminus\Om'_{\delta}}G_{i-1}(p_k,z_k)\, dx+\int_{\Om\setminus \Om'}G_{i-1}(p,z)\, dx\,.
\end{eqnarray*}
We observe that there exists a constant $C(\mathcal{T}')$ such that $G_{i-1}(p_k,z_k)\leq C(\mathcal{T}')$ a.e.\ on $\Om'$ for every $k$. As $1_{A_T^k}\wto\lambda_T$ weakly$^*$ in $L^{\infty}(T)$ as $k\to\infty$, using \eqref{Gi-1} we obtain
\begin{eqnarray*}
&\displaystyle\limsup_{k\to\infty}\G_{i-1}(p_k,z_k)\leq \sum_{T\in \mathcal{T}'}(G_{i,T}(p_T, z_T)+\e)\Ln(T_\delta)+\\
&\displaystyle{}+C(\mathcal{T}')\Ln(\Om'\setmeno\Om'_{\delta})+\int_{\Om\setminus \Om'}
G_{i-1}(p,z)\, dx\leq\\
&\displaystyle \leq\G_{i}(p,z)+\e\Ln(\Om)+C(\mathcal{T}')\Ln(\Om'\setmeno\Om'_{\delta})+\int_{\Om\setminus \Om'}G_{i-1}(p,z)\, dx\,,
\end{eqnarray*} 
which gives \eqref{1appv} with
$$
\eta:=\e\Ln(\Om)+C(\mathcal{T}')\Ln(\Om'\setmeno\Om'_{\delta})+\int_{\Om\setminus \Om'}G_{i-1}(p,z)\, dx\,.
$$
Passing to the limit first  as $\delta\to 0$, then as $\e\to0$, and finally as 
$\Om'\nearrow \Om$, we can make $\eta$ arbitrarily small, and this concludes the proof.
\end{proof}

\subsection{Relaxation in spaces of Young measures}

The following theorem shows the relationships between the incremental problem in $A_{reg}(\tilde w)$ with $H$ and $V$, the same problem in $A(\tilde w)$ with $H_{\rm eff}$, and a similar problem in a suitable space of generalized Young measures. The statement of the theorem uses the decomposition $\mu=\ol\mu^Y+\hat\mu^\infty$ of 
\cite[Theorem 4.3]{DM-DeS-Mor-Mor-1}, the notion of translation introduced
in \eqref{Tp}, and the
homogeneous function $\{V\}\colon \R{\times}\R\to\R$  defined by
\begin{equation}\label{Vhom}
\{V\}(\theta,\eta):=
\begin{cases}
\eta\, V(\theta/\eta) & \text{if }\eta>0\,,
\\
V^\infty(\theta) &  \text{if }\eta\le 0\,.
\end{cases}
\end{equation}

\begin{theorem}\label{lm:eqmin}
Let $w_0,\tilde w \in H^1(\Om;\Rn)$, let $(u_0,e_0,p_0)\in A(w_0)$, let $z_0\in M_b(\ol\Om)$,
let $\mu_0\in GY(\ol\Om; \MD{\times}\R)$ such that
$\bary(\mu_0)=(p_0,z_0)$. Assume that $\ol\mu_0^Y=\delta_{(\ol p_0,\ol z_0)}$ with
$\ol p_0\in L^1(\Om;\MD)$ and $\ol z_0\in L^1(\Om)$.
Then the following equalities hold:
\begin{eqnarray}
&\displaystyle
\!\!\!\!\!\!\!\!\inf_{(\tilde u,\tilde e,\tilde p)\in A_{reg}(\tilde w),\,
\tilde z\in L^1(\Om)}
\big[\QQ(e_0+\tilde e) + \HH(\tilde p,\tilde z)
+\langle \{V\}(\theta_1,\eta), \T_{(\tilde p,\tilde z)}
(\mu_0) \rangle
\big] =
\label{finf1}
\\
&\displaystyle
=\min_{(\tilde u,\tilde e,\tilde p)\in A(\tilde w),\,
\tilde z\in M_b(\ol\Om)}
\big[\QQ(e_0+\tilde e) +\HH_{\rm eff}(\tilde p,\tilde z)+
\langle \{V\}(\theta_0,\eta), \mu_0 \rangle
\big] =
\label{fmin1}
\\
&\displaystyle
=\inf_{( u, e, \muu)\in B}
\big[\QQ( e)+\langle H(\xi_1-\xi_0,\theta_1-\theta_0)+\{V\}(\theta_1,\eta),  
\muu_{t_0t_1}
\rangle\big]\,,
\label{fmin3}
\end{eqnarray}
where the measure $\T_{(\tilde p,\tilde z)}(\mu_0)$ acts on $(x,\xi_1,\theta_1,\eta)$, 
the measure $\mu_0$ acts on $(x,\xi_0,\theta_0,\eta)$, while the measure $\muu_{t_0t_1}$ acts on 
$(x,\xi_0,\theta_0,\xi_1,\theta_1,\eta)$. Here $B$ denotes the class of all triplets 
$(u, e,\muu)$, with $u\in BD(\Om)$, $e\in L^2(\Om;\Mnn)$, 
$\muu\in SGY(\{t_0,t_1\},\ol\Om;\MD{\times}\R)$,
such that $\muu_{t_0}=\mu_0$ and $(u, e, p)\in A(w_0+\tilde w)$, where 
$( p, z):=\bary(\muu_{t_1})$.
\end{theorem}

\begin{proof} 
We start by showing that the infimum in \eqref{finf1} is less than
or equal to the minimum in \eqref{fmin1}.
Let $(\tilde u, \tilde e, \tilde p)\in A(\tilde w)$ and $\tilde z\in M_b(\ol\Om)$ be a minimizer of \eqref{fmin1}. 
By Theorem~\ref{thm:1app} there exist a sequence 
$(\tilde u_m,\tilde e_m,\tilde p_m)\in A_{reg}(\tilde w)$ 
and a sequence $\tilde z_m\in L^1(\Om)$ such that 
$\tilde e_m\to \tilde e$ strongly in $L^2(\Om;\Mnn)$, $\tilde p_m\wto \tilde p$ weakly$^*$ in $M_b(\ol\Om;\MD)$, 
$\tilde z_m\wto \tilde z$ weakly$^*$ in $M_b(\ol\Om)$, and
\begin{equation}\label{110app1}
\HH(\tilde p_m,\tilde z_m) 
+\V(\ol z_0+\tilde z_m)-\V(\ol z_0) \quad \longrightarrow \quad
\HH_{\rm eff}(\tilde p,\tilde z)\,.
\end{equation}

We claim that
$$
\HH(\tilde p_m,\tilde z_m)
+\langle \{V\}(\theta_1,\eta), \T_{(\tilde p_m,\tilde z_m)}
(\mu_0) \rangle\quad
 \longrightarrow\quad \HH_{\rm eff}(\tilde p,\tilde z)+
\langle \{V\}(\theta_0,\eta), \mu_0 \rangle\,.
$$
Indeed, using the definition of $\T_{(\tilde p_m,\tilde z_m)}$, we have
$$
\langle \{V\}(\theta_1,\eta), \T_{(\tilde p_m,\tilde z_m)}(\mu_0) \rangle -
\langle \{V\}(\theta_0,\eta), \mu_0 \rangle =
\langle \{V\}(\theta_0+\eta\tilde z_m(x),\eta)-\{V\}(\theta_0,\eta), 
\mu_0 \rangle\,.
$$
As $\{V\}(\theta_0+\eta\tilde z_m(x),\eta)-\{V\}(\theta_0,\eta)$ vanishes for $\eta=0$, 
we obtain
$$
\langle \{V\}(\theta_1,\eta), \T_{(\tilde p_m,\tilde z_m)}(\mu_0) \rangle -
\langle \{V\}(\theta_0,\eta), \mu_0 \rangle =
\langle \{V\}(\theta_0+\eta \tilde z_m(x),\eta)-\{V\}(\theta_0,\eta), 
\ol\mu^Y_0\rangle\,.
$$
By the assumption $\ol\mu^Y_0=\delta_{(\ol p_0,\ol z_0)}$ we find
\begin{equation}\label{eqVV1}
\langle \{V\}(\theta_1,\eta), \T_{(\tilde p_m,\tilde z_m)}(\mu_0) \rangle -
\langle \{V\}(\theta_0,\eta), \mu_0 \rangle =
\V(\ol z_0+\tilde z_m)-\V(\ol z_0)\,.
\end{equation}
{}From (\ref{110app1}) and (\ref{eqVV1}) we obtain the claim, which, in turn, together with the strong
convergence of $\tilde e_m$ to $\tilde e$, shows that the infimum \eqref{finf1} is less than or equal to the minimum \eqref{fmin1}.

Let $( u, e,\muu)\in B$. By the Jensen inequality for generalized Young measures
(see \cite[Theorem~6.5]{DM-DeS-Mor-Mor-1}) we have
$$
\begin{array}{c}
\HH_{\rm eff}( p-p_0, z-z_0)\le 
\langle H_{\rm eff}(\xi_1-\xi_0,\theta_1-\theta_0),  
\muu_{t_0t_1}
\rangle\le
\smallskip
\\
\le \langle H(\xi_1-\xi_0,\theta_1-\theta_0)+\{V\}(\theta_1,\eta)- \{V\}(\theta_0,\eta),  
\muu_{t_0t_1}
\rangle\,.
\end{array}
$$
Since $( u-u_0, e-e_0, p-p_0)\in A(\tilde w)$, 
the minimum \eqref{fmin1} is less than or equal to 
the infimum in \eqref{fmin3}.

On the other hand the infimum in \eqref{finf1} is greater than or equal to
the infimum in \eqref{fmin3}, since for every $(\tilde u,\tilde e,\tilde p)\in A_{reg}(\tilde w)$ and every $\tilde z\in L^1(\Om)$ we can construct a triple $( u, e,\muu)\in B$ 
by setting $ u:=u_0+\tilde u$, $ e:=e_0+\tilde e$, and 
$\muu_{t_0t_1}:=
\T^1_{(\tilde p,\tilde z)}(\mu_0)$,
where $\T^1_{(\tilde p,\tilde z)}\colon \ol\Om{\times}\MD{\times}\R{\times}\R\to
\ol\Om{\times}(\MD{\times}\R)^2{\times}\R$ is defined by
$$
\T^1_{(\tilde p,\tilde z)}(x,\xi_0,\theta_0,\eta)
:=(x,\xi_0,\theta_0,\xi_0+\eta
\tilde p(x), \theta_0+\eta \tilde z(x),\eta)\,.
$$
This concludes the proof of the theorem.
\end{proof}

\subsection{Some structure theorems}

We prove now two structure theorems for generalized Young measures 
whose action on $H+\{V\}$ equals the relaxed 
functional $\HH_{\rm eff}$ evaluated on
their barycentres.

\begin{theorem}\label{thm311}
Let $p_0\in L^1(\Om;\MD)$, $z_0\in L^1(\Om)$, 
$\mu_1\in GY(\ol\Om;\MD{\times}\R)$, let $(p_1,z_1):=\bary(\mu_1)$,
let $\lambda$ be the total variation of the measure $(p_1^s,z_1^s)$, and
let $(p_1^\lambda, z_1^\lambda)$ be the Radon-Nikodym derivative of the measure
$(p_1^s,z_1^s)$ with respect to $\lambda$. 
Assume that 
\begin{equation}\label{eq092}
\begin{array}{c}
\langle  H(\xi_1-\eta p_0(x), \theta_1-\eta z_0(x)) + 
\{V\}(\theta_1,\eta),\mu_1(x,\xi_1,\theta_1,\eta)\rangle=
\\
=\HH_{\rm eff}(p_1-p_0,z_1-z_0)+\V(z_0)\,.
\end{array}
\end{equation}
By Lemma~\ref{cothetaG} there exist $z\in L^1(\Om)$, with $z(z_1^a-z_0)\ge 0$
a.e.\ on $\Om$, and $\alpha\in L^\infty(\Om)$,
with $\tfrac12\leq\alpha\leq1$  a.e.\ on $\Om$, such that 
\begin{eqnarray}
& z_1^a=\alpha(z_0+ z) + (1-\alpha)(z_0- z)\,, \label{eq090}
\\
& \label{eq09}
H_{\rm eff}(p_1^a-p_0,z_1^a-z_0)= H(p_1^a-p_0, z)+ V^\infty(z)
\end{eqnarray}
a.e.\ in $\Om$, and there exist $z_\lambda\in L^1_\lambda(\ol\Om)$, with $z_\lambda 
z_1^\lambda\ge 0$ $\lambda$-a.e.\ on $\ol\Om$,
and $\alpha_\lambda\in L^\infty_\lambda(\ol\Om)$, with $\tfrac12\le\alpha_\lambda\le1$
$\lambda$-a.e.\ on $\ol\Om$,
such that
\begin{eqnarray}
& z_1^\lambda=\alpha_\lambda z_\lambda + (1-\alpha_\lambda)
(-z_\lambda)\,,
\\
& \label{eq09+}
H_{\rm eff}(p_1^\lambda,z_1^\lambda)= 
H(p_1^\lambda, z_\lambda)+ V^\infty(z_\lambda)
\end{eqnarray}
$\lambda$-a.e.\ in $\ol\Om$. 
Then 
$$
\mu_1=\delta_{(p_0,z_0)}+
\alpha\omega_{(p_1^a-p_0, z)}^{\Ln}+
(1-\alpha)\omega_{(p_1^a-p_0,- z)}^{\Ln}+
\alpha_\lambda\omega_{(p_1^\lambda, z_\lambda)}^{\lambda}+
(1-\alpha_\lambda)\omega_{(p_1^\lambda, -z_\lambda)}^{\lambda}\,, 
$$
that is, according to \eqref{omega},  
\begin{equation}\label{eq091}
\begin{array}{c}
\displaystyle
\langle f,\mu_1\rangle =\int_\Om f(x,p_0(x),z_0(x),1)\, dx +
\int_\Om \alpha(x)f(x,p_1^a(x)-p_0(x), z(x),0)\, dx+{}
\\
\displaystyle
{}+\int_\Om (1-\alpha(x))f(x,p_1^a(x)-p_0(x),- z(x),0)\, dx +{}
\\
\displaystyle
{}+\int_{\ol\Om} \alpha_\lambda(x)f(x,p_1^\lambda(x), z_\lambda(x),0)\, d\lambda(x)+{}
\\
\displaystyle
{}+\int_{\ol\Om} (1-\alpha_\lambda(x))f(x,p_1^\lambda(x), -z_\lambda(x),0)\, d\lambda(x)
\end{array}
\end{equation}
for every $f\in B^{hom}_{\infty,1}(\ol\Om{\times}\MD{\times}\R{\times}\R)$
(see \cite[Definition~3.14]{DM-DeS-Mor-Mor-1}).
\end{theorem}

\begin{proof}
According to \cite[Remark~4.5]{DM-DeS-Mor-Mor-1} there exist $\lambda_1^\infty\in M_b^+(\ol\Om)$, 
a family $(\mu_1^{x,Y})_{x\in\Om}$ of probability measures on $\MD{\times}\R$, and
a family $(\mu_1^{x,\infty})_{x\in\ol\Om}$ of probability measures on 
$\Sigma:=\{(\xi,\theta)\in\MD{\times}\R: |\xi|^2+|\theta|^2=1\}$ such that
\begin{equation}\label{disint}
\begin{array}{c}
\displaystyle
\langle f,\mu_1\rangle =
\int_\Om \Big(\int_{\MD{\times}\R} 
f(x,\xi_1,\theta_1,1)\, d\mu_1^{x,Y}(\xi_1,\theta_1)\Big) dx
+ {}
\\
\displaystyle
{}+\int_{\ol\Om} \Big(\int_{\Sigma} f(x,\xi_1,\theta_1,0)\, 
d\mu_1^{x,\infty}(\xi_1,\theta_1)\Big) d\lambda_1^\infty(x)
\end{array}
\end{equation}
for every $f\in B^{hom}_{\infty,1}(\ol\Om{\times}\MD{\times}\R{\times}\R)$.
According to \cite[Remark~6.3]{DM-DeS-Mor-Mor-1}, we also have
$$
(p_1^a(x),z_1^a(x))= \int_{\MD{\times}\R}\!\!\! \!\!\! 
(\xi_1,\theta_1)\, d\mu_1^{x,Y}(\xi_1,\theta_1)+
\lambda_1^{\infty,a}(x)\int_\Sigma (\xi_1,\theta_1)\, d\mu_1^{x,\infty}(\xi_1,\theta_1)
$$
for $\Ln$-a.e.\ $x\in\Om$,
where $\lambda_1^{\infty,a}$ is the absolutely continuous part of $\lambda_1^\infty$. 
This implies
\begin{equation}\label{barac}
\begin{array}{c}
\displaystyle
(p_1^a(x)-p_0(x),z_1^a(x)-z_0(x))= \int_{\MD{\times}\R}\!\!\! \!\!\! 
 (\xi_1-p_0(x),\theta_1-z_0(x))\, d\mu_1^{x,Y}(\xi_1,\theta_1)+{}
 \\
 \displaystyle
{}+\lambda_1^{\infty,a}(x)\int_\Sigma (\xi_1,\theta_1)\, d\mu_1^{x,\infty}(\xi_1,\theta_1)\,.
\end{array}
\end{equation}
Let $\lambda_1^{\infty,s}$ be the singular part of $\lambda_1^\infty$.
By \cite[Remark~6.3]{DM-DeS-Mor-Mor-1}
we also have that $\lambda<<\lambda^{\infty,s}_1$ and
\begin{equation}\label{barsing}
(p_1^\lambda(x), z_1^\lambda(x))\frac{d\lambda}{d\lambda^{\infty,s}_1}(x)
=\int_\Sigma (\xi_1,\theta_1)\, d\mu_1^{x,\infty}(\xi_1,\theta_1)
\end{equation}
for $\lambda^{\infty,s}_1$-a.e.\ $x\in\ol\Om$.

Let $G^\infty$ be defined by \eqref{Ginfty}.
{}From \eqref{eq092} and \eqref{disint} we obtain 
\begin{equation}\label{bfor}
\begin{array}{c}
\displaystyle
\int_\Om H_{\rm eff}(p_1^a(x)-p_0(x),z_1^a(x)-z_0(x))\, dx +
 \int_{\ol\Om} H_{\rm eff}(p_1^\lambda(x),z_1^\lambda(x))\, d\lambda(x)
 =
\smallskip
\\
\displaystyle
= \int_\Om\!\!\Big(\!\int_{\MD{\times}\R} 
\!\!\!\!\big(H(\xi_1\!-\!p_0(x),\theta_1\!-\!z_0(x))+
V(\theta_1)-V(z_0(x))\big)\,d\mu_1^{x,Y}(\xi_1,\theta_1)\Big)dx +{}
\smallskip
\\
\displaystyle
{}+ \int_\Om \Big(\int_\Sigma G^\infty(\xi_1,\theta_1)
\,d\mu_1^{x,\infty}(\xi_1,\theta_1)\Big) 
\lambda_1^{\infty,a}(x)\, dx+{}
\smallskip
\\
\displaystyle
{}+ \int_{\ol\Om}\Big(\int_\Sigma G^\infty(\xi_1,\theta_1)
\,d\mu_1^{x,\infty}(\xi_1,\theta_1)\Big) 
d\lambda_1^{\infty,s}(x)\,.
\end{array}
\end{equation}
As $H_{\rm eff}={\rm co}\,G^\infty$, using the homogeneity of $H_{\rm eff}$ 
and the Jensen inequality, we
deduce from \eqref{barac} that for a.e.\ $x\in\Om$ we have
$$
\begin{array}{c}
\displaystyle
H_{\rm eff}(p_1^a(x)-p_0,z_1^a(x)-z_0(x))\leq
\int_{\MD{\times}\R} G^\infty(\xi_1-p_0(x),\theta_1-z_0(x))
\,d\mu_1^{x,Y}(\xi_1,\theta_1)+{}
\\
\displaystyle
{}+\lambda_1^{\infty,a}(x) \int_\Sigma G^\infty(\xi_1,\theta_1)
\,d\mu_1^{x,\infty}(\xi,\theta)\le
\\
\displaystyle
\le \int_{\MD{\times}\R}\big (H(\xi_1-p_0(x),\theta_1-z_0(x))+V(\theta_1)-V(z_0(x))\big)
\,d\mu_1^{x,Y}(\xi_1,\theta_1)+{}
\\
\displaystyle
{}+ \lambda_1^{\infty,a}(x) \int_\Sigma G^\infty(\xi_1,\theta_1)
\,d\mu_1^{x,\infty}(\xi_1,\theta_1)
\,,
\end{array}
$$
where the second inequality follows from \eqref{coF1}. Analogously, from \eqref{barsing}
we deduce that 
$$
H_{\rm eff}(p_1^\lambda(x), z_1^\lambda(x))\frac{d\lambda}{d\lambda^{\infty,s}_1}(x)
\le\int_\Sigma G^\infty(\xi_1,\theta_1)\, d\mu_1^{x,\infty}(\xi_1,\theta_1)
$$
for $\lambda_1^{\infty,s}$-a.e.\ $x\in \ol\Omega$.
Therefore, we deduce from \eqref{bfor} that
\begin{equation}\label{342p}
\begin{array}{c}
\displaystyle
\vphantom{ \int_{\MD{\times}\R}}
H_{\rm eff}(p_1^a(x)-p_0(x),z_1^a(x)-z_0(x))=
\\
\displaystyle
= \int_{\MD{\times}\R}\big (H(\xi_1-p_0(x),\theta_1-z_0(x))+V(\theta_1)-V(z_0(x))\big)
\,d\mu_1^{x,Y}(\xi_1,\theta_1)+{}
\\
\displaystyle
{}+ \lambda_1^{\infty,a}(x) \int_\Sigma G^\infty(\xi_1,\theta_1)
\,d\mu_1^{x,\infty}(\xi_1,\theta_1)
\end{array}
\end{equation}
for a.e.\ $x\in\Om$,
\begin{equation}\label{342psin}
H_{\rm eff}(p_1^\lambda(x), z_1^\lambda(x))
\frac{d\lambda}{d\lambda^{\infty,s}_1}(x)
=
\int_\Sigma G^\infty(\xi_1,\theta_1)\, d\mu_1^{x,\infty}(\xi_1,\theta_1)
\end{equation}
for $\lambda^{\infty,s}_1$-a.e.\ $x\in\ol\Om$.

Let us define
$A:=\{x\in\Om: p_1^a(x)=p_0(x),\, z_1^a(x)=z_0(x)\}$ and 
$A_\lambda:=\{x\in\ol\Om: p_1^\lambda(x)=0,\, z_1^\lambda(x)=0\}$.
By \eqref{342p} and \eqref{342psin} we have
\begin{eqnarray}
&\displaystyle
\hspace{-3em}\int_{\MD{\times}\R} 
\!\!\!\!\!\!\!\!\big(H(\xi_1\!-\!p_0(x),\theta_1\!-\!z_0(x))\!+\!
V(\theta_1)\!-\!V(z_0(x))\big)d\mu_1^{x,Y}(\xi_1,\theta_1)
=0
\ \text{for a.e.\ } x\in A,\label{339p}
\\
&\displaystyle
\lambda_1^{\infty,a}(x)
\int_\Sigma G^\infty(\xi_1,\theta_1)\,d\mu_1^{x,\infty}(\xi_1,\theta_1)=0
\quad \text{for a.e.\ } x\in A\,, \label{340p}
\\
&\displaystyle
\int_\Sigma G^\infty(\xi_1,\theta_1)\,d\mu_1^{x,\infty}(\xi_1,\theta_1)=0
\quad \text{for }\lambda^{\infty,s}_1\text{-a.e.\ } x\in A_\lambda\,. \label{340.5p}
\end{eqnarray}
By \eqref{gammaM} and \eqref{339p} $\mu_1^{x,Y}$ is concentrated on 
$(p_0(x),z_0(x))$,
hence
\begin{equation}\label{end1}
\mu_1^{x,Y}=\delta_{(p_0(x),z_0(x))} \quad \text{for a.e.\ } x\in A\,.
\end{equation}
Since $G^\infty$ is strictly positive on $\Sigma$ and $\mu_1^{x,\infty}$ are
probability measures, we deduce from \eqref{340p} and \eqref{340.5p} that
\begin{eqnarray}
& \label{end1.5}
\lambda_1^{\infty,a}(x)=0 \quad \text{for a.e.\ } x\in A\,,
\\
& \label{end4}
\lambda_1^{\infty,s}(A_\lambda)=0\,.
\end{eqnarray}

Let us consider now $\mu_1^{x,Y}$ for $x\in B:=\Om\setminus A$.
For every $x\in B$ let $L(x,\cdot,\cdot)\colon\MD{\times}\R\to \R$
be a linear function such that $L(x,p_1^a(x)-p_0(x),z_1^a(x)-z_0(x))=
H_{\rm eff}(p_1^a(x)-p_0(x),z_1^a(x)-z_0(x))$ and
$L(x,\xi,\theta)\leq H_{\rm eff}(\xi,\theta)$
for every $(\xi,\theta)\in \MD{\times}\R$.
Using \eqref{barac}, \eqref{342p}, and the
linearity of $L$, for a.e.\ $x\in B$ we obtain
\begin{equation}\label{346p}
\begin{array}{c}
\displaystyle
\int_{\MD{\times}\R} L(x,\xi_1-p_0(x),\theta_1-z_0(x))\,d\mu_1^{x,Y}(\xi_1,\theta_1)+{}
\\
\displaystyle
{}+\lambda_1^{\infty,a}(x) \int_\Sigma L(x,\xi_1,\theta_1)\,d\mu_1^{x,\infty}(\xi_1,\theta_1)=
\\
\displaystyle
= \int_{\MD{\times}\R}\big (H(\xi_1-p_0(x),\theta_1-z_0(x))+V(\theta_1)-V(z_0(x))\big)
\,d\mu_1^{x,Y}(\xi_1,\theta_1)+{}
\\
\displaystyle
{}+ \lambda_1^{\infty,a}(x) \int_\Sigma G^\infty(\xi_1,\theta_1)
\,d\mu_1^{x,\infty}(\xi_1,\theta_1)\,.
\end{array}
\end{equation}
Using \eqref{coF1} we find
$L(x,\xi_1-p_0(x),\theta_1-z_0(x))\leq 
H(\xi_1-p_0(x),\theta_1-z_0(x))+V(\theta_1)-V(z_0(x))$.
Therefore, equality \eqref{346p} implies that for a.e.\ $x\in B$ we have
\begin{equation}\label{347p}
L(x,\xi_1-p_0(x),\theta_1-z_0(x))= H(\xi_1-p_0(x),\theta_1-z_0(x))+
V(\theta_1)-V(z_0(x))
\end{equation}
for $\mu_1^{x,Y}$-a.e.\ $(\xi_1,\theta_1)\in \MD{\times}\R$, and
\begin{equation}\label{348p}
\lambda_1^{\infty,a}(x) L(x,\xi_1,\theta_1)=\lambda_1^{\infty,a}(x) G^\infty(\xi_1,\theta_1)  \end{equation}
for $\mu_1^{x,\infty}$-a.e.\ $(\xi_1,\theta_1)\in \Sigma$.
As $L(x,\xi_1-p_0(x),\theta_1-z_0(x))\leq H_{\rm eff}(\xi_1-p_0(x),\theta_1-z_0(x)) \leq H(\xi_1-p_0(x),\theta_1-z_0(x))+V(\theta_1)-V(z_0(x))$, we deduce 
from \eqref{347p} and Lemma~\ref{HeffG} that 
$(\xi_1,\theta_1)=(p_0(x),z_0(x))$ for $\mu_1^{x,Y}$-a.e.\ 
$(\xi_1,\theta_1)\in \MD{\times}\R$. This implies that 
$\mu_1^{x,Y}$ is concentrated on $(p_0(x),z_0(x))$. Since $\mu_1^{x,Y}$ is a probability measure, 
we conclude that 
\begin{equation}\label{end2}
\mu_1^{x,Y}=\delta_{(p_0(x),z_0(x))} \quad \text{for a.e.\ } x\in B\,.
\end{equation}

We now consider the measures $\mu_1^{x,\infty}$.
We first observe that $z(x)\neq0$ for a.e.\ $x\in B$ by \eqref{eq090}, \eqref{eq09},
and Lemma~\ref{lemma01}.
For every $x\in B$ we define 
$$
\varphi(x):=\sqrt{|p_1^a(x)-p_0(x)|^2+ z(x)^2}\quad \hbox{and} \quad
(\hat p(x), \hat z(x)):=(p_1^a(x)-p_0(x), z(x))/\varphi(x)\,.
$$
By \eqref{348p} and Lemma~\ref{lemma01} for a.e.\ $x\in B$
with $\lambda_1^{\infty,a}(x)\neq0$ we have
$$
(\xi_1,\theta_1)\in \{(\hat p(x),\hat z(x)), (\hat p(x),-\hat z(x))\}
\quad \text{for } \mu_1^{x,\infty}\text{-a.e.\ } (\xi_1,\theta_1)\in\Sigma\,,
$$
so that
\begin{equation}\label{muinfD}
\mu_1^{x,\infty}=\beta(x)\delta_{(\hat p(x),\hat z(x))}+(1-\beta(x))
\delta_{(\hat p(x),-\hat z(x))}
\end{equation}
for a suitable $\beta(x)\in[0,1]$.
Using \eqref{barac} we find that
\begin{equation}\label{pzbar}
p_1^a-p_0= \hat p\,\lambda_1^{\infty,a}\,, \qquad 
z_1^a-z_0=(2\beta-1)\hat z\,\lambda_1^{\infty,a}
\end{equation}
a.e.\ in $B$. 
Since $p_1^a-p_0= \hat p \varphi$, the first equality implies 
that 
\begin{equation}\label{end2.5}
\lambda_1^{\infty,a}=\varphi \quad \text{a.e.\ in }\{x\in B: \hat p(x)\neq0\}\,. 
\end{equation}
Since $z_1^a-z_0=(2\alpha-1)z=(2\alpha-1)\varphi\hat  z$ and $\hat z\neq0$, 
the second equality in \eqref{pzbar} implies that
$\alpha=\beta$ a.e.\ in $\{x\in B: \hat p(x)\neq0\}$. Therefore,
\begin{equation}\label{end3}
\mu_1^{x,\infty}=\alpha(x)\delta_{(\hat p(x),\hat z(x))}+
(1-\alpha(x))\delta_{(\hat p(x),-\hat z(x))} \quad \text{a.e.\ in } \{x\in B: \hat p(x)\neq0\}\,.
\end{equation}
As $H_{\rm eff}(0,\theta)=G^\infty(0,\theta)=(a_K-b_V)|\theta|$ for every $\theta\in \R$
by \eqref{H0zeta}, if $\hat p(x)=0$
we deduce from \eqref{eq09} that $z(x)=z_1^a(x)-z_0(x)$ and $\alpha(x)=1$. Then the second
equality in \eqref{pzbar} implies that 
\begin{equation}\label{eq362}
(2\beta(x)-1)\lambda^{\infty,a}_1(x)=|z(x)| \quad \text{a.e.\ in } \{x\in B: \hat p(x)=0\}\,.
\end{equation}
Therefore, from \eqref{342p} and \eqref{muinfD} we deduce that
$$
\begin{array}{c}
G^\infty(0,z(x))=H_{\rm eff}(0,z(x))=
\smallskip
\\
=
\lambda_1^{\infty,a}(x)\big(
\beta(x)G^\infty(0,\frac{z(x)}{|z(x)|})+(1-\beta(x))G^\infty(0,-\frac{z(x)}{|z(x)|})\big)=
\smallskip
\\
=
\frac{1}{2\beta(x)-1}G^\infty(0,z(x))\,,
\end{array}
$$
hence $\beta=\alpha=1$ a.e.\ in $\{x\in B: \hat p(x)=0\}$. By \eqref{eq362}
we have $\lambda_1^{\infty,a}=|z|=\varphi$ a.e.\ in $ \{x\in B: \hat p(x)=0\}$.
Using also \eqref{muinfD}, \eqref{end2.5}, and \eqref{end3}
we conclude that
\begin{equation}\label{end3t}
\begin{array}{c}
\mu_1^{x,\infty}=\alpha(x)\delta_{(\hat p(x),\hat z(x))}+
(1-\alpha(x))\delta_{(\hat p(x),-\hat z(x))} \quad \text{a.e.\ in } B\,,
\smallskip
\\
\lambda_1^{\infty,a}=\varphi\quad \text{a.e.\ in } B\,.
\end{array}
\end{equation}

Let us consider now the properties of the measure $\mu_1^{x,\infty}$
for $\lambda^{\infty,s}_1$-a.e.\ $x\in\ol\Om$. For every $x\in B_\lambda:=
\ol\Om\setminus A_\lambda$ we define 
$$
\varphi_\lambda(x):=\sqrt{|p_1^\lambda(x))|^2+ z_\lambda(x)^2}\quad \hbox{and} \quad
(\hat p_\lambda(x), \hat z_\lambda(x)):=(p_1^\lambda(x), z_\lambda(x))/
\varphi_\lambda(x)\,,
$$
and notice that 
$\varphi_\lambda(x)\geq \sqrt{|p_1^\lambda(x))|^2+ z_1^\lambda(x)^2}=1$.
As in the previous step we consider a linear function
$L_\lambda(x,\cdot,\cdot)\colon\MD{\times}\R\to \R$
such that $L_\lambda(x,p_1^\lambda(x),z_1^\lambda(x))=
H_{\rm eff}(p_1^\lambda(x),z_1^\lambda(x))$ and
$L_\lambda(x,\xi,\theta)\leq H_{\rm eff}(\xi,\theta)$
for every $(\xi,\theta)\in \MD{\times}\R$.
Using \eqref{barsing}, \eqref{342psin}, and the
linearity of $L_\lambda$, for $\lambda^{\infty,s}_1$-a.e.\ $x\in B_\lambda$ we obtain
\begin{equation}\label{346psin}
\int_\Sigma L_\lambda(x,\xi_1,\theta_1)\, d\mu_1^{x,\infty}(\xi_1,\theta_1)=
\int_\Sigma G^\infty(\xi_1,\theta_1)\, d\mu_1^{x,\infty}(\xi_1,\theta_1)\,.
\end{equation}
Since $L_\lambda(x,\xi,\theta)\leq G^\infty(\xi,\theta)$, for $\lambda^{\infty,s}_1$-a.e.\ $x\in B_\lambda$ we deduce that
$$
L_\lambda(x,\xi_1,\theta_1)=G^\infty(\xi_1,\theta_1)
$$
for $\mu_1^{x,\infty}$-a.e.\ $(\xi_1,\theta_1)\in\Sigma$.
By Lemma~\ref{lemma01} for $\lambda^{\infty,s}_1$-a.e.\ $x\in B_\lambda$
we have $z_\lambda(x)\neq0$ and
$$
(\xi_1,\theta_1)\in \{(\hat p_\lambda(x),\hat z_\lambda(x)), 
(\hat p_\lambda(x),-\hat z_\lambda(x))\}
\quad \text{for } \mu_1^{x,\infty}\text{-a.e.\ } (\xi_1,\theta_1)\in\Sigma\,,
$$
so that
\begin{equation}\label{muinfDS}
\mu_1^{x,\infty}=\beta_\lambda(x)\delta_{(\hat p_\lambda(x),\hat z_\lambda(x))}
+(1-\beta_\lambda(x))
\delta_{(\hat p_\lambda(x),-\hat z_\lambda(x))}
\end{equation}
for a suitable $\beta_\lambda(x)\in[0,1]$.
Using \eqref{barsing} we find that
\begin{equation}\label{pzbarsin}
p_1^\lambda\frac{d\lambda}{d\lambda^{\infty,s}_1}= \hat p_\lambda\,, \qquad 
z_1^\lambda\frac{d\lambda}{d\lambda^{\infty,s}_1}=(2\beta_\lambda-1)
\hat z_\lambda
\end{equation}
$\lambda^{\infty,s}_1$-a.e.\ in $B_\lambda$.
Since $p_1^\lambda= \hat p_\lambda \varphi_\lambda$, the first equality implies 
that 
\begin{equation}\label{end2.5sin}
\varphi_\lambda\frac{d\lambda}{d\lambda^{\infty,s}_1}=1 \quad 
\lambda^{\infty,s}_1\text{-a.e.\ in } \{x\in B_\lambda: \hat p_\lambda(x)\neq0\}\,. 
\end{equation}
Since $z_1^\lambda=(2\alpha_\lambda-1)z_\lambda=(2\alpha_\lambda-1)\varphi_\lambda\hat  z_\lambda$, 
the second equality in \eqref{pzbarsin} implies that
$\alpha_\lambda=\beta_\lambda$ $\lambda^{\infty,s}_1$-a.e.\ in $\{x\in B_\lambda: \hat p_\lambda(x)\neq0\}$. Therefore,
\begin{equation}\label{end3sin}
\begin{array}{c}
\mu_1^{x,\infty}=\alpha_\lambda(x)\delta_{(\hat p_\lambda(x),\hat z_\lambda(x))}+
(1-\alpha_\lambda(x))\delta_{(\hat p_\lambda(x),-\hat z_\lambda(x))}
\medskip
 \\
\lambda^{\infty,s}_1\text{-a.e.\ in } \{x\in B_\lambda: \hat p_\lambda(x)\neq0\}\,. 
\end{array}
\end{equation}
As $H_{\rm eff}(0,\theta)=G^\infty(0,\theta)$ for every $\theta\in \R$, 
if $\hat p_\lambda(x)=0$ we deduce from \eqref{eq09+} that $z_\lambda(x)=z_1^\lambda(x)$ and $\alpha_\lambda(x)=1$. Then the second
equality in \eqref{pzbarsin} implies that 
\begin{equation}\label{star}
2\beta_\lambda(x)-1=|z_\lambda(x)|\frac{d\lambda}{d\lambda^{\infty,s}_1}(x) \quad 
\lambda^{\infty,s}_1\text{-a.e.\ in } \{x\in B_\lambda: \hat p_\lambda(x)=0\}\,.
\end{equation}
Therefore, from \eqref{342psin} and \eqref{muinfDS} we deduce that
$$
\begin{array}{c}
G^\infty(0,z_\lambda(x))=H_{\rm eff}(0,z_\lambda(x))=
\smallskip
\\
=
\frac{|z_\lambda(x)|}{2\beta_\lambda(x)-1}\big(
\beta_\lambda(x)G^\infty(0,\frac{z_\lambda(x)}{|z_\lambda(x)|})+(1-\beta_\lambda(x))G^\infty(0,-\frac{z_\lambda(x)}{|z_\lambda(x)|})\big)=
\smallskip
\\
=
\frac{1}{2\beta_\lambda(x)-1}G^\infty(0,z_\lambda(x))\,,
\end{array}
$$
hence $\beta_\lambda=\alpha_\lambda=1$ $\lambda^{\infty,s}_1$-a.e.\ in $\{x\in B_\lambda: \hat p_\lambda(x)=0\}$. 
By \eqref{star} we have $\varphi_\lambda\frac{d\lambda}{d\lambda^{\infty,s}_1}=
|z_\lambda|\frac{d\lambda}{d\lambda^{\infty,s}_1}=1$
$\lambda^{\infty,s}_1$-a.e.\ in $\{x\in B_\lambda: \hat p_\lambda(x)=0\}$.
Using also \eqref{end4}, \eqref{muinfDS}, \eqref{end2.5sin}, and \eqref{end3sin}
we conclude that
\begin{equation}\label{end3sint}
\begin{array}{c}
\mu_1^{x,\infty}=\alpha_\lambda(x)\delta_{(\hat p_\lambda(x),\hat z_\lambda(x))}+
(1-\alpha_\lambda(x))\delta_{(\hat p_\lambda(x),-\hat z_\lambda(x))} \quad 
\lambda^{\infty,s}_1\text{-a.e.\ in } \ol\Om\,,
\smallskip
\\
\displaystyle\varphi_\lambda\frac{d\lambda}{d\lambda^{\infty,s}_1}=1 \quad 
\lambda^{\infty,s}_1\text{-a.e.\ in } \ol\Om\,.
\end{array}
\end{equation}
It follows that $\lambda^{\infty,s}_1<<\lambda$ and that
\begin{equation}\label{starstar}
\frac{d\lambda^{\infty,s}_1}{d\lambda}=
\varphi_\lambda \quad 
\lambda\text{-a.e.\ in } \ol\Om\,.
\end{equation}

The conclusion follows from \eqref{disint}, \eqref{end1}, \eqref{end1.5},
\eqref{end2}, \eqref{end3t}, \eqref{end3sint}, and \eqref{starstar},
using the homogeneity of $f$.
\end{proof}

To prove the next theorem we need two technical results.

\begin{lemma}\label{prop-prod}
Let $\Xi_1,\Xi_2$ be finite dimensional Hilbert spaces and let 
$\pi_i\colon\ol\Om{\times}\Xi_1{\times} \Xi_2{\times}\R\to\ol\Om{\times}\Xi_i{\times}\R$, 
$i=1,2$, be the projections defined by $\pi_i(x,\xi_1,\xi_2,\eta):=(x,\xi_i,\eta)$. 
Let $\mu\in GY(\ol\Om;\Xi_1{\times}\Xi_2)$ and let $p\in L^1(\Om;\Xi_1)$. 
Assume that $\pi_1(\mu)=\delta_p$ and let $\mu_2:=\pi_2(\mu)$. Then
\begin{equation}\label{prod}
\langle f(x,\xi_1,\xi_2,\eta),\mu(x,\xi_1,\xi_2,\eta)\rangle = 
\langle f(x,\eta p(x),\xi_2,\eta),\mu_2(x,\xi_2,\eta)\rangle
\end{equation}
for every $f\in B^{hom}_{\infty,1}(\ol\Om{\times}\Xi_1{\times} \Xi_2{\times}\R)$.
\end{lemma}

\begin{proof}
Using \cite[Definition~3.16]{DM-DeS-Mor-Mor-1} and standard arguments in measure theory, it is enough to prove
\eqref{prod} for every $f\in C^{hom}(\ol\Om{\times}\Xi_1{\times} \Xi_2{\times}\R)$.

By the definition of $\mu_2$ we have
$$
\langle f(x,\eta p(x),\xi_2,\eta),\mu_2(x,\xi_2,\eta)\rangle
=\langle f(x,\eta p(x),\xi_2,\eta),\mu(x,\xi_1,\xi_2,\eta)\rangle\,.
$$
Therefore, to prove (\ref{prod}) it is enough to show that
$$
\langle f(x,\xi_1,\xi_2,\eta) - f(x,\eta p(x),\xi_2,\eta),\mu(x,\xi_1,\xi_2,\eta)\rangle=0\,.
$$
By approximation it suffices to prove this equality when $f$ is Lipschitz continuous 
with respect to $\xi_1,\xi_2,\eta$ with a constant $L$ independent of $x$ 
(see \cite[Lemma~2.4]{DM-DeS-Mor-Mor-1}). 
In this case we have
$$
|\langle f(x,\xi_1,\xi_2,\eta) - f(x,\eta p(x),\xi_2,\eta),\mu(x,\xi_1,\xi_2,\eta)\rangle|\le
L\langle |\eta p(x)-\xi_1|,\mu(x,\xi_1,\xi_2,\eta)\rangle\,.
$$
As $\pi_1(\mu)=\delta_p$, we have
$$
\langle |\eta p(x)-\xi_1|,\mu(x,\xi_1,\xi_2,\eta)\rangle=
\langle |\eta p(x)-\xi_1|,\delta_p(x,\xi_1,\eta)\rangle=0\,,
$$
which concludes the proof.
\end{proof}

\begin{corollary}\label{cor-prod}
Let $\Xi_1$, $\Xi_2$, $\pi_1$, $\pi_2$, $\mu$, and $p$ be as in 
Lemma~\ref{prop-prod}, and let $\mu_1:=\pi_1(\mu)$ and $\mu_2:=\pi_2(\mu)$. Assume that $\ol\mu_1^Y=\delta_p$. Then
$$
\langle f(x,\xi_1,\xi_2,\eta),\ol\mu^Y(x,\xi_1,\xi_2,\eta)\rangle = 
\langle f(x,\eta p(x),\xi_2,\eta),\ol\mu_2^Y(x,\xi_2,\eta)\rangle
$$
for every $f\in B^{hom}_{\infty,1}(\ol\Om{\times}\Xi_1{\times} \Xi_2{\times}\R)$.
\end{corollary}

\begin{proof}
It is enough to apply Lemma~\ref{prop-prod} to $\ol\mu^Y$, using  \cite[Lemma~4.8]{DM-DeS-Mor-Mor-1}.
\end{proof}

\begin{theorem}\label{lmstruc}
Let $\muu\in SGY(\{t_0,t_1\},\ol\Om; \MD{\times}\R)$, let $(p_0,z_0):=\bary(\muu_{t_0})$, and
let $( p_1, z_1):=\bary(\muu_{t_1})$.
Assume that 
\begin{equation}\label{eq092dP}
\begin{array}{c}
\langle  H(\xi_1-\xi_0, \theta_1-\theta_0) + \{V\}(\theta_1,\eta)
- \{V\}(\theta_0,\eta),\muu_{t_0t_1}(x,\xi_0,\theta_0,\xi_1,\theta_1,\eta)\rangle=
\smallskip
\\
=\HH_{\rm eff}( p_1-p_0, z_1-z_0)
\end{array}
\end{equation}
and that
\begin{equation}\label{eq092d+P}
\ol\muu_{t_0}^Y=\delta_{(\ol p_0,\ol z_0)}
\end{equation}
with $\ol p_0\in L^1(\Om;\MD)$ and $\ol z_0\in L^1(\Om)$.
Then $\ol\muu_{t_1}^Y=\delta_{(\ol p_0,\ol z_0)}$.
\end{theorem}

\begin{proof}
If $p_0=\ol p_0\in L^1(\Om;\MD)$, $z_0=\ol z_0\in L^1(\Om)$, and $\muu_{t_0}=\delta_{(p_0,z_0)}$, then \eqref{eq092dP} implies \eqref{eq092} 
by Lemma~\ref{prop-prod} with $\mu_1:=\muu_{t_1}$, and the conclusion follows from
Theorem~\ref{thm311}. 

We consider now the general case. 
Let $\phi(x,\xi_0,\theta_0,\xi_1,\theta_1,\eta):=(x,\xi_0,\theta_0, \xi_1-\xi_0,\theta_1-\theta_0,\eta)$ and 
let $\pi(x,\xi_0,\theta_0,\tilde\xi,\tilde\theta,\eta):=
(x, \tilde\xi, \tilde\theta,\eta)$. 
We define $\nu:=(\pi\circ\phi)(\muu_{t_0t_1})$ and observe that
\begin{equation}\label{bnutd}
\bary(\nu)=(p_1 -p_0, z_1 -z_0)\,.
\end{equation}
By (\ref{eq092dP}) we have
\begin{equation}\label{eq9111d}
\langle H(\tilde\xi, \tilde \theta)+\{V\}(\theta_0+\tilde\theta,\eta)-\{V\}(\theta_0,\eta), \phi(\muu_{t_0t_1}) 
\rangle = 
\HH_{\rm eff}(p_1 -p_0, z_1 -z_0)\,,
\end{equation}
where the measure $\phi(\muu_{t_0t_1})$ acts on the variables 
$(x,\xi_0,\theta_0,\tilde\xi,\tilde\theta,\eta)$.
Moreover, since $\nu=\pi(\phi(\muu_{t_0t_1}))$, we have that 
\begin{equation}\label{Hpp1d}
\langle H(\tilde\xi, \tilde \theta),\phi(\muu_{t_0t_1}) \rangle = 
\langle H(\tilde\xi, \tilde \theta), \nu \rangle\,,
\end{equation}
where the measure $\nu$ acts on the variables $(x,\tilde\xi,\tilde\theta,\eta)$.
We consider the decomposition
\begin{eqnarray}
&\nonumber
\langle \{V\}(\theta_0+\tilde\theta,\eta)-
\{V\}(\theta_0,\eta), \phi(\muu_{t_0t_1}) \rangle = 
\langle \{V\}(\theta_0+\tilde\theta,\eta)-
\{V\}(\theta_0,\eta), \ol{\phi(\muu_{t_0t_1})}{}^Y \rangle+{}
\smallskip 
\\
\label{Vpp1d}
&{}+ \langle \{V\}(\theta_0+\tilde\theta,\eta)-
\{V\}(\theta_0,\eta), \widehat{\phi(\muu_{t_0t_1})}{}^\infty \rangle
\end{eqnarray}
given by \cite[Theorem~4.3]{DM-DeS-Mor-Mor-1}.
As $\muu_{t_0}$ is the image of the measure $\phi(\muu_{t_0t_1})$ under the map 
$(x,\xi_0,\theta_0,\tilde\xi,\tilde\theta,\eta)
\mapsto(x,\xi_0,\theta_0,\eta)$, 
by \eqref{eq092d+P} we can apply Corollary~\ref{cor-prod} and we obtain
\begin{equation}\label{Vpp2d}
\langle \{V\}(\theta_0+\tilde\theta,\eta)-
\{V\}(\theta_0,\eta), \ol{\phi(\muu_{t_0t_1})}{}^Y \rangle=
\langle \{V\}(\eta \ol z_0(x)+\tilde\theta,\eta)-
\{V\}(\eta \ol z_0(x),\eta), \ol \nu^Y\rangle\,.
\end{equation}
Since by concavity $\{V\}(\theta_0+\tilde\theta,\eta)
-\{V\}(\theta_0,\eta)\ge V^\infty(\tilde\theta)$, we have
\begin{equation}\label{Vpp3d}
\begin{array}{c}
\langle \{V\}(\theta_0+\tilde\theta,\eta)-
\{V\}(\theta_0,\eta), \widehat{\phi(\muu_{t_0t_1})}{}^\infty \rangle \ge
\langle V^\infty(\tilde\theta),  \widehat{\phi(\muu_{t_0t_1})}{}^\infty \rangle = 
\smallskip
\\
=\langle V^\infty(\tilde\theta),\pi(\widehat{\phi(\muu_{t_0t_1})}{}^\infty)\rangle
= \langle V^\infty(\tilde\theta),\hat\nu^\infty\rangle =
\smallskip
\\
=\langle \{V\}(\eta \ol z_0(x)+\tilde\theta,\eta)-
\{V\}(\eta \ol z_0(x),\eta), \hat\nu^\infty \rangle
\,,
\end{array}
\end{equation}
where the second equality follows from \cite[Lemma 4.8]{DM-DeS-Mor-Mor-1},
taking into account that $\nu=\pi(\phi(\muu_{t_0t_1}))$.
By (\ref{eq9111d})--(\ref{Vpp3d}) we obtain
$$
\HH_{\rm eff}(p_1 -p_0, z_1 -z_0)\ge
\langle H(\tilde\xi, \tilde \theta)+\{V\}(\eta \ol z_0(x)+\tilde\theta,\eta)
-\{V\}(\eta \ol z_0(x),\eta), \nu\rangle\,.
$$
By the Jensen inequality for generalized Young measures
\cite[Theorem~6.5]{DM-DeS-Mor-Mor-1} we deduce from \eqref{bnutd} that
\begin{equation}\label{Vpp4d}
\HH_{\rm eff}(p_1 -p_0, z_1 -z_0) =
\langle H(\tilde\xi, \tilde \theta)+\{V\}(\eta \ol z_0(x)+\tilde\theta,\eta)
-\{V\}(\eta \ol z_0(x),\eta), \nu\rangle\,.
\end{equation}

Let us fix $x\in\Om$ and let $G\colon \MD{\times}\R{\times}\R\to\R$ be
the function defined by
$$
G(\tilde\xi,\tilde\theta,\eta):=
H(\tilde\xi, \tilde \theta)+\{V\}(\eta \ol z_0(x)+\tilde\theta,\eta)-\{V\}(\eta \ol z_0(x),\eta)\,.
$$
It follows from (\ref{coF3}) and \eqref{Heff} that the function 
$H_{\rm eff}(\tilde\xi,\tilde\theta)$ is the convex envelope of $G(\tilde\xi,\tilde\theta,\eta)$
with respect to $(\tilde\xi,\tilde\theta,\eta)$. 
Moreover, by (\ref{He=G}) we deduce that for every $\eta>0$ 
the equality
$$
H_{\rm eff}(\tilde\xi,\tilde\theta)=H(\tilde\xi, \tilde \theta)+\{V\}(\eta \ol z_0(x)+\tilde\theta,\eta)-\{V\}(\eta \ol z_0(x),\eta)
$$ 
holds if and only if $(\tilde\xi,\tilde\theta)=(0,0)$. 
Therefore \eqref{bnutd}, (\ref{Vpp4d}), and
\cite[Lemma~6.7]{DM-DeS-Mor-Mor-1} imply that
$$
\supp\,\nu\subset\{(x,0,0,\eta):x\in\ol\Om\,,\
\eta\ge 0\}\cup(\ol\Om{\times}\MD{\times}\R{\times}\{0\})\,,
$$
and, in particular,
$$
\supp\,\ol\nu^Y\subset\{(x,0,0,\eta): x\in\ol\Om\,,\ \eta\ge 0\}\,,
$$
hence $\ol\nu^Y=\delta_{(0,0)}$. 

{}From the definitions of $\nu$ and from
 \cite[Lemma~4.8]{DM-DeS-Mor-Mor-1}
it follows that
\begin{equation}\label{yy23}
\begin{array}{c}\displaystyle
\langle f,\delta_{(\ol p_0,\ol z_0)}\rangle =
\int_\Om f(x,\ol p_0(x),\ol z_0(x),1)\, dx 
= \langle f(x,\tilde\xi+\eta \ol p_0(x),\tilde\theta+\eta \ol z_0(x),\eta), \ol\nu^Y\rangle =
\smallskip \\
=
\langle f(x,\xi_1 -\xi_0+\eta \ol p_0(x), \theta_1 -\theta_0+\eta \ol z_0(x),\eta), \ol{\muu}{}^Y_{t_0t_1}\rangle
\end{array}
\end{equation}
for every 
$f\in C^{hom}(\ol\Om{\times}\MD{\times}\R{\times}\R)$.
By \eqref{eq092d+P} we can apply Corollary 
\ref{cor-prod} and we obtain that the last term in the previous formula equals
$\langle f(x,\xi_1, \theta_1 ,\eta), \ol{\muu}{}^Y_{t_1}\rangle$.
Therefore, $\ol{\muu}{}^Y_{t_1}=\delta_{(\ol p_0,\ol z_0)}$.
\end{proof}
\end{section}

\begin{section}{Globally stable quasistatic evolution for Young measures}\label{sec:4}

\subsection{Definitions and main result}
We begin with the definition of the set of {\it admissible triples} in the Young measure formulation, with boundary datum $\ww$ on $\Gamma_0$.

\begin{definition}\label{newD}
Given a set $\Theta\subset\R$ and a map $\ww\colon \Theta\to H^1(\Om;\Rn)$, 
we define $AY(\Theta,\ww)$ as the set of all triples $(\uu,\ee,\muu)$ with 
$\uu\colon\Theta\to BD(\Om)$, $\ee\colon\Theta\to L^2(\Om;\Mnn)$, 
$\muu\in SGY(\Theta, \ol\Om;\MD{\times}\R)$,
with the following property: for every finite sequence $t_1,\dots,t_m$ in $\Theta$, 
with $t_1<\dots<t_m$, and every $i=1,\dots,m$
there exist a sequence $(u_k^i,e_k^i,p_k^i)\in A(\ww(t_i))$ and a sequence 
$z_k^i\in M_b(\ol\Om)$ such that 
\begin{eqnarray*}
& u_k^i\wto \uu(t_i) \quad \text{weakly}^* \text{ in } BD(\Om)\,,
\\
& e_k^i\to \ee(t_i) \quad \text{strongly in } L^2(\Om;\Mnn)\,,
\end{eqnarray*}
and
\begin{equation}\label{apprmeas}
\delta_{((p^1_k,z^1_k),\dots, (p^m_k,z^m_k))} \wto
\muu_{t_1\dots t_m} \quad \text{weakly}^* \text{ in } 
GY(\ol\Om;(\MD{\times}\R)^m)\,.
\end{equation}
\end{definition}

\begin{remark}\label{rmk:bound}
Since the weak$^*$ convergence in $GY(\ol\Om;(\MD{\times}\R)^m)$ implies
the convergence of the norms $\|\cdot\|_*$ (see
\cite[Remark~3.12]{DM-DeS-Mor-Mor-1}), it is not restrictive to assume that
$$
\|\eki \|_2\le \|\ee(t_i)\|_2+1\,,\quad
\|(\pki,\zki)\|_*\le \|\muu_{t_i}\|_*+1
$$
for every $i$ and $k$.
As $(\uki,\eki,\pki)\in A(\ww(t_i))$, 
there exists a constant $C_i$, depending only on $\ww$, $t_i$, $ \|\ee(t_i)\|_2$,
$\|\muu_{t_i}\|_*$, such that
$$
\|\uki\|_1+\|E\uki\|_1\le C_i
$$
for every $k$.
\end{remark}

\begin{remark}\label{rmk:barAYn}
It follows from \cite[Remark~6.4]{DM-DeS-Mor-Mor-1} that
$$
\begin{array}{c}
\pki \wto \pp(t_i) \quad \text{weakly}^* \text{ in } M_b(\ol\Om;\MD)\,,
\smallskip
\\
\zki \wto \zz(t_i) \quad \text{weakly}^* \text{ in } M_b(\ol\Om)\,,
\end{array}
$$
where $(\pp(t_i),\zz(t_i)):=\bary(\muu_{t_i})$.
As $(\uki,\eki,\pki)\in A(\ww(t_i))$,
by \cite[Lemma~2.1]{DM-DeS-Mor} we conclude that 
\begin{equation}\label{baryti}
(\uu(t_i),\ee(t_i),\pp(t_i))\in A(\ww(t_i))\,.
\end{equation}
\end{remark}

\begin{remark}\label{rmk:closAY}
The inequalities proved in Remark~\ref{rmk:bound} allow to use the metrizability
of the weak$^*$ topology on bounded subsets of the dual of a separable Banach space
and to prove that the set $AY(\Theta,\ww)$ satisfies
the following closure property: 
if $\uu\colon\Theta\to BD(\Om)$, $\ee\colon\Theta\to L^2(\Om;\Mnn)$,
$\muu\in SGY(\Theta, \ol\Om;\MD{\times}\R)$, and
$(\uu_k,\ee_k,\muu_k)$ is a sequence in $AY(\Theta,\ww)$ 
such that
\begin{eqnarray}
& \uu_k(t)\wto \uu(t) \quad \text{weakly}^* \text{ in } BD(\Om)\,,
\label{eqrmk1}
\\
& \ee_k(t)\to \ee(t) \quad \text{strongly in } L^2(\Om;\Mnn)
\label{eqrmk2}
\end{eqnarray}
for every $t\in\Theta$, and 
\begin{equation}\label{eqrmk3}
(\muu_k)_{t_1\dots t_m} \wto
\muu_{t_1\dots t_m} \quad \text{weakly}^* \text{ in } 
GY(\ol\Om;(\MD{\times}\R)^m)
\end{equation}
for every finite sequence $t_1,\dots,t_m$ in $\Theta$, 
with $t_1<\dots<t_m$, then
$(\uu,\ee,\muu)\in AY(\Theta,\ww)$.

More in general, if $\uu\colon\Theta\to BD(\Om)$, $\ee\colon\Theta\to L^2(\Om;\Mnn)$,
$\muu\in SGY(\Theta, \ol\Om;\MD{\times}\R)$, and
$(\uu_k,\ee_k,\muu_k)$ is a sequence in $AY(\Theta,\ww_k)$  
such that \eqref{eqrmk1}--\eqref{eqrmk3} hold and $\ww_k(t)\to \ww(t)$
strongly in $H^1(\Om;\Rn)$
for every $t\in\Theta$, then $(\uu,\ee,\muu)\in AY(\Theta,\ww)$.
This follows from the closure property, observing that $(\uu_k-\ww_k+\ww,\ee_k-E\ww_k+E\ww,\muu_k)$
belongs to $AY(\Theta,\ww)$.
\end{remark}

\begin{remark}\label{rmk:primo}
Using Theorem~\ref{thm:d-app}, Lemma~\ref{lm:comp}, and \cite[Theorem~3]{Res} 
(see also \cite[Appendix]{Luc-Mod}), it is easy to see that the definition does not change
if we replace $A(\ww(t_i))$ by $A_{reg}(\ww(t_i))$.
\end{remark}

Given $\muu\in SGY([0,+\infty),\ol\Om; \MD{\times}\R)$, its dissipation $\D_{\!H}(\muu;a,b)$ on the time interval $[a,b]\subset [0,+\infty)$ is defined as 
\begin{equation}\label{nudiss}
\sup \sum_{i=1}^k 
\langle H(\xi_i-\xi_{i-1},\theta_i-\theta_{i-1}), 
\muu_{t_0t_1\dots t_k}(x,\xi_0,\theta_0,\dots,\xi_k,\theta_k,\eta) \rangle\,,
\end{equation}
where the supremum is taken over all finite families $t_0,t_1,\dots,t_k$ such that 
$a=t_0<t_1<\dots<t_k=b$.
As in the case of the variation ${\rm Var}(\muu;a,b)$ considered in 
\cite[Section 8]{DM-DeS-Mor-Mor-1}, we have
\begin{equation}\label{nudiss2}
\D_{\!H}(\muu;a,b) = 
\sup \sum_{i=1}^k \langle H(\xi_i-\xi_{i-1},\theta_i-\theta_{i-1}),
\muu_{t_{i-1}t_i}(x,\xi_{i-1},\theta_{i-1},\xi_i,\theta_i,\eta)\rangle\,,
\end{equation}
where the supremum is taken over all finite families $t_0,t_1,\dots,t_k$ such that 
$a=t_0<t_1<\dots<t_k=b$.

In the following definition we use the homogeneous function $\{V\}$ defined by \eqref{Vhom} and the notion of weakly$^*$ left-continuous system of generalized Young measures introduced in \cite[Definition~7.6]{DM-DeS-Mor-Mor-1}.

\begin{definition}\label{maindef}
Given $\ww\in AC_{loc}([0,+\infty);H^1(\Om;\Rn))$, a {\it globally stable quasistatic evolution of Young measures\/} with boundary datum $\ww$ is a triple 
$(\uu,\ee,\muu)\in AY([0,+\infty), \ww)$, 
with $\uu$, $\ee$, $\muu$ weakly$^*$ left-continuous,
such that the following conditions are satisfied:
\begin{itemize}
\smallskip
\item[(ev1)] {\it global stability:\/} for every $t\in[0,+\infty)$
we have
\begin{eqnarray}
& \QQ(\ee(t))+\langle \{V\}(\theta,\eta), \muu_t(x, \xi,\theta,\eta)\rangle
\le \nonumber \\
& \le \QQ(\ee(t)+\tilde e) + \HH(\tilde p,\tilde z) +
\langle \{V\}(\theta+\eta\, \tilde z(x),\eta), 
\muu_t(x, \xi, \theta,\eta)\rangle \nonumber
\end{eqnarray}
for every $(\tilde u,\tilde e,\tilde p)\in A_{reg}(0)$ and every $\tilde z\in L^1(\Om)$;
\smallskip
\item[(ev2)] {\it energy balance:\/} for every $T\in (0,+\infty)$ we have
${\rm Var}(\muu;0,T)<+\infty$ and
\begin{eqnarray}
& \QQ(\ee(T)) + \D_{\!H}(\muu;0,T) +
 \langle \{V\}(\theta,\eta), \muu_T(x, \xi, \theta, \eta)\rangle
= \nonumber \\
& \displaystyle
 = \QQ(\ee(0)) + \langle \{V\}(\theta,\eta), \muu_0(x, \xi, \theta, \eta)\rangle +
\int_0^T \langle\sigmaa(t), E\dot \ww(t)\rangle\, dt\,,\nonumber
\end{eqnarray}
where $\sigmaa(t):=\C \ee(t)$.
\end{itemize}
\end{definition}

We are now in a position to state the main theorem of the paper.

\begin{theorem}\label{young-main} 
Let $\ww\in AC_{loc}([0+\infty);H^1(\Om;\Rn))$, $(u_0,e_0,p_0)\in A(\ww(0))$, and 
$z_0\in  M_b(\ol\Om)$. Assume that
\begin{equation}\label{initstab}
\QQ(e_0)+ \V(z_0) \le
 \QQ(e_0+\tilde e)+\HH(\tilde p, \tilde z)
+\V(z_0+\tilde z)
\end{equation}
for every $(\tilde u,\tilde e,\tilde p)\in A_{reg}(0)$ and every $\tilde z\in  L^1(\Om)$.
Then there exists a globally stable quasistatic evolution of Young measures 
$(\uu,\ee,\muu)$ with boundary datum $\ww$ such that $\uu(0)=u_0$,
$\ee(0)=e_0$, and $\muu_0=\delta_{(p_0,z_0)}$.
\end{theorem}

\subsection{The incremental minimum problems}
The proof of Theorem \ref{young-main} will be obtained by time discretization, using an implicit Euler scheme.
Let us fix a sequence of subdivisions $(\tki)^{}_{i\geq 0}$ of the half-line $[0,+\infty)$, with
\begin{eqnarray}
& 0=t_k^0<t_k^1<\dots<t_k^{i-1}<t_k^{i}\to+\infty\quad\hbox{as } i\to\infty\,,
\label{subdiv5}\\
&\displaystyle
\tau_k:=\sup_i (\tki-\tkim)\to 0\quad\hbox{as } k\to\infty\,.
\label{fine5}
\end{eqnarray}
For every $k$ let $w_k^i:=\ww(t_k^i)$ for $i\ge 0$ and let
$\tilde w_k^i:=\ww(t_k^i)-\ww(t_k^{i\!-\!1})$ for $i\ge 1$. 

We define $u^i_k\in BD(\Om)$, $e^i_k\in L^2(\Om;\Mnn)$, and 
$\muu^i_k\in SGY(\{{t^0_k,\dots, t^i_k}\},{\ol\Om }; \MD{\times}\R)$ by induction on $i$. 
We set $u^0_k:=u_0$, $e^0_k:=e_0$, $\muu^0_k:=\delta_{(p_0,z_0)}$, 
and for $i\ge 1$ we define $(u^i_k,e^i_k,\muu^i_k)$ as a minimizer 
(see Lemma~\ref{existinc} below)
of the functional
\begin{equation}\label{functinc}
\QQ(e)+\langle H(\xi_i -\xi_{i\!-\!1}, \theta_i -\theta_{i\!-\!1})+ \{V\}(\theta_i,\eta), 
\nuu_{t^{i\!-\!1}_k t^i_k}(x, \xi_{i\!-\!1}, \theta_{i\!-\!1}, \xi_i,\theta_i,\eta)\rangle 
\end{equation}
over the set $A^i_k$ of all triplets $(u,e,\nuu)$ with
$u\in BD(\Om)$, $e\in L^2(\Om;\Mnn)$, and
$\nuu\in SGY(\{{t^0_k,\dots,t^i_k}\},\ol\Om;\MD{\times}\R)$,  
with the following property:
there exist a sequence\break 
$(\tilde u_m,\tilde e_m,\tilde p_m)\in A_{reg}(\tilde
w^i_k)$ and
a sequence $\tilde z_m\in L^1(\Om)$ such that
\begin{eqnarray*}
& u_k^{i\!-\!1}+\tilde u_m \wto u \quad \text{weakly}^* \text{ in } BD(\Om)\,,
\\
& e_k^{i\!-\!1}+\tilde e_m \to e \quad \text{strongly in } L^2(\Om;\Mnn)\,,
\\
&
\T^i_{(\tilde p_m, \tilde z_m)}((\muu_k^{i\!-\!1})_{t_k^0\dots t_k^{i\!-\!1}}) \wto \nuu_{t_k^0\dots t_k^i} \quad \text{weakly}^* \text{ in } GY(\ol\Om; (\MD{\times}\R)^{i\!+\!1})\,,
\end{eqnarray*}
where $\T^i_{(\tilde p,\tilde z)}\colon \ol\Om{\times}(\MD{\times}\R)^i{\times}\R\to
\ol\Om{\times}(\MD{\times}\R)^{i+1}{\times}\R$ is defined by
$$
\T^i_{(\tilde p,\tilde z)}(x,\xi_0,\theta_0,\dots,\xi_{i\!-\!1}, \theta_{i\!-\!1},\eta)
:=(x,\xi_0,\theta_0,\dots,\xi_{i\!-\!1}, \theta_{i\!-\!1},\xi_{i\!-\!1}+\eta
\tilde p(x), \theta_{i\!-\!1}+\eta \tilde z(x),\eta)\,.
$$

We note that if $(u,e,\nuu)\in A^i_k$, then
\begin{eqnarray}
&\nuu_{t^0_k\dots t^{i\!-\!1}_k}=(\muu^{i\!-\!1}_k)_{t^0_k\dots t^{i\!-\!1}_k}\,,
\label{compatib2}
\\
&(u,e,p)\in A(\wki)\,,
\label{AY}
\end{eqnarray}
where $(p,z):=\bary(\nuu_{\tki})$. Then we define
$(p^i_k,z^i_k):=\bary((\muu^i_k)_{t^i_k})$. 

\begin{remark}\label{rmk:rmk}
The following equalities hold:
\begin{eqnarray*}
&\displaystyle 
\inf_{(u, e,\nuu)\in A_k^i} \big[
\QQ(e)+\langle H(\xi_i -\xi_{i\!-\!1}, \theta_i -\theta_{i\!-\!1})+ \{V\}(\theta_i,\eta), 
\nuu_{t^{i\!-\!1}_k t^i_k}\rangle 
\big]=
\\
&\displaystyle
=\!\!\!\!\!\!\!\!\inf_{\genfrac{}{}{0pt}2{\scriptstyle (\tilde u,\tilde e,\tilde p)\in 
A_{reg}(\tilde w_k^i)}
{\scriptstyle \tilde z\in L^1(\Om)}}
\!\!\big[\QQ(e_k^{i\!-\!1}+\tilde e) 
+\langle H(\xi_i-\xi_{i\!-\!1},\theta_i-\theta_{i\!-\!1})+\{V\}(\theta_i,\eta),
 \T^i_{(\tilde p,\tilde z)}
((\muu_k^{i\!-\!1})_{t^0_k\dots t^{i\!-\!1}_k}) \rangle
\big]\!=
\\
&\displaystyle
\!\!\!\!\!\!\!\!=\inf_{\genfrac{}{}{0pt}2{\scriptstyle (\tilde u,\tilde e,\tilde p)\in 
A_{reg}(\tilde w_k^i)}
{\scriptstyle \tilde z\in L^1(\Om)}}
\big[\QQ(e_k^{i\!-\!1}+\tilde e) + \HH(\tilde p,\tilde z)
+\langle \{V\}(\theta_i,\eta), \T_{(\tilde p,\tilde z)}
((\muu_k^{i\!-\!1})_{t^{i\!-\!1}_k}) \rangle
\big]=
\\
&\displaystyle
=\inf_{(u, e, \nuu)\in B_k^i} \big[
\QQ(e)+\langle H(\xi_i -\xi_{i\!-\!1}, \theta_i -\theta_{i\!-\!1})+ \{V\}(\theta_i,\eta),  
\nuu_{t^{i\!-\!1}_k t^i_k}\rangle
\big]\,,
\end{eqnarray*}
where $B_k^i$ is the class of all triplets $(u, e,\nuu)$, with $u\in BD(\Om)$, $e\in L^2(\Om;\Mnn)$, 
$\nuu\in SGY(\{t_k^0,\dots,\tki \},\ol\Om;\MD{\times}\R)$,
such that $\nuu_{t_k^0\dots t^{i\!-\!1}_k}=(\muu^{i\!-\!1}_k)_{t^0_k\dots t^{i\!-\!1}_k}$ and 
$(u,e, p)\in A(w_k^i)$, where 
$(p,z):=\bary(\nuu_{t_k^i})$.
The first two equalities follow from the definition of $A_k^i$ and the continuity properties of the functional \eqref{functinc}. On the other hand the infimum in the last line is greater than or equal to the infimum in the previous line by Theorem~\ref{lm:eqmin}, and is less than or equal to the infimum in the first line, since $A_k^i\subset B_k^i$ by 
\eqref{compatib2} and~\eqref{AY}.
\end{remark}

The existence of a minimizer $(u^i_k,e^i_k,\muu^i_k)$
to \eqref{functinc} is guaranteed by the following lemma.

\begin{lemma}\label{existinc}
For every $i$ the functional (\ref{functinc}) has a minimizer on $A^i_k$, every minimizer $(u^i_k, e^i_k, \muu^i_k)$ satisfies $\ol{(\muu_k^i)}_{t_k^i}^Y=\delta_{(p_0^a,z_0^a)}$,
and
\begin{equation}\label{eqbmin}
\QQ(e^i_k)+\HH_{\rm eff}(p^i_k -p^{i\!-\!1}_k, z^i_k -z^{i\!-\!1}_k) \le
\QQ(e)+\HH_{\rm eff}(p -p^{i\!-\!1}_k, z -z^{i\!-\!1}_k) 
\end{equation}
for every $(u,e,p)\in A(\ww(t^i_k))$ and every $z\in M_b(\ol\Om)$. 
\end{lemma}

\begin{proof}
The lemma will be proved by induction on $i$.
Assume that $\muu_k^{i\!-\!1}$ is defined and
$\ol{(\muu_k^{i\!-\!1})}_{t_k^{i\!-\!1}}^Y=\delta_{(p_0^a,z_0^a)}$.
We shall prove that the functional (\ref{functinc}) has a minimizer $(\uki,\eki,\muu_k^i)$
in $A^i_k$ and
\begin{equation}\label{assind}
\ol{(\muu_k^i)}_{t_k^i}^Y=\delta_{(p_0^a,z_0^a)}\,.
\end{equation}

Thanks to Remark~\ref{rmk:rmk} there exists a minimizing sequence
$(u_m,e_m,\nuu^m)$ in $A^i_k$ with $\nuu^m_{t_k^0\dots t_k^i}=\T^i_{(\tilde p_m,\tilde z_m)}((\muu^{i\!-\!1}_k)_{t_k^0\dots t_k^{i\!-\!1}})$ and $(u_m,e_m,\tilde p_m)\in A_{reg}(\tilde w^i_k)$. 
By (\ref{gammaM}) we have
$$
H(\xi_i\!-\!\xi_{i\!-\!1}, \theta_i \!-\!\theta_{i\!-\!1})+\{V\}(\theta_i,\eta)
\ge C^K_V |\xi_i\!-\!\xi_{i\!-\!1}| +  C^K_V |\theta_i\!-\!\theta_{i\!-\!1}| +  
\{V\}(\theta_{i\!-\!1},\eta)\,,
$$
hence by the compatibility condition~(7.2) of \cite{DM-DeS-Mor-Mor-1}
the sequence
\begin{eqnarray*}
&\QQ(e^m)+C^K_V\langle |\xi_i-\xi_{i\!-\!1}|+ |\theta_i-\theta_{i\!-\!1}|,
\nuu^m_{t^{i\!-\!1}_k t^i_k}(x, \xi_{i\!-\!1}, \theta_{i\!-\!1}, \xi_i,\theta_i,\eta)\rangle+
\\
&{}+\langle \{V\}(\theta_{i\!-\!1},\eta), 
\nuu^m_{t^{i\!-\!1}_k}(x, \xi_{i\!-\!1}, \theta_{i\!-\!1},\eta)\rangle 
\end{eqnarray*}
is bounded uniformly with respect to $m$. 
By (\ref{compatib2}) we have $\nuu^m_{t^{i\!-\!1}_k}=(\muu^{i\!-\!1}_k)_{t^{i\!-\!1}_k}$, 
so that
\begin{equation}\label{bound20}
\QQ(e^m)+C^K_V\langle |\xi_i-\xi_{i\!-\!1}|+ |\theta_i -\theta_{i\!-\!1}|,
\nuu^m_{t^{i\!-\!1}_k t^i_k}(x, \xi_{i\!-\!1}, \theta_{i\!-\!1}, \xi_i,\theta_i,\eta)\rangle
\end{equation}
is bounded uniformly with respect to $m$. 
Since by \cite[Remark~2.9 and~(7.2)]{DM-DeS-Mor-Mor-1}
\begin{eqnarray*}
&\|\nuu^m_{ t^i_k}\|_*\le \langle |\xi_i|+ |\theta_i|,
\nuu^m_{ t^i_k}(x, \xi_i,\theta_i,\eta)\rangle=
\\
&=
\langle  |\xi_i|+ |\theta_i|,
\nuu^m_{t^{i\!-\!1}_k t^i_k}(x, \xi_{i\!-\!1}, \theta_{i\!-\!1}, \xi_i,\theta_i,\eta)\rangle \le
\\
&
\le \langle |\xi_i\!-\!\xi_{i\!-\!1}|+ |\theta_i \!-\!\theta_{i\!-\!1}|,
\nuu^m_{t^{i\!-\!1}_k t^i_k}(x, \xi_{i\!-\!1}, \theta_{i\!-\!1}, \xi_i,\theta_i,\eta)\rangle
+
\\
&{}+
\langle |\xi_{i\!-\!1}|+ |\theta_{i\!-\!1}|,
\nuu^m_{t^{i\!-\!1}_k t^i_k}(x, \xi_{i\!-\!1}, \theta_{i\!-\!1}, \xi_i,\theta_i,\eta)\rangle=
\\
&
=  \langle |\xi_i\!-\!\xi_{i\!-\!1}|+ |\theta_i \!-\!\theta_{i\!-\!1}|,
\nuu^m_{t^{i\!-\!1}_k t^i_k}
(x, \xi_{i\!-\!1}, \theta_{i\!-\!1}, \xi_i,\theta_i,\eta)\rangle
+
\\
&{}+
\langle |\xi_{i\!-\!1}|+ |\theta_{i\!-\!1}|, (\muu^{i\!-\!1}_k)_{t^{i\!-\!1}_k}
(x, \xi_{i\!-\!1}, \theta_{i\!-\!1}, \eta)\rangle\,,
\end{eqnarray*}
it follows from (\ref{boundsC}) and (\ref{bound20}) and that $e^m$ 
is bounded in $L^2(\Om;\Mnn)$ and 
$\nuu^m_{ t^i_k}$ is bounded in $GY({\ol\Om };\MD{\times}\R)$. 
Using (\ref{compatib2}) and \cite[Lemma~7.8]{DM-DeS-Mor-Mor-1} we obtain also that 
$\nuu^m_{t^0_k\dots t^i_k}$ is bounded in $GY({\ol\Om };(\MD{\times}\R)^{i\!+\!1})$.

Passing to a subsequence, we may assume that $e_m\wto e$ weakly in 
$L^2(\Om;\Mnn)$ and $\nuu^m_{t^0_k \dots t^i_k}\wto \nu_{0\dots i}$ 
weakly$^*$ in  $GY({\ol\Om };(\MD{\times}\R)^{i\!+\!1})$. 
Let $\nuu\in SGY(\{{t^0_k,\dots,t^i_k}\},\ol\Om;\MD{\times}\R)$ be the system 
associated with $\nu_{0\dots i}$ according to 
\cite[Remark 7.9]{DM-DeS-Mor-Mor-1} and let $(p,z):=\bary(\nuu_{\tki})$. 
Note that $(p_k^{i-1}+\tilde p_m,z_k^{i-1}+\tilde z_m)=\bary(\nuu^m_{ t^i_k})
\wto (p,z)$ weakly$^*$ in $M_b(\ol\Om;\MD){\times}M_b(\ol\Om)$
by \cite[Remark~6.4]{DM-DeS-Mor-Mor-1}.
Since $(u_m,e_m,\nuu^m)\in A^i_k$ we have $\|Eu^m\|_1\le \|e^m\|_1+
\|\bary(\nuu^m_{ t^i_k})\|_1$ and $\|\ww(t^i_k)-u^m\|_{1,\Gamma_0}\le 
\|\bary(\nuu^m_{ t^i_k})\|_1$. 
By \cite[Proposition~2.4 and Remark~2.5]{Tem} it follows that $u^m$ is bounded in
$BD(\Om)$. Therefore, passing to a subsequence, we may assume that $u^m\wto u$ weakly$^*$ in $BD(\Om)$. 
By \cite[Lemma~2.1]{DM-DeS-Mor} it follows that $(u,e,p)\in A(w_k^i)$, hence
$(u,e,\nuu)\in B_k^i$.

We claim that
\begin{equation}\label{estrong}
e^m\to e \qquad \text{strongly in } L^2(\Om;\Mnn)\,.
\end{equation}
Indeed, if not, then we can find a subsequence (not relabelled) such that
\begin{equation}\label{ABstrong}
\QQ(e)<\lim_m \QQ(e_m)\,.
\end{equation}
Since the other term in (\ref{functinc}) is continuous with respect to the 
weak$^*$ convergence of $\nuu_{t_k^{i\!-\!1}t_k^i}^m$ to $\nuu_{t_k^{i\!-\!1}t_k^i}$, \eqref{ABstrong}
would imply that
$$
\begin{array}{c}
\QQ(e)+ \langle H(\xi_i -\xi_{i\!-\!1}, \theta_i -\theta_{i\!-\!1})+ \{V\}(\theta_i,\eta), 
\nuu_{t^{i\!-\!1}_k t^i_k}(x, \xi_{i\!-\!1}, \theta_{i\!-\!1}, \xi_i,\theta_i,\eta)\rangle <
\smallskip
\\
\displaystyle
< \inf_{(\hat u,\hat e,\hat\nuu)\in A_k^i} \big[
\QQ(\hat e)+\langle H(\xi_i -\xi_{i\!-\!1}, \theta_i -\theta_{i\!-\!1})+ \{V\}(\theta_i,\eta), 
\hat \nuu_{t^{i\!-\!1}_k t^i_k}\rangle 
\big]\,,
\end{array}
$$
which contradicts the equalities in Remark~\ref{rmk:rmk}, since $(u,e,\nuu)\in B_k^i$. Therefore, \eqref{estrong} is proved.

We deduce from \eqref{estrong} that $(u,e,\nuu)\in A_k^i$ and that it is a minimizer of (\ref{functinc}) in $A^i_k$. From now on we set $(\uki,\eki,\muu_k^i):=(u,e,\nuu)$.
By Remark~\ref{rmk:rmk} and Theorem~\ref{lm:eqmin} we obtain
\begin{equation}\label{rest}
\begin{array}{c}
\displaystyle
\min_{(\tilde u,\tilde e,\tilde p)\in A(\tilde w_k^i),\,
\tilde z\in M_b(\ol\Om)}
\big[\QQ(e_k^{i\!-\!1} +\tilde e) +\HH_{\rm eff}(\tilde p,\tilde z)+
\langle \{V\}(\theta_{i\!-\!1},\eta), (\muu_k^{i\!-\!1})_{t^{i\!-\!1}_k} \rangle
\big] =
\smallskip
\\
\displaystyle
=\QQ(\eki)+\langle H(\xi_i-\xi_{i\!-\!1},\theta_i-\theta_{i\!-\!1})+\{V\}(\theta_i,\eta),  
(\muu_k^i)_{t^{i\!-\!1}_k t^i_k} \rangle \geq
\medskip
\\
\displaystyle
\geq \QQ(\eki)+\HH_{\rm eff}(\pki-p_k^{i\!-\!1},\zki-z_k^{i\!-\!1})+ 
\langle \{V\}(\theta_{i\!-\!1},\eta), (\muu_k^{i\!-\!1})_{t^{i\!-\!1}_k} \rangle\,,
\end{array}
\end{equation}
where the last inequality follows from Jensen inequality.
Since $(\uki-u_k^{i\!-\!1},\eki-e_k^{i\!-\!1},\pki-p_k^{i\!-\!1})\in A(\tilde w_k^i)$, we deduce that
the previous inequalities are in fact equalities. Theorem~\ref{lmstruc} now yields \eqref{assind}.
Finally, \eqref{eqbmin} easily follows from \eqref{rest}.
\end{proof}

\begin{corollary}\label{corbarymin}
For every $i$ and $k$ we have
$$
\QQ(e^i_k)+\HH_{\rm eff}(p^i_k -p^{i\!-\!1}_k, 0) \le
\QQ(e)+\HH_{\rm eff}(p -p^{i\!-\!1}_k,0) 
$$
for every $(u,e,p)\in A(\ww(t^i_k))$.
\end{corollary}

\begin{proof}
It is enough to take $z=z^{i\!-\!1}_k$ in (\ref{eqbmin}) and to use the inequality 
$H_{\rm eff}(\xi,\theta)\ge H_{\rm eff}(\xi,0)$, which follows from the fact that
 $\theta\mapsto H_{\rm eff}(\xi,\theta)$ is convex and even.
\end{proof}

The following theorem shows that the incremental problems can be considered as a
relaxed version of incremental problems defined on functions.
For different approaches to the relaxation problem in the context of rate-independent processes we refer to \cite{Mie-Ort} and \cite{Mie-Rou-Ste}.

\begin{proposition}\label{relax}
Let us fix $k$. Let $(u^{m}, e^{m}, p^{m})$ be a sequence in $A_{reg}(\ww(0))$ and let  $z^{m}$ be a sequence in $L^1(\Om)$. For every $i\ge 1$ let us consider two sequences (with respect to the index $m$) 
$(\tilde u^{i,m}, \tilde e^{i,m}, \tilde p^{i,m})
\in A_{reg}(\tilde w_k^i)$ and $\tilde z^{i,m}\in L^1(\Om)$.
For every multiindex $m_0\dots m_i$ with $i+1$ components
we define
 $$
\begin{array}{c}
\displaystyle
u^{m_0\dots m_i}:=u^{m_0}+\sum_{j=1}^i \tilde u^{j,m_j}\,, \qquad
e^{m_0\dots m_i}:=e^{m_0}+\sum_{j=1}^i \tilde e^{j,m_j}\,,
\smallskip
\\
\displaystyle
p^{m_0\dots m_i}:=p^{m_0}+\sum_{j=1}^i \tilde p^{j,m_j}\,, \qquad
z^{m_0\dots m_i}:=z^{m_0}+\sum_{j=1}^i \tilde z^{j,m_j}\,.
\end{array}
$$
Note that 
$$
(u^{m_0\dots m_i},e^{m_0\dots m_i},p^{m_0\dots m_i})\in A_{reg}(w_k^i)\,.
$$
Moreover, we define 
$\mu^{m_0\dots m_i}\in GY(\ol\Om;(\MD{\times}\R)^{i+1})$  by
$$
\mu^{m_0\dots m_i}:=\delta_{((p^{m_0},z^{m_0}),(p^{m_0m_1},z^{m_0m_1}),\dots, 
(p^{m_0\dots m_i},z^{m_0\dots m_i}))}\,.
$$
Suppose that there exist $\hat e^i\in L^2(\Om;\Mnn)$  and 
$\hat\muu^i\in SGY(\{t^0_k,\dots, t^i_k\}, \ol\Om; \MD{\times}\R)$ such that for every $i\ge 0$
\begin{eqnarray}
&\displaystyle
\lim_{m_i\to\infty}\dots \lim_{m_0\to\infty} e ^{m_0\dots m_i} = \hat e^i
\,,\label{eapprx}
\\
&\displaystyle
\lim_{m_i\to\infty}\dots \lim_{m_0\to\infty}\mu^{m_0\dots m_i}=
\hat\muu^i_{t_k^0\dots t_k^i}\,,
\label{muapprx}
\end{eqnarray}
where  in the former formula
all  limits are with respect to  weak convergence in $L^2(\Om;\Mnn)$, while in the latter they are taken in the 
weak$^*$ convergence in $GY(\ol\Om;(\MD{\times}\R)^{i\!+\!1})$. 
Then for every $i\ge 1$
\begin{equation}\label{geq}
\begin{array}{c}
\displaystyle
\liminf_{m_i\to\infty}\dots \liminf_{m_0\to\infty} \big[
\QQ(e^{m_0\dots m_i}) + \HH(\tilde p^{i, m_i},\tilde z^{i, m_i}) 
+\V(z^{m_0\dots m_i})
\big] \geq
\smallskip
\\
\displaystyle
\geq 
\QQ(\hat e^{i})  +
\langle H(\xi_i-\xi_{i\!-\!1}, \theta_i-\theta_{i\!-\!1})+\{V\}(\theta_i,\eta),
 \hat\muu^i_{t^{i\!-\!1}_kt^i_k}\rangle
\,,
\end{array}
\end{equation}
where $ \hat\muu^i_{t^{i\!-\!1}_kt^i_k}$ acts on the variable 
$(x,\xi_{i\!-\!1},\theta_{i\!-\!1}, \xi_i,\theta_i,\eta)$.

Conversely, if $\hat e^i$ and $\hat\muu^i$ coincide with the function $e^i_k$ and the 
measure $\muu^i_k$ obtained in the incremental construction, then there exist two sequences $(u^{m}, e^{m}, p^{m})\in A_{reg}(\ww(0))$, $z^{m}\in L^1(\Om)$ and 
for every $i\ge 1$ two  sequences  
$(\tilde u^{i,m}, \tilde e^{i,m}, \tilde p^{i,m})
\in A_{reg}(\tilde w_k^i)$ and $\tilde z^{i,m}\in L^1(\Om)$ such that for every $i\ge 0$ \eqref{eapprx} holds with respect to strong convergence and  \eqref{muapprx} holds with respect to weak$^*$ convergence, 
while
\begin{equation}\label{recovery}
\begin{array}{c}
\displaystyle
\lim_{m_i\to\infty}\dots \lim_{m_0\to\infty} \big[
\QQ(e^{m_0\dots m_i}) + \HH(\tilde p^{i, m_i},\tilde z^{i, m_i}) 
+\V(z^{m_0\dots m_i})
\big] =
\smallskip
\\
\displaystyle
=
\QQ( e_k^{i})  +
\langle H(\xi_i-\xi_{i\!-\!1}, \theta_i-\theta_{i\!-\!1})+\{V\}(\theta_i,\eta),
 (\muu^i_k)_{t^{i\!-\!1}_kt^i_k}\rangle
\end{array}
\end{equation}
for every $i\ge 1$.
\end{proposition}

\begin{proof}
Inequality \eqref{geq} follows from \eqref{eapprx} and \eqref{muapprx} by the lower 
semicontinuity of $\QQ$ in the weak topology of $L^2(\Om;\Mnn)$ and the continuity 
of the duality product in the weak$^*$ topology of $GY(\ol\Om; (\MD{\times}\R)^2)$.

By Theorem~\ref{thm:d-app} and Lemma~\ref{lm:comp} there  exist a sequence 
$(u^{m},e^{m},p^{m})\in A_{reg}(\ww(0))$ and a sequence $z^{m}\in L^1(\Om)$ 
such that $u^{m}\wto u_0$ weakly$^*$ in $BD(\Om)$, $e^{m}\to e_0$ strongly in 
$L^2(\Om;\Mnn)$ $p^{m}\wto p_0$ weakly$^*$ in $M_b(\ol\Om;\MD)$, 
$z^{m}\wto z_0$ weakly$^*$ in $M_b(\ol\Om)$, and
 $\|(p^{m},z^{m})\|_1\to \|(p_0,z_0)\|_1$. Using
 \cite[Theorem~3]{Res} (see also \cite[Appendix]{Luc-Mod}) we obtain that
 $ \delta_{(p^{m},z^{m})}\wto \delta_{(p_0,z_0)}$ weakly$^*$ in
 $GY(\ol\Om;\MD{\times}\R)$.

For every $i\ge1$, by definition of $A_k^i$ there exist a sequence 
$(\tilde u^{i,m}, \tilde e^{i,m}, \tilde p^{i,m})\in A_{reg}(\tilde w_k^i)$ and a sequence 
$\tilde z^{i,m}\in L^1(\Om)$ such that
\begin{eqnarray}
& u_k^{i\!-\!1}+\tilde u^{i,m} \wto u_k^i \quad \text{weakly}^* \text{ in } BD(\Om)\,,\nonumber
\\
& e_k^{i\!-\!1}+\tilde e^{i,m} \to e_k^i \quad \text{strongly in } L^2(\Om;\Mnn)\,,\label{strong}
\\
&
\T^i_{(\tilde p^{i,m}, \tilde z^{i,m})}((\muu^{i\!-\!1}_k)_{t_k^0\dots t_k^{i\!-\!1}}) \wto (\muu^i_k)_{t_k^0\dots t_k^i} \quad \text{weakly}^* \text{ in } GY(\ol\Om; (\MD{\times}\R)^{i\!+\!1})\,.\label{aikdef}
\end{eqnarray}
Condition \eqref{eapprx} is trivially satisfied thanks to \eqref{strong}.
To prove \eqref{muapprx} we observe that for every $i\ge 1$
$$
\mu^{m_0\dots m_i}=\T^i_{(\tilde p^{i,m_i}, \tilde z^{i,m_i})}(\mu^{m_0\dots 
m_{i\!-\!1}})\,.
$$
We now proceed by induction on $i$. Equality \eqref{muapprx} for $i=0$ is true by construction.
Assume that \eqref{muapprx} holds for $i-1$. Then by Lemma~\ref{trlem}
$$
\lim_{m_{i\!-\!1}\to\infty}\dots \lim_{m_0\to\infty}\mu^{m_0\dots m_i}=
\T^i_{(\tilde p^{i,m_i}, \tilde z^{i,m_i})}((\muu_k^{i\!-\!1})_{t^0_k\dots t_k^{i\!-\!1}})
$$
The conclusion for $i$ follows from \eqref{aikdef}.
\end{proof}

\subsection{Further minimality properties}
We now prove that the solutions of the incremental problems satisfy some
additional minimality conditions.

\begin{lemma}\label{secondproblem}
For every $i$ and $ k$ and every  $t>t^i_k$ we have
\begin{equation}\label{secpb}
\begin{array}{c}
\QQ(e^i_k) + \langle \{V\}(\theta_i,\eta), (\muu^i_k)_{t^i_k}(x, \xi_i,\theta_i,\eta)\rangle \le
\smallskip
\\
\le \QQ(e) + \langle H(\xi - \xi_i,\theta- \theta_i), \nuu_{t^i_k  t}
(x,  \xi_i,\theta_i,\xi,\theta,\eta)\rangle+
\langle \{V\}(\theta,\eta),  \nuu_t(x, \xi, \theta,\eta)\rangle
\end{array}
\end{equation}
for every $(u,e, \nuu)\in BD(\Om){\times} L^2(\Om;\Mnn){\times} 
SGY(\{{t^0_k,\dots,t^i_k,t}\},\ol\Om;\MD{\times}\R)$ such that
\begin{eqnarray}
&\nuu_{t^0_k\dots t^i_k}=(\muu^i_k)_{t^0_k\dots t^i_k}\,,
\label{compatib2*}
\\
& (u,e,p)\in A(\ww(t^i_k))\,,
\label{AY*}
\end{eqnarray}
where $(p,z):=\bary(\nuu_t)$.
\end{lemma}

\begin{proof}
Let us fix $(u, e, \nuu)$ as in the statement of the lemma, 
and let $\tilde\nuu$ be the system in $SGY(\{t^0_k\dots t^i_k\},\ol\Om;\MD{\times}\R)$
associated with 
$\pi^{t^0_k\dots t^i_k  t}_{t^0_k\dots t^{i\!-\!1}_k t}(\nuu_{t^0_k\dots t^i_k t})\in 
GY(\ol\Om;(\MD{\times}\R)^{i+1})$ according to
\cite[Remark 7.9]{DM-DeS-Mor-Mor-1}.
Since $\muu^i_k$ satisfies (\ref{compatib2}), by (\ref{compatib2*}) and (\ref{AY*}) the triplet  $( u, e,\tilde \nuu)$ satisfies (\ref{compatib2}) and (\ref{AY}), hence 
$(u, e,\tilde \nuu)$ belongs to the set $B^i_k$ defined in Remark~\ref{rmk:rmk}. 
By minimality we have
$$
\begin{array}{c}
\QQ(e^i_k)+\langle H(\xi_i -\xi_{i\!-\!1}, \theta_i -\theta_{i\!-\!1}),
(\muu^i_k)_{t^{i\!-\!1}_k t^i_k}(x, \xi_{i\!-\!1}, \theta_{i\!-\!1}, \xi_i,\theta_i,\eta)\rangle
+{}
\smallskip
\\
{}+ \langle \{V\}(\theta_i,\eta), 
(\muu^i_k)_{t^i_k}(x, \xi_i,\theta_i,\eta)\rangle \le
\smallskip
\\
\le \QQ( e) + \langle H(\xi_i -\xi_{i\!-\!1}, \theta_i -\theta_{i\!-\!1}),
\tilde \nuu_{t^{i\!-\!1}_k t^i_k}(x, \xi_{i\!-\!1}, \theta_{i\!-\!1}, \xi_i,\theta_i,\eta)\rangle
+{}
\smallskip
\\
{}+ \langle \{V\}(\theta_i,\eta), 
\tilde\nuu_{t^i_k}(x, \xi_i,\theta_i,\eta)\rangle\,.
\end{array}
$$
Since $(\muu^i_k)_{t^{i\!-\!1}_k t^i_k}=\nuu_{t^{i\!-\!1}_k t^i_k}$, $\tilde \nuu_{t^{i\!-\!1}_k t^i_k}= \nuu_{t^{i\!-\!1}_k t}$, and $\tilde\nuu_{t^i_k}=\nuu_t$, we get
$$
\begin{array}{c}
\QQ(e^i_k)+\langle H(\xi_i -\xi_{i\!-\!1}, \theta_i -\theta_{i\!-\!1}),
 \nuu_{t^{i\!-\!1}_k t^i_k}(x, \xi_{i\!-\!1}, \theta_{i\!-\!1}, \xi_i,\theta_i,\eta)\rangle
+{}
\smallskip
\\
{}+ \langle \{V\}(\theta_i,\eta), 
(\muu^i_k)_{t^i_k}(x, \xi_i,\theta_i,\eta)\rangle \le
\smallskip
\\
\le \QQ(e) + \langle H(\xi -\xi_{i\!-\!1}, \theta -\theta_{i\!-\!1}),
 \nuu_{t^{i\!-\!1}_k t}(x, \xi_{i\!-\!1}, \theta_{i\!-\!1}, \xi, \theta,\eta)\rangle
+{}
\smallskip
\\
{}+ \langle \{V\}(\theta,\eta), 
 \nuu_t(x, \xi,\theta,\eta)\rangle\,.
\end{array}
$$
{}From the compatibility condition~(7.2) of \cite{DM-DeS-Mor-Mor-1} we obtain
$$
\begin{array}{c}
\QQ(e^i_k)+\langle H(\xi_i -\xi_{i\!-\!1}, \theta_i -\theta_{i\!-\!1}),
\nuu_{t^{i\!-\!1}_k t^i_k  t}(x, \xi_{i\!-\!1}, \theta_{i\!-\!1}, \xi_i, \theta_i, \xi, \theta, \eta)\rangle
+{}
\smallskip
\\
{}+ \langle \{V\}(\theta_i,\eta), 
(\muu^i_k)_{t^i_k}(x, \xi_i,\theta_i,\eta)\rangle \le
\smallskip
\\
\le \QQ(e) + \langle H(\xi -\xi_{i\!-\!1}, \theta -\theta_{i\!-\!1}),
\nuu_{t^{i\!-\!1}_k  t^i_k t}(x, \xi_{i\!-\!1}, \theta_{i\!-\!1}, 
\xi_i, \theta_i, \xi, \theta,\eta)\rangle
+{}
\smallskip
\\
{}+ \langle \{V\}(\theta,\eta), 
 \nuu_t(x,\xi,\theta,\eta)\rangle\,.
\end{array}
$$
By the triangle inequality (\ref{triangle0}) we deduce that
$$
\begin{array}{c}
\QQ(e^i_k)+
\langle \{V\}(\theta_i,\eta), 
(\muu^i_k)_{t^i_k}(x, \xi_i,\theta_i,\eta)\rangle \le
\smallskip
\\
\le \QQ(e) + \langle H(\xi -\xi_i, \theta -\theta_i),
\nuu_{t^{i\!-\!1}_k  t^i_k  t}(x, \xi_{i\!-\!1}, \theta_{i\!-\!1}, \xi_i, \theta_i, \xi, \theta,\eta)\rangle
+{}
\smallskip
\\
{}+ \langle \{V\}(\theta,\eta), 
\nuu_t(x, \xi,\theta,\eta)\rangle\,,
\end{array}
$$
which gives (\ref{secpb}) by  the compatibility condition~(7.2) of 
\cite{DM-DeS-Mor-Mor-1}.
\end{proof}

{}For every $i$ and $k$ we set $\sigma_k^i:=\C e_k^i$ and for every $t\in[0,+\infty)$ we
consider the piecewise constant interpolations defined by
\begin{equation}\label{ukt2}
\begin{array}{c} 
\uu_k(t):=\uki\,, \quad \ee_k(t):=\eki\,, \quad\pp_k(t):=p^i_k, \quad\zz_k(t):=z^i_k,
\smallskip 
\\
\sigmaa_k(t):=\sigma_k^i\,, 
\quad \ww_k(t):=w_k^i \,, \quad [t]_k:=\tki\,,
\vspace{.1cm}\\
\end{array}
\end{equation}
for $t\in[\tki,t_k^{i\!+\!1})$. We define also $\muu_k$ as the unique
system in $SGY([0,+\infty),\ol\Om;\MD{\times}\R)$ whose restrictions
to the time intervals $[0,t^i_k]$ coincide with the
piecewise constant interpolations of $\muu^i_k\in SGY(\{{t^0_k,\dots,t^i_k}\},\ol\Om;\MD{\times}\R)$ introduced in \cite[Definition 7.10]{DM-DeS-Mor-Mor-1}.
As $(\pp_k(t),\zz_k(t))=\bary((\muu_k)^{}_t)$, we have also
\begin{equation}\label{interp10}
(\uu_k(t),\ee_k(t),\pp_k(t))\in A(\ww_k(t))
\end{equation}
for every $t\in[0,+\infty)$.

\begin{lemma}\label{min100}
Let $t, \hat t\in[0,+\infty)$ with $t<\hat t$. Then
\begin{equation}\label{secpb2}
\begin{array}{c}
\QQ(\ee_k(t)) + \langle \{V\}(\theta,\eta), (\muu_k)^{}_t(x, \xi,\theta,\eta)\rangle \le
\smallskip
\\
\le \QQ(\ee_k(\hat t)-E\ww_k(\hat t)+E\ww_k(t)) +
\langle H(\hat \xi - \xi, \hat \theta- \theta), (\muu_k)_{t \hat t}
(x,  \xi,\theta,\hat \xi,\hat \theta,\eta)\rangle+{}
\\
{}+
\langle \{V\}(\hat \theta,\eta), (\muu_k)_{\hat t}(x, \hat \xi, \hat \theta,\eta)\rangle
\end{array}
\end{equation}
for every~$k$.
\end{lemma}

\begin{proof}
Let $u:= \uu_k( \hat t)-\ww_k(\hat  t)+\ww_k(t)$ and 
$e:=\ee_k( \hat t)-E\ww_k(\hat  t)+E\ww_k(t)$, and let $i$ be the greatest index such that $t^i_k\le t$. Since the triplet $(u,e, (\muu_k)_{t^0_k\dots t^i_k \hat t})$ satisfies (\ref{compatib2*}) and (\ref{AY*}), the result follows from Lemma~\ref{secondproblem}.
\end{proof}

\begin{lemma}\label{min101}
Let $t\in[0,+\infty)$. Then
\begin{equation}\label{secpb3}
\begin{array}{c}
\QQ(\ee_k(t)) + \langle \{V\}(\theta,\eta), (\muu_k)^{}_t(x, \xi,\theta,\eta)\rangle \le
\smallskip
\\
\le \QQ(\ee_k(t)+\tilde e) + \HH(\tilde p, \tilde z) +
\langle \{V\}(\theta+\eta \tilde z(x),\eta), (\muu_k)^{}_t(x, \xi, \theta,\eta)\rangle
\end{array}
\end{equation}
for every $(\tilde u,\tilde e,\tilde p)\in A_{reg}(0)$ and every $\tilde z\in L^1(\Om)$.
\end{lemma}

\begin{proof}
Let us fix $(\tilde u,\tilde e,\tilde p)\in A_{reg}(0)$ and $\tilde z\in L^1(\Om)$.
Let $i$ be the greatest index such that $t^i_k\le t$ and let $\hat t>t^i_k$.
We set $\hat u:= u^i_k+ \tilde u=\uu_k(t)+ \tilde u$, $\hat e:= e^i_k+ \tilde e= \ee_k(t)+ \tilde e$, 
$\nu_{t^0_k\dots t^i_k \hat t}:=\T^{i\!+\!1}_{(\tilde p,\tilde z)}((\muu^i_k)_{t^0_k\dots t^i_k})= 
\T^{i\!+\!1}_{(\tilde p,\tilde z)}((\muu_k)_{t^0_k\dots t^i_k})$, and we define $\nuu\in SGY(\{t^0_k,\dots, t^i_k ,\hat t\},\ol\Om;\MD{\times}\R)$
as the system associated with $\nu_{t^0_k\dots t^i_k \hat t}$ according to \cite[Remark~7.9]{DM-DeS-Mor-Mor-1}.
Since the triplet $(\hat u,\hat e, \nuu)$ satisfies (\ref{compatib2*}) and (\ref{AY*}), the
conclusion follows from Lemma~\ref{secondproblem}.
\end{proof}

\subsection{Energy estimates}
We now prove some energy estimates for the solutions of the incremental minimum problems.

\begin{lemma}\label{lm:141}
For every $T>0$ there exists a sequence $\omega_k^T\to 0^+$ such that 
\begin{equation}\label{inc-ineq-10}
\begin{array}{c}
\displaystyle
\vphantom{\int_0^t}
\QQ(\ee_k(t))+\D_H(\muu_k;0,t)+
\langle \{V\}(\theta,\eta), (\muu_k)_t(x,\xi,\theta,\eta) \rangle
\le
\\
\displaystyle
\le \QQ(e_0)+\V(z_0)
+\int_0^{[t]_k}\langle\sigmaa_k(t), E\dot \ww(t)\rangle\, dt +\omega_k^T
\end{array}
\end{equation}
for every $k$ and every $t\in[0,T]$.
\end{lemma}

\begin{proof}
Let us fix $T>0$ and $t\in[0,T]$. Arguing as in \cite[Remark~8.5]{DM-DeS-Mor-Mor-1}, we can prove that
$$
\D_H(\muu_k;0,t)=\sum_{r=1}^i
 \langle H(\xi_r-\xi_{r\!-\!1}, \theta_r-\theta_{r\!-\!1}),
 (\muu_k)_{t^{r\!-\!1}_k t^r_k}(x,\xi_{r\!-\!1},\theta_{r\!-\!1},\xi_r,\theta_r,\eta)\rangle\,,
$$
where $i$ is the largest integer such that $\tki\le t$. 
Therefore, using the definition of piecewise constant interpolation of
a generalized Young measures,  
we have to show that there exists a sequence $\omega_k^T\to 0^+$ such that
\begin{equation}\label{inc-ineq-102}
\begin{array}{c}
\displaystyle
\vphantom{\int_0^t}
\QQ(e_k^i)+
\sum_{r=1}^i  \langle H(\xi_r-\xi_{r\!-\!1}, \theta_r-\theta_{r\!-\!1}),
 (\muu^r_k)_{t^{r\!-\!1}_k t^r_k}(x,\xi_{r\!-\!1},\theta_{r\!-\!1},\xi_r,\theta_r,\eta)\rangle
+{}
\\
{}+
\langle \{V\}(\theta,\eta), (\muu^i_k)_{t^i_k}(x,\xi,\theta,\eta) \rangle
\le
\\
\displaystyle
\le \QQ(e_0)+\V(z_0)
+\int_0^{\tki}\langle\sigmaa_k(t), E\dot \ww(t)\rangle\, dt +\omega_k^T
\end{array}
\end{equation}
for every $k$ and every $i$ with $\tki\le T$.

Fix an integer $r$ with $1\le r\le i$ and let 
$\pi^{(r)}\colon \ol\Om{\times}(\MD{\times}\R)^r{\times}\R \to
\ol\Om{\times}(\MD{\times}\R)^{r+1}\allowbreak{\times}\R$ be the map defined by 
$$
\pi^{(r)}(x,(\xi_0,\theta_0),\dots, (\xi_{r\!-\!1},\theta_{r\!-\!1}),\eta)=(x,(\xi_0,\theta_0),\dots, (\xi_{r\!-\!1},\theta_{r\!-\!1}),(\xi_{r\!-\!1}, 
\theta_{r\!-\!1}),\eta)\,.
$$
Let $\hat u:=u_k^{r-1}-w_k^{r-1}+w_k^r$, $\hat e:=e_k^{r-1}-Ew_k^{r-1}+Ew_k^r$, and
$\hat\mu_{t^0_k\dots t^r_k}:= \pi^{(r)}((\muu^{r\!-\!1}_k)_{t^0_k\dots t^{r\!-\!1}_k})$. Let $\hat \muu \in SGY(\{t^0_k,\dots, t^r_k\},\ol\Om;\MD{\times}\R)$ be the system associated with $\hat\mu_{t^0_k\dots t^r_k}$ by \cite[Remark~7.9]{DM-DeS-Mor-Mor-1}.
It is easy to check that
$$
\hat \muu_{t^0_k\dots t^{r\!-\!1}_k}=(\muu^{r\!-\!1}_k)_{t^0_k\dots t^{r\!-\!1}_k}\,.
$$
Let $(\hat p,\hat z):= \bary(\hat \muu_{t^r_k})$. Since $\hat p=p_k^{r-1}$, we find that
$(\hat u,\hat e,\hat p)\in A(w_k^r)$, hence $(\hat u,\hat e,\hat\muu)$ belongs to the class
$B_k^r$, introduced in Remark~\ref{rmk:rmk}.
By minimality we have
\begin{equation}\label{b1001}
\begin{array}{c}
\QQ(e_k^r)+\langle H(\xi_r -\xi_{r\!-\!1}, \theta_r -\theta_{r\!-\!1})+\{V\}(\theta_r,\eta), (\muu^r_k)_{t^{r\!-\!1}_k t^r_k}(x, \xi_{r\!-\!1}, \theta_{r\!-\!1}, \xi_r,\theta_r,\eta)\rangle  \le \vspace{.1cm}\\
\le 
\QQ(\hat e) + \langle H(\xi_r -\xi_{r\!-\!1}, \theta_r -\theta_{r\!-\!1})+\{V\}(\theta_r,\eta), \hat \muu_{t^{r\!-\!1}_k t^r_k}(x, \xi_{r\!-\!1}, \theta_{r\!-\!1}, \xi_r,\theta_r,\eta)\rangle\,.
\end{array}
\end{equation}
As $\hat \muu_{t^{r\!-\!1}_k t^r_k}=
\big(\pi^{t^0_k\dots t^r_k}_{t^{r\!-\!1}_k t^r_k}\circ \pi^{(r)}\big)
 \big((\muu^{r\!-\!1}_k)_{t^0_k\dots t^{r\!-\!1}_k}\big)$ and
$$
\big(\pi^{t^0_k\dots t^r_k}_{t^{r\!-\!1}_k t^r_k}\circ \pi^{(r)}\big)
(x,(\xi_0,\theta_0),\dots, (\xi_{r\!-\!1},\theta_{r\!-\!1}),\eta)=
(x,(\xi_{r\!-\!1},\theta_{r\!-\!1}),(\xi_{r\!-\!1}, 
\theta_{r\!-\!1}),\eta)\,,
$$
we have
$$
\langle H(\xi_r -\xi_{r\!-\!1}, \theta_r -\theta_{r\!-\!1}), \hat \muu_{t^{r\!-\!1}_k t^r_k}(x, \xi_{r\!-\!1}, \theta_{r\!-\!1}, \xi_r,\theta_r,\eta)\rangle=0
$$
and
$$
\langle \{V\}(\theta_r,\eta), \hat \muu_{t^{r\!-\!1}_k t^r_k}(x, \xi_{r\!-\!1}, \theta_{r\!-\!1}, \xi_r,\theta_r,\eta)\rangle = \langle \{V\}(\theta_{r\!-\!1},\eta), (\muu^{r\!-\!1}_k)_{t^{r\!-\!1}_k}(x, \xi_{r\!-\!1}, \theta_{r\!-\!1},\eta)\rangle\,.
$$
Therefore (\ref{b1001}) gives, thanks to the compatibility condition~(7.2) of
\cite{DM-DeS-Mor-Mor-1},
\begin{equation}\label{b1001*}
\begin{array}{c}
\QQ(e_k^r)+\langle H(\xi_r -\xi_{r\!-\!1}, \theta_r -\theta_{r\!-\!1}), (\muu^r_k)_{t^{r\!-\!1}_k t^r_k}(x, \xi_{r\!-\!1}, \theta_{r\!-\!1}, \xi_r,\theta_r,\eta)\rangle +
\vspace{.1cm}\\
{}+ \langle \{V\}(\theta_r,\eta), (\muu^r_k)_{ t^r_k}(x, \xi_r,\theta_r,\eta)\rangle  \le \vspace{.1cm}\\
\le 
\QQ(e_k^{r-1}+Ew_k^r -Ew_k^{r-1})
 + \langle \{V\}(\theta_{r\!-\!1},\eta), (\muu^{r\!-\!1}_k)_{t^{r\!-\!1}_k}(x, \xi_{r\!-\!1}, \theta_{r\!-\!1},\eta)\rangle
 \,,
\end{array}
\end{equation}
where the quadratic form in the right-hand side can be developed as
\begin{equation}\label{b1002}
\QQ(e_k^{r-1}+Ew_k^r-Ew_k^{r-1})
=\QQ(e_k^{r-1})+\langle \sigma_k^{r-1} , Ew_k^r- 
Ew_k^{r-1} \rangle 
+ \QQ(Ew_k^r-Ew_k^{r-1})\,.
\end{equation}
{}From the absolute continuity of $w$ with respect to $t$ we obtain
$$
w_k^r-w_k^{r-1}=\int_{t_k^{r-1}}^{t_k^r}\dot \ww(t)\, dt\,,
$$
where we use a Bochner integral of a function with values in $H^1(\Om;\Rn)$.
This implies that
\begin{equation}\label{b1003}
Ew_k^r- Ew_k^{r-1}=\int_{t_k^{r-1}}^{t_k^r} E\dot \ww(t)\, dt\, ,
\end{equation}
where we use a Bochner integral of a function with values in $L^2(\Om;\Mnn)$.
By (\ref{boundsC}) and (\ref{b1003}) we get
\begin{equation}\label{b1003.5}
\QQ(Ew_k^r-Ew_k^{r-1})\le \beta_{\C}\Big(\int_{t_k^{r-1}}^{t_k^r} 
\|E\dot \ww(t)\|_2\, dt\Big)^2. 
\end{equation}
By (\ref{b1001*})--(\ref{b1003.5}) we obtain
\begin{equation}\label{b1004}
\begin{array}{c}
\displaystyle\vphantom{\int_{t_k^{r-1}}^{t_k^r}}
\QQ(e_k^r) +\langle H(\xi_r -\xi_{r\!-\!1}, \theta_r -\theta_{r\!-\!1}), (\muu^r_k)_{t^{r\!-\!1}_k t^r_k}(x, \xi_{r\!-\!1}, \theta_{r\!-\!1}, \xi_r,\theta_r,\eta)\rangle +
\vspace{.1cm}
\\
\displaystyle\vphantom{\int_{t_k^{r-1}}^{t_k^r}}
{}+ \langle \{V\}(\theta_r,\eta), (\muu^r_k)_{ t^r_k}(x, \xi_r,\theta_r,\eta)\rangle  \le \vspace{.1cm}
\\
\displaystyle\vphantom{\int_{t_k^{r-1}}^{t_k^r}}
\le \QQ(e_k^{r-1})+ \langle \{V\}(\theta_{r\!-\!1},\eta),
 (\muu^{r\!-\!1}_k)_{t^{r\!-\!1}_k}(x, \xi_{r\!-\!1}, \theta_{r\!-\!1},\eta)\rangle
+ {} \vspace{.1cm}
\\
\displaystyle
{}+\int_{t_k^{r-1}}^{t_k^r} \langle \sigma_k^{r-1} , E\dot \ww(t)\rangle \, dt +\beta_{\C}\Big(\int_{t_k^{r-1}}^{t_k^r}\|E\dot \ww(t)\|_2\, dt 
\Big)^2
 \le \vspace{.1cm} 
\\
\displaystyle\vphantom{\int_{t_k^{r-1}}^{t_k^r}}
\le \QQ(e_k^{r-1})+ \langle \{V\}(\theta_{r\!-\!1},\eta),
 (\muu^{r\!-\!1}_k)_{t^{r\!-\!1}_k}(x, \xi_{r\!-\!1}, \theta_{r\!-\!1},\eta)\rangle
+ {} \vspace{.1cm}
\\
\displaystyle
{}+\int_{t_k^{r-1}}^{t_k^r} \langle \sigma_k^{r-1} , E\dot \ww(t)\rangle \, dt 
+\rho_k^T\int_{t_k^{r-1}}^{t_k^r}\|E\dot \ww(t)\|_2\, dt\,, 
\end{array}
\end{equation}
where
$$
\rho_k^T:=\max_{t_k^{r}\le T}
\beta_{\C}\int_{t_k^{r-1}}^{t_k^r}\|E\dot 
\ww(t)\|_2\, dt \ \to \ 0
$$
by the absolute continuity of the integral.
Iterating now inequality (\ref{b1004}) for $1\le r\le i$, we get 
(\ref{inc-ineq-102}) 
with $\omega_k^T:=\rho_k^T\int_0^T \|E\dot \ww(t)\|_2\, dt$.
\end{proof}

\subsection{Proof of the main theorem}
Let us fix a sequence of subdivisions $(\tki)_{i\geq0}$ of the 
half-line $[0,+\infty)$ satisfying (\ref{subdiv5}) and (\ref{fine5}). For every $k$ let 
$(\uki,\eki,\muu^i_k)$,
$i=1,\ldots,k$, be defined inductively as minimizers of the functional (\ref{functinc}) on 
the sets $A^i_k$, with
$(u_k^0,e_k^0, \muu_k^0)=(u_0,e_0,\delta_{(p_0,z_0)})$, and let 
$\uu_k(t)$, $\ee_k(t)$, $\sigmaa_k(t)$, $\ww_k(t)$, and $[t]_k$ be defined by 
(\ref{ukt2}) and 
let $\muu_k$ be the unique system in $SGY([0,+\infty);\ol\Om;\MD{\times}\R)$ whose restrictions
to the intervals $[0,\tki]$ coincide with the
piecewise constant interpolations of 
$\muu^i_k$ (see \cite[Definition~7.10]{DM-DeS-Mor-Mor-1}).
Using Lemma~\ref{trlem} and the definition of $A^i_k$ we can prove by induction on $i$
that $(\uu_k,\ee_k,\muu_k)\in AY(\{t_k^0,t^1_k,\dots,t_k^i\},\ww_k)$
for every $i$ and $k$. This implies that
\begin{equation}\label{b1035}
(\uu_k,\ee_k,\muu_k)\in AY([0,+\infty),\ww_k)
\end{equation}
for every $k$.

Let us prove that for every $T>0$ there exists a constant $C_T$, independent of $k$, such that
\begin{equation}\label{b1005}
 \sup_{t\in[0,T]}\|\ee_k(t)\|_2\leq C_T\,, \qquad 
{\rm Var}(\muu_k;0,T)\leq C_T\,.
\end{equation}
By (\ref{gammaM}) we have
\begin{equation}\label{bdiss10}
\begin{array}{c}
\D_H(\muu_k;0,t)+ \langle \{V\}(\theta,\eta), (\muu_k)^{}_t(x, \xi,\theta,\eta)\rangle \ge \smallskip
\\
\ge \langle H(\xi -\xi_0, \theta -\theta_0)+ \{V\}(\theta,\eta), 
(\muu_k)^{}_{0 t}(x, \xi_0, \theta_0, \xi,\theta,\eta)\rangle  \ge 
\smallskip
\\
\ge \langle C^K_V |\xi-\xi_0|+  C^K_V|\theta-\theta_0|+\{V\}(\theta_0,\eta),
(\muu_k)^{}_{0 t}(x, \xi_0, \theta_0, \xi,\theta,\eta)\rangle=
\smallskip
\\
= C^K_V \langle |\xi-\xi_0|+ |\theta-\theta_0|,
(\muu_k)^{}_{0 t}(x, \xi_0, \theta_0, \xi,\theta,\eta)\rangle + \V(z_0)
\,,
\end{array}
\end{equation}
where the last equality follows from the fact that $(\muu_k)^{}_0=\delta_{(p_0,z_0)}$.
{}From (\ref{boundsC}), (\ref{normC}), (\ref{inc-ineq-10}), and (\ref{bdiss10}) we deduce that
$$
\alpha_{\C}\|\ee_k(t)\|^2_2 \le 
\beta_{\C}\|e_0\|^2_2  + 2\beta_\C\, \sup_{t\in[0,T]}\|\ee_k(t)\|_2\int_0^T  \|E\dot \ww(t)\|_2 \, dt +\omega_k^T
$$
for every $k$ and every $t\in[0,T]$. The first estimate in (\ref{b1005}) can be obtained now by using the Cauchy inequality.

By (\ref{inc-ineq-10}) and the first inequality in (\ref{b1005}) we have that 
\begin{equation}\label{10024}
 \D_H(\muu_k;0,t)+ \langle \{V\}(\theta,\eta), (\muu_k)^{}_t(x, \xi,\theta,\eta)\rangle
\end{equation}
is bounded uniformly with respect to $k$ and $t\in[0,T]$. By (\ref{bdiss10}) this implies the boundedness of
\begin{equation}\label{10024*}
 \langle |\xi-\xi_0|+ |\theta-\theta_0|,
(\muu_k)^{}_{0 t}(x, \xi_0, \theta_0, \xi,\theta,\eta)\rangle \,.
\end{equation}
By the compatibility condition~(7.2) of \cite{DM-DeS-Mor-Mor-1} and by the equality
$(\muu_k)^{}_0=\delta_{(p_0,z_0)}$ we have
$$
\langle |\xi_0|+ |\theta_0|,
(\muu_k)^{}_{0 t}(x, \xi_0, \theta_0, \xi,\theta,\eta)\rangle=\|p_0\|_1+\|z_0\|_1\,,
$$
which, together with the boundedness of (\ref{10024*}), gives that
$\langle |\xi|+ |\theta|,
(\muu_k)^{}_{t}(x, \xi,\theta,\eta)\rangle$ is bounded.
This implies that
$\langle \{V\}(\theta,\eta), (\muu_k)^{}_t(x, \xi,\theta,\eta)\rangle$ is bounded too, so that 
(\ref{boundsH}) and the boundedness of (\ref{10024}) yield the second estimate in (\ref{b1005}).

By the Helly Theorem for compatible systems of generalized Young measures proved in
\cite[Theorem~8.10]{DM-DeS-Mor-Mor-1} 
there exist a subsequence, still denoted $\muu_k$, a set $\Theta\subset[0,+\infty)$,
containing $0$ and
with $[0,+\infty)\setmeno\Theta$ at most countable, and a left continuous
$\muu\!\in \!SGY\!([0,+\infty),\ol\Om;\MD{\times}\R)\!$, with
\begin{equation}\label{boundvar10}
 \sup_{t\in[0,T]} \|\muu_t\|_*<+\infty\,,\qquad {\rm Var}(\muu;0,T)<+\infty
\end{equation}
for every $T>0$, such that
\begin{equation}\label{convergence10}
(\muu_k)^{}_{t_1\dots t_m}\wto \muu^{}_{t_1\dots t_m}\quad \hbox{weakly}^* \hbox{ in }GY(\ol\Om;(\MD{\times}\R)^m)
\end{equation}
for every finite sequence $t_1,\dots, t_m$ in $\Theta$ with ${t_1<\dots <t_m}$. 

Let $\pp_k(t)\in M_b(\ol\Om;\MD)$, $\zz_k(t)\in M_b(\ol\Om)$, $\pp(t)\in M_b(\ol\Om;\MD)$, and $\zz(t)\in M_b(\ol\Om)$ be the measures defined by
\begin{equation}\label{pzbarnu10}
(\pp_k(t),\zz_k(t)):=\bary((\muu_k)^{}_t)\qquad\hbox{and}\qquad 
(\pp(t),\zz(t)):=\bary(\muu_t)\,.
\end{equation}
By (\ref{convergence10}) and by \cite[Remark~6.4]{DM-DeS-Mor-Mor-1} we have
\begin{equation}\label{pktop}
\pp_k(t)\wto \pp(t)\quad\hbox{weakly}^* \hbox{ in } M_b(\ol\Om;\MD)
\end{equation}
for every $t\in\Theta$.

By Corollary~\ref{corbarymin} the sequence $(\uu_k(t),\ee_k(t),\pp_k(t))$ 
coincides with the discrete-time approximation of the quasistatic evolution 
corresponding to the function $\xi\mapsto H_{\rm eff}(\xi,0)$ according to \cite[Definition~4.2]{DM-DeS-Mor}. 
Using \cite[Theorems~4.5, 4.8, 5.2]{DM-DeS-Mor} we obtain that there exist 
a subsequence, still denoted $(\uu_k,\ee_k,\pp_k)$, a continuous function
$t\mapsto (\uu(t),\ee(t))$ from $[0,+\infty)$ into $BD(\Om){\times}L^2(\Om;\Mnn)$
and an extension of $t\mapsto \pp(t)$ to $[0,+\infty)$, still denoted by the same symbol, such that 
$t\mapsto (\uu(t), \ee(t), \pp(t))$ is a quasistatic evolution of the problem corresponding to the function 
$\xi\mapsto H_{\rm eff}(\xi,0)$, and
\begin{eqnarray}
& \ee_k(t) \to \ee(t) \quad \hbox{strongly in } L^2(\Om;\Mnn)\,,
\\
& \uu_k(t) \wto \uu(t) \quad \hbox{weakly}^*\hbox{ in } BD(\Om)\,,
\end{eqnarray}
for every $t\in [0,+\infty)$.

By Remark~\ref{rmk:closAY} and by \eqref{b1035} the triple $(\uu,\ee,\muu)$ belongs to $AY(\Theta,\ww)$.
By the left continuity of $\uu$, $\ee$,
$\muu$ we have also $(\uu,\ee,\muu)\in AY([0,+\infty),\ww)$.

Let us fix $(\tilde u,\tilde e,\tilde p)\in A_{reg}(0)$ and $\tilde z\in L^1(\Om)$.
Passing to the limit in (\ref{secpb3}) thanks to Lemma~\ref{trlem}
we obtain that (ev1) is satisfied for every $t\in \Theta$. By left continuity the same inequality holds for every $t\in[0,+\infty)$.

By (\ref{b1005}) and by the weak$^*$ lower semicontinuity of the dissipation 
we can pass to the limit in 
(\ref{inc-ineq-10}) and we obtain
\begin{equation}\label{en-ineq-20}
\begin{array}{c}
\displaystyle
\vphantom{\int_0^t}
\QQ(\ee(T))+\D_{\!H}(\muu;0,T)+
\langle \{V\}(\theta,\eta), \muu_T(x,\xi,\theta,\eta) \rangle
\le
\\
\displaystyle
\le \QQ(e_0)+\V(z_0)
+\int_0^T \langle\sigmaa(t), E\dot \ww(t)\rangle\, dt\,,
\end{array}
\end{equation}
 for every $T\in \Theta$. By  left continuity the same inequality holds for every 
 $T\in[0,+\infty)$.

Passing to the limit in (\ref{secpb2}), we obtain
\begin{equation}\label{secpb12}
\begin{array}{c}
 \QQ(\ee(t)) + \langle \{V\}(\theta,\eta), \muu_t(x, \xi,\theta,\eta)\rangle \le
\smallskip
\\
\le \QQ(\ee(\hat t))-\langle\sigmaa(\hat t), E\ww(\hat t)-E\ww(t)\rangle 
+ \QQ(E\ww(\hat t)-E\ww(t))+{}
\smallskip
\\
{}+ \langle H(\hat\xi - \xi, \hat\theta- \theta), \muu_{t \hat t}
(x,  \xi,\theta,\hat\xi,\hat\theta,\eta)\rangle+
\langle \{V\}(\hat\theta,\eta), \muu_{\hat t}(x, \hat \xi, \hat \theta,\eta)\rangle
\end{array}
\end{equation}
for every $t,\hat t\in\Theta$ with $t<\hat t$. By left continuity the same inequality holds for every $t,\hat t\in[0,+\infty)$ with $t<\hat t$.

Using this inequality, we want to prove that
\begin{equation}\label{en-ineq-21}
\begin{array}{c}
\displaystyle
\vphantom{\int_0^t}
\QQ(\ee(T))+\D_{\!H}(\muu;0,T)+
\langle \{V\}(\theta,\eta), \muu_T(x,\xi,\theta,\eta) \rangle
\ge
\\
\displaystyle
\ge \QQ(e_0)+\V(z_0)
+\int_0^T \langle\sigmaa(t), E\dot \ww(t)\rangle\, dt\,,
\end{array}
\end{equation}
for every $T\in(0,+\infty)$.

Let us fix $T\in(0,+\infty)$ and a sequence of subdivisions 
$(s_k^i)_{0\le i\le k}$ of $[0,T]$ with
\begin{equation}\label{sh101}
\begin{array}{c}
0=s_k^0<s_k^1<\dots<s_k^{k-1}<s_k^{k}=T\,, \smallskip
\\
\displaystyle \lim_{k\to\infty}\,\max_{1\le i\le k}(s_k^i-s_k^{i-1})=0\,.
\end{array}
\end{equation} 
For every $i=1,\dots,k$ we apply (\ref{secpb12}) with times $s_k^{i\!-\!1}$ and $s_k^i$, and we obtain
\begin{eqnarray}
&\displaystyle
\QQ(\ee(s_k^{i\!-\!1})) + \langle \{V\}(\theta_{i\!-\!1},\eta),
 \muu_{s_k^{i\!-\!1}}(x, \xi_{i\!-\!1},
\theta_{i\!-\!1},\eta)\rangle \le
\smallskip
\nonumber
\\
&\displaystyle
\vphantom{\muu_{s_k^{i\!-\!1}}}
\le \QQ(\ee(s_k^i))-\langle\sigmaa(s_k^i), E\ww(s_k^i)-E\ww(s_k^{i\!-\!1})\rangle 
+ \QQ(E\ww(s_k^i)-E\ww(s_k^{i\!-\!1}))+{}
\smallskip
\label{secpb120}
\\
&\displaystyle
{}+ \langle H(\xi_i - \xi_{i\!-\!1}, \theta_i - \theta_{i\!-\!1}), \muu_{s_k^{i\!-\!1} s_k^i}
(x,  \xi_{i\!-\!1},\theta_{i\!-\!1}, \xi_i, \theta_i,\eta)\rangle+
\langle \{V\}(\theta_i,\eta), \muu_{s_k^i}(x, \xi_i, \theta_i,\eta)\rangle\,.
\nonumber
\end{eqnarray}
We notice that
\begin{equation}\label{Hold101}
\begin{array}{c}
\displaystyle\langle\sigmaa(s_k^i), E\ww(s_k^i)-E\ww(s_k^{i\!-\!1})\rangle = 
\int_{s_k^{i\!-\!1}}^{s_k^i}\langle\sigmaa(s_k^i), E\dot \ww(s)\rangle\,ds\,,
\smallskip
\\
\displaystyle
 \QQ(E\ww(s_k^i)-E\ww(s_k^{i\!-\!1}))
\le\beta_\C \Big(\int_{s_k^{i\!-\!1}}^{s_k^i}\|E\dot \ww(s)\|_2\, ds\Big)^2\,,
\smallskip
\\
\displaystyle
\langle H(\xi_i - \xi_{i\!-\!1}, \theta_i - \theta_{i\!-\!1}), \muu_{s_k^{i\!-\!1} s_k^i}
(x,  \xi_{i\!-\!1},\theta_{i\!-\!1}, \xi_i, \theta_i,\eta)\rangle\le \D_{\!H}(\muu;s_k^{i\!-\!1}, s_k^i)\,.
\end{array}
\end{equation}
On $[0,T]$ we define the piecewise constant function 
${\overline\sigmaa}_k(s):=\sigmaa(s_k^i)$, where $i$ is the smallest index such that $s\le s_k^i$. Summing the inequalities (\ref{secpb120}) for $1\le i\le k$, we obtain
\begin{equation}\label{ct-10044}
\begin{array}{c}
\displaystyle
\vphantom{\int_0^T}
\QQ(\ee(T))+\D_{\!H}(\muu;0,T)+
\langle \{V\}(\theta,\eta), \muu_t(x,\xi,\theta,\eta) \rangle +
\rho_k\int_0^T\|E\dot \ww(s)\|_2\, ds
\ge
\\
\displaystyle
\ge \QQ(e_0)+\V(z_0)
+\int_0^T \langle\ol\sigmaa_k(s), E\dot \ww(s)\rangle\, ds\,,
\end{array}
\end{equation}
where 
$$
\rho_k:=\sup_{1\le i\le k}\beta_\C\int_{s_k^{i\!-\!1}}^{s_k^i}\|E\dot \ww(s)\|_2\, ds\,.
$$
Now conditions (\ref{sh101}) and the continuity of $\sigmaa$ guarantee that 
$\rho_k\to 0$ and that ${\overline \sigmaa}_k\to\sigmaa$ strongly in 
$L^2([0,T];L^2(\Om;\Mnn))$. Hence, taking the limit as $k\to\infty$ in (\ref{ct-10044}), we obtain inequality (\ref{en-ineq-21}), which, together with  (\ref{en-ineq-20}), gives~(ev2).

\subsection{Some properties of the solutions}
We conclude this section by proving some qualitative properties of the Young measure solutions to the evolution problem.

\begin{theorem}\label{thm:keff}
Let $(\uu,\ee,\muu)$ be a globally stable quasistatic evolution of Young measures. 
For every $t\in [0,+\infty)$ let $(\pp(t),\zz(t)):=\bary(\muu_t)$. 
Then $(\uu,\ee,\pp)$ is a quasistatic evolution corresponding to the function 
$\xi\mapsto H_{\rm eff}(\xi,0)$ according to \cite[Definition~4.2]{DM-DeS-Mor}. 
\end{theorem}

\begin{proof}
Let us fix $t\in[0,+\infty)$. We want to prove that
\begin{equation}\label{minim300}
\QQ(\ee(t)) \le \QQ(\ee(t)+\tilde e)+ \HH_{\rm eff}(\tilde p,0)\,.
\end{equation}
for every $(\tilde u, \tilde e, \tilde p)\in A(0)$. 
Let us fix  $(\tilde u, \tilde e, \tilde p)$, let $(\mu^Y_t,\mu^\infty_t)$ be the pair 
of measures associated with $\muu_t$ by \cite[Theorem~4.3]{DM-DeS-Mor-Mor-1},
let $(\mu_t^{x,Y})_{x\in\Om}$ be the disintegration of $\mu^Y_t$
considered in \cite[Remark~4.5]{DM-DeS-Mor-Mor-1}, let  
$$
F(x,\tilde \xi,\tilde \theta):= \int_{\MD{\times}\R}[H(\tilde \xi,\tilde \theta) + V(\theta+\tilde \theta)-  V(\theta)]\,d\mu_t^{x,Y}(\xi,\theta)\,,
$$
let $F^\infty$ be the recession function of $F$  with respect to $(\tilde \xi,\tilde \theta)$,
and let ${\rm co}\,F$ be the convex envelope of $F$ with respect to $(\tilde \xi,\tilde \theta)$. It is easy to see that $F^\infty=H+V^\infty$. We claim that 
\begin{equation}\label{coFx}
{\rm co}\,F=H_{\rm eff}\,.
\end{equation}
Indeed, as $V$ is concave, we have $V(\theta+ \tilde\theta)-V(\theta)\ge V^\infty(\tilde\theta)$, which gives
$$
F(x,\tilde \xi,\tilde \theta)\ge  \int_{\MD{\times}\R}[H(\tilde \xi,\tilde \theta) + V^\infty(\tilde \theta)]\,d\mu_t^{x,Y}(\xi,\theta)=H(\tilde \xi,\tilde \theta) + V^\infty(\tilde \theta)\ge H_{\rm eff}(\tilde \xi,\tilde \theta)\,,
$$
where the intermediate equality follows from the fact that $\mu_t^{x,Y}$ is a probability measure. This implies ${\rm co}\,F\ge H_{\rm eff}$. The opposite inequality can be obtained arguing as in the proof of~(\ref{coF3}).

By Theorem~\ref{thm:1app} there exist a sequence 
$(\tilde u_k, \tilde e_k, \tilde p_k)\in A(0)$ and a sequence $\tilde z_k\in L^1(\Om)$ 
such that $ \tilde p_k\in L^1(\Om;\MD)$, $\tilde e_k\to \tilde e$ strongly in  $L^2(\Om;\Mnn)$, $\tilde p_k\wto \tilde p$ weakly$^*$ in $M_b(\ol\Om;\MD)$, $\tilde z_k\wto 0$ 
weakly$^*$ in $M_b(\ol\Om)$, and
$$
\HH(\tilde p_k, \tilde z_k)+\V(\tilde z_k) \quad\longrightarrow  \quad
\HH_{\rm eff}(\tilde p,0) +\V(0)\,.
$$
Passing to a subsequence, we may assume that $\delta_{(\tilde p_k, \tilde z_k)}$
converges weakly$^*$ to some $\nu \in GY(\ol\Om;\MD{\times}\R)$. We note that 
$\bary(\nu)=(\tilde p,0)$ and that 
$$
\langle H(\tilde \xi,\tilde\theta) +\{V\}(\tilde\theta,\eta),\nu(x,\tilde \xi,\tilde\theta,\eta)\rangle
=\HH_{\rm eff}(\tilde p,0) +\V(0)\,.
$$
By Theorem~\ref{thm311} we deduce that
\begin{equation}\label{677}
\nu=\delta_{(0,0)}+
\tfrac12 \omega_{(\tilde p^a, z)}^{\Ln}+
\tfrac12 \omega_{(\tilde p^a,- z)}^{\Ln}+
\tfrac12 \omega_{(\tilde p^\lambda, z_\lambda)}^{\lambda}+
\tfrac12 \omega_{(\tilde p^\lambda, -z_\lambda)}^{\lambda}\,, 
\end{equation}
where $\lambda:=|\tilde p^s|$ and $\tilde p^\lambda$ is the Radon-Nikodym derivative of
$\tilde p^s$ with respect to $|\tilde p^s|$, while $z$ and $z_\lambda$ are two nonnegative functions such that
\begin{equation}\label{678}
\begin{array}{c}
H_{\rm eff}(\tilde p^a,0)= H(\tilde p^a, z)+ V^\infty(z)\quad
\text{a.e.\ in } \Om\,
\smallskip\\
H_{\rm eff}(\tilde p^\lambda,0)= 
H(\tilde p^\lambda, z_\lambda)+ V^\infty(z_\lambda)
\quad \lambda\text{-a.e.\ in }\ol\Om
\,.
\end{array}
\end{equation}
As $\delta_{(\tilde p_k, \tilde z_k)}\wto\nu$ weakly$^*$ in $GY(\ol\Om;\MD{\times}\R)$,
we have
$$
 \int_\Om F(x,\tilde p_k, \tilde z_k)\,dx \quad\longrightarrow  \quad
 \langle \{F\},\nu\rangle\,,
$$
where 
$$
\{F\}(\tilde \xi,\tilde\theta,\eta):=
\begin{cases}
\eta F(\tilde \xi/\eta,\tilde\theta/\eta)&\text{if }\eta>0\,,
\\
F^\infty(\tilde \xi,\tilde\theta)&\text{if }\eta\le 0\,.
\end{cases}
$$
As $F^\infty=H+V^\infty$, using \eqref{677} and \eqref{678} we obtain
$\langle \{F\},\nu\rangle=\HH_{\rm eff}(\tilde p,0)$, hence using also the strong
convergence of $\tilde e_k$ to $\tilde e$, we deduce
$$
\QQ(\ee(t)+\tilde e_k)+ \int_\Om F(x,\tilde p_k, \tilde z_k)\,dx \quad\longrightarrow  \quad\QQ(\ee(t)+\tilde e) + \HH_{\rm eff}(\tilde p,0)\,.
$$
{}From the definition of $F(x,\tilde \xi, \tilde \theta)$ this is equivalent to saying that
$$
\QQ(\ee(t)+\tilde e_k)+  \HH(\tilde p_k,\tilde z_k)+ \langle \{V\}(\theta + \eta \tilde z_k(x),\eta)-  \{V\}(\theta,\eta), \muu_t (x,\xi,\theta,\eta)\rangle
$$
converges to
$\QQ(\ee(t)+\tilde e) + \HH_{\rm eff}(\tilde p,0)$. Since
$$
\QQ(\ee(t))\le \QQ(\ee(t)+\tilde e_k)+  \HH(\tilde p_k,\tilde z_k)+ \langle \{V\}(\theta + \eta \tilde z_k(x),\eta)-  \{V\}(\theta,\eta), \muu_t (x,\xi,\theta,\eta)\rangle
$$
by (ev1), we obtain (\ref{minim300}) by passing to the limit as $k\to\infty$.

Thanks to \cite[Theorem~4.7]{DM-DeS-Mor}, to conclude the proof of the theorem it is enough to show that
\begin{equation}\label{enineq300}
\QQ(\ee(T)) +\D_{\!H_{\rm eff}}((\pp,0);0,T) \le \QQ(e_0)+ \int_0^T\langle \sigmaa(t),E\dot \ww(t)\rangle\,dt
\end{equation}
for every $T\in(0,+\infty)$, where $\D_{\!H_{\rm eff}}((\pp,0);0,T)$ is defined as in
\cite[Section~4]{DM-DeS-Mor}.
By (ev2) it suffices to prove that
\begin{equation}\label{diss300}
\begin{array}{c}
\D_{\!H_{\rm eff}}((\pp,0);0,T) \le \D_{\!H}(\muu;0,T)+ \langle \{V\}(\theta ,\eta),
\muu_T (x,\xi,\theta,\eta)\rangle-{}
\smallskip
\\
{}- \langle \{V\}(\theta ,\eta),
\muu_0 (x,\xi,\theta,\eta)\rangle\,.
\end{array}
\end{equation}
To show this we fix any subdivision $(t_i)_{0\le i\le k}$ of the interval $[0,T]$. 
{}From the definition of $\D_{\!H}$ we obtain, using the compatibility 
condition~(7.2) of \cite{DM-DeS-Mor-Mor-1},
$$
\begin{array}{c}
\D_{\!H}(\muu;0,T)+ \langle \{V\}(\theta ,\eta),\muu_T (x,\xi,\theta,\eta)\rangle-  
\langle \{V\}(\theta ,\eta),\muu_0 (x,\xi,\theta,\eta)\rangle \ge
 \\
\displaystyle \ge \sum_{i=1}^k\langle H(\xi_i-\xi_{i-1},\theta_i-\theta_{i-1})+ 
\{V\}(\theta_i ,\eta)-\{V\}(\theta_{i-1} ,\eta), \muu_{t_{i\!-\!1}t_i} (x,\xi_{i-1},\theta_{i-1},\xi_i,\theta_i,\eta)\rangle\ge \vspace{-5pt}
\\ 
\displaystyle \ge  \sum_{i=1}^k\langle H_{\rm eff}(\xi_i-\xi_{i-1},\theta_i-\theta_{i-1}),  \muu_{t_{i\!-\!1}t_i} (x,\xi_{i-1},\theta_{i-1},\xi_i,\theta_i,\eta)\ge \vspace{-5pt} 
\\
\displaystyle \ge  \sum_{i=1}^k \HH_{\rm eff}(\pp(t_i)-\pp(t_{i-1}),\zz(t_i)-\zz(t_{i-1}))\, ,
\end{array}
$$
where the last inequality follows from the Jensen inequality.
Recalling that $H_{\rm eff}(\xi,\theta)\ge H_{\rm eff}(\xi,0)$ for every $\xi$ and $\theta$, from the arbitrariness of the subdivision we obtain (\ref{diss300}).
\end{proof}

Every globally stable quasistatic evolution of Young measures is absolutely
continuous with respect to time, as made precise by the following theorem.

\begin{theorem}\label{ACmuu}
Let $(\uu,\ee,\muu)$ be a globally stable quasistatic evolution of Young measures.
Then for every $T\in[0,+\infty)$ the functions $\uu$ and $\ee$ are absolutely continuous
on $[0,T]$ with values in $BD(\Om)$ and $L^2(\Om;\Mnn)$, respectively, while
$\muu$ is absolutely continuous on $[0,T]$ according to \cite[Definition~10.1]{DM-DeS-Mor-Mor-1}.
\end{theorem}

\begin{proof}
The assertion on $\uu$ and $\ee$ follows from
Theorem~\ref{thm:keff} and \cite[Theorem~5.2]{DM-DeS-Mor}.

By \eqref{gammaM} we have 
$$
C^K_V{\rm Var}(\muu;t_1,t_2)\!\leq\! \D_{\!H}(\muu;t_1,t_2)\!+
\!\langle \{V\}(\theta_2,\eta), \muu_{t_2}(x, \xi_2, \theta_2, \eta)\rangle\!-\!
\langle \{V\}(\theta_1,\eta), \muu_{t_1}(x, \xi_1, \theta_1, \eta)\rangle\,.
$$
It follows from the energy balance (ev2) that the right-hand side of the previous inequality is equal to 
$$
\QQ(\ee(t_1))- \QQ(\ee(t_2)) +
\int_{t_1}^{t_2} \langle\sigmaa(t), E\dot \ww(t)\rangle\, dt\,,
$$
therefore
$$
\begin{array}{c}
C^K_V\langle |\xi_2-\xi_1|+|\theta_2-\theta_1|, \muu_{t_1t_2}(x,\xi_1, \theta_1, \xi_2, \theta_2, \eta)\rangle
\leq 
\smallskip
\\
\displaystyle
\leq
|\QQ(\ee(t_1))- \QQ(\ee(t_2))| +
\int_{t_1}^{t_2} |\langle\sigmaa(t), E\dot \ww(t)\rangle|\, dt\,.
\end{array}
$$
Since the functions $t\mapsto \QQ(\ee(t))$ and $t\mapsto \int_0^t |\langle\sigmaa(s), E\dot \ww(s)\rangle|\, ds$ are absolutely continuous, we conclude that
$\muu$ satisfies \cite[Definition~10.1]{DM-DeS-Mor-Mor-1}.
\end{proof}

Owing to the previous theorem, if $(\uu,\ee,\muu)$ is a globally stable
quasistatic evolution of Young measures, then $\muu$ has a weak$^*$ derivative
$\dot\muu_t$ at a.e.\ time $t\in[0+\infty)$ in the sense of 
\cite[Definition~9.4]{DM-DeS-Mor-Mor-1}.
The next theorem deals with the structure of $\dot\muu_t$ and shows that the finite part $\ol\muu^Y_t$ of
$\muu_t$ does not evolve. 

\begin{theorem}\label{thm4.18}
Let $\ol p_0\in L^1(\Om;\MD)$,  $\ol z_0\in L^1(\Om)$,  
$\ww\in AC_{loc}([0,+\infty); H^1(\Om;\Rn))$,
let $(\uu,\ee,\muu)$ be a globally stable quasistatic evolution of Young measures
with boundary datum $\ww$ such that $\ol\muu_0^Y=\delta_{(\ol p_0,\ol z_0)}$,
and let $(\pp(t),\zz(t)):=\bary(\muu_t)$.
Denote  the total variation of the measure $(\dot\pp^s(t),\dot\zz^s(t))$ by $\lambdaa(t)$, and
let $(\dot\pp^\lambdaa(t), \dot\zz^\lambdaa(t))$ be the Radon-Nikodym derivative
of the measure
$(\dot\pp^s(t),\dot\zz^s(t))$ with respect to $\lambdaa(t)$. 
By Lemma~\ref{cothetaG} for a.e.\ $t\in[0,+\infty)$ there exist $\hat\zz(t)\in L^1(\Om)$, 
with $\hat\zz(t)\,\dot\zz^a(t)\ge 0$ a.e.\ on $\Om$, and $\alphaa(t)\in L^\infty(\Om)$,
with $0\leq\alphaa(t)\leq1$  a.e.\ on $\Om$, such that 
\begin{equation}\label{eqthm418}
\begin{array}{c}
\dot\zz^a(t)=\alphaa(t)\hat\zz(t) + (1-\alphaa(t))(-\hat\zz(t))\,, 
\smallskip
\\ 
H_{\rm eff}(\dot \pp^a(t),\dot \zz^a(t))= H(\dot \pp^a(t),\hat \zz(t))+ V^\infty(\hat\zz(t))
\end{array}
\end{equation}
a.e.\ in $\Om$, and there exist $\hat\zz_\lambdaa(t)\in L^1_{\lambdaa(t)}(\ol\Om)$, with 
$\hat\zz_\lambdaa(t)\,\dot\zz^\lambdaa(t)\ge 0$ $\lambdaa(t)$-a.e.\ on $\ol\Om$,
and $\alphaa_\lambdaa(t)\in L^\infty_{\lambdaa(t)}(\ol\Om)$, with $0\le\alphaa_\lambdaa(t)\le1$
$\lambdaa(t)$-a.e.\ on $\ol\Om$,
such that
\begin{equation}\label{eqthm418b}
\begin{array}{c}
\dot\zz^\lambdaa(t)=\alphaa_\lambdaa(t) \hat\zz_\lambdaa(t) + (1-\alphaa_\lambdaa(t))
(-\hat\zz_\lambdaa(t))\,,
\smallskip
\\ 
H_{\rm eff}(\dot\pp^\lambdaa(t), \dot\zz^\lambdaa(t))= 
H(\dot\pp^\lambdaa(t), \hat\zz_\lambdaa(t))+ V^\infty(\hat\zz_\lambdaa(t))
\end{array}
\end{equation}
$\lambdaa(t)$-a.e.\ in $\ol\Om$. 
Then 
\begin{equation}\label{mupunto}
\begin{array}{c}
\dot\muu_t=\delta_{(0,0)}+
\alphaa(t)\omega_{(\dot\pp^a(t), \hat\zz(t))}^{\Ln}+
(1-\alphaa(t))\omega_{(\dot\pp^a(t), -\hat\zz(t))}^{\Ln}+{}
\smallskip\\
{}+
\alphaa_\lambdaa(t)\omega_{(\dot\pp^\lambdaa(t), \hat\zz_\lambdaa(t))}^{\lambdaa(t)}+
(1-\alphaa_\lambdaa(t))\omega_{(\dot\pp^\lambdaa(t), -\hat\zz_\lambdaa(t))}^{\lambdaa(t)}
\end{array}
\end{equation}
and 
\begin{equation}\label{heffzpunto}
\HH_{\rm eff}(\dot\pp(t),\dot\zz(t))=\HH_{\rm eff}(\dot\pp(t),0)
\end{equation}
for a.e.\ $t\in[0,+\infty)$. Moreover,
\begin{equation}\label{weq1}
\ol\muu_t^Y=\delta_{(\ol p_0,\ol z_0)} 
\end{equation}
for every $t\in [0,+\infty)$.
\end{theorem}

\begin{proof}
As the system $\muu$ is absolutely continuous with respect to time by Theorem~\ref{ACmuu}, 
we have by \cite[Theorem~10.4]{DM-DeS-Mor-Mor-1} that
for every $t_1,t_2\in[0,+\infty)$ with $t_1<t_2$
\begin{equation}\label{en2eq}
\D_{\!H}(\muu;t_1,t_2)= \int_{t_1}^{t_2}\langle H(\xi,\theta),\dot \muu_t(x,\xi,\theta,\eta)
\rangle\, dt\,,
\end{equation}
where $\dot\muu_t$ is the weak$^*$ derivative of $\muu$ at time $t$
in the sense of \cite[Definition~9.4]{DM-DeS-Mor-Mor-1}.
By \cite[Remark~9.6]{DM-DeS-Mor-Mor-1} we also have that 
the maps $t\mapsto \pp(t)$ and $t\mapsto \zz(t)$ are absolutely continuous
and $(\dot\pp(t),\dot\zz(t))=\bary(\dot\muu_t)$
for a.e.\ $t\in[0,+\infty)$.

From Theorem~\ref{thm:keff} and \cite[Proposition~5.6]{DM-DeS-Mor} 
it follows that
$$
\QQ(\ee(t_2))+\int_{t_1}^{t_2}\HH_{\rm eff}(\dot \pp(t),0)\, dt =
\QQ(\ee(t_1))+\int_{t_1}^{t_2} \langle\sigmaa(t), E\dot\ww(t)\rangle\, dt\,.
$$
By (ev2) and \eqref{en2eq} we deduce that for every $t_1<t_2$
\begin{equation}\label{balance}
\int_{t_1}^{t_2}\langle H(\xi,\theta),\dot \muu_t
\rangle\, dt
+\langle \{V\}(\theta_2,\eta)-\{V\}(\theta_1,\eta), \muu_{t_1t_2}\rangle
=  \int_{t_1}^{t_2}\HH_{\rm eff}(\dot \pp(t),0)\, dt\,,
\end{equation}
where the measure $\dot\muu_t$ acts on $(x,\xi,\theta,\eta)$, while
$\muu_{t_1t_2}$ acts on $(x,\xi_1,\theta_1,\xi_2,\theta_2,\eta)$.
By concavity we have 
$\{V\}(\theta_2,\eta)-\{V\}(\theta_1,\eta)\geq V^\infty(\theta_2-\theta_1)$, so that
\eqref{balance} yields
\begin{equation}\label{balanceinf}
\int_{t_1}^{t_2}\langle H(\xi,\theta),\dot \muu_t
\rangle\, dt
+\langle V^\infty(\theta_2-\theta_1), \muu_{t_1t_2}\rangle
\leq  \int_{t_1}^{t_2}\HH_{\rm eff}(\dot \pp(t),0)\, dt\,.
\end{equation}
Dividing by $t_2-t_1$ and letting $t_2\to t_1=t$ we obtain
\begin{equation}\label{balancecon}
\langle H(\xi,\theta)+V^\infty(\theta),\dot \muu_t (x,\xi,\theta,\eta) \rangle
\leq  \HH_{\rm eff}(\dot \pp(t),0) \leq  \HH_{\rm eff}(\dot \pp(t),\dot\zz(t))
\end{equation}
for a.e.\ $t\in[0,+\infty)$.
Using the Jensen inequality for generalized Young measures 
\cite[Theorem~6.5]{DM-DeS-Mor-Mor-1} we conclude that 
\begin{equation}\label{fundeq}
\langle H(\xi,\theta)+V^\infty(\theta), \dot\muu_t(x,\xi,\theta,\eta) \rangle =
\HH_{\rm eff}(\dot \pp(t),0)=
\HH_{\rm eff}(\dot \pp(t),\dot \zz(t))
\end{equation}
for a.e.\ $t\in[0,+\infty)$, which shows, in particular, \eqref{heffzpunto}. By taking the derivative of \eqref{balance} 
we obtain from \eqref{fundeq}
\begin{equation}\label{vhder}
\lim_{t_2\to t_1^+}\langle \tfrac{1}{t_2-t_1}(\{V\}(\theta_2,\eta)-\{V\}(\theta_1,\eta))
-V^\infty(\tfrac{\theta_2-\theta_1}{t_2-t_1}), \muu_{t_1t_2}\rangle
=0\,.
\end{equation}

As $-V^\infty$ is convex and positively homogeneous, by 
\cite[Theorem~10.4]{DM-DeS-Mor-Mor-1} we have that for every $t\in [0,+\infty)$
$$
\langle V^\infty(\theta_1-\theta_0),\muu_{0t}\rangle
\ge \int_0^t \langle  V^\infty,\dot \muu_s\rangle\,ds\,,
$$
where $\muu_{0t}$ acts on the variables $(x,\xi_0,\theta_0,\xi_1,\theta_1,\eta)$.
Since $s\mapsto  \langle  \{V\}, \muu_s\rangle$ is absolutely continuous
and by \eqref{vhder} 
$$
\langle  V^\infty,\dot \muu_s\rangle =\tfrac{d}{ds} \langle  \{V\}, \muu_s\rangle
$$
for a.e.\ $s\in[0,+\infty)$,
we deduce that
\begin{equation}\label{yyy1}
\langle V^\infty(\theta_1-\theta_0),\muu_{0t}\rangle
\ge 
 \langle  \{V\}(\theta_1,\eta)- \{V\}(\theta_0,\eta), \muu_{0t}\rangle\,.
\end{equation}

Let $\psi\colon\ol\Om{\times}(\MD{\times}\R)^2{\times}\R\to
\Om{\times}\R{\times}\R$ be defined by
$$
\psi(x,\xi_0,\theta_0,\xi_1,\theta_1,\eta):=(x,\theta_1-\theta_0,\eta)\,,
$$
and let $\nu:=\psi(\muu_{0t})$. By repeating the arguments used in the proof of 
\eqref{Vpp1d}--\eqref{Vpp3d} we obtain
$$
 \langle  \{V\}(\theta_1,\eta)- \{V\}(\theta_0,\eta), \muu_{0t}\rangle
 \ge
  \langle  \{V\}(\eta\ol z_0(x)+ \tilde\theta,\eta)- \{V\}(\eta\ol z_0(x),\eta), \nu\rangle
 \,,
$$
where $\nu$ acts on $(x,\tilde\theta,\eta)$. This inequality, together with \eqref{yyy1}, yields
$$
 \langle V^\infty, \nu\rangle
 \ge
  \langle  \{V\}(\eta\ol z_0(x)+ \tilde\theta,\eta)- \{V\}(\eta\ol z_0(x),\eta), \nu\rangle
 \,,
$$
By \eqref{VinftyV} it follows that $\ol \nu^Y=\delta_0$. Arguing as in the proof of 
\eqref{yy23}, we conclude that 
\begin{equation}\label{stop1}
\pi_\R(\ol\muu_t^Y)=\delta_{\ol z_0}\,.
\end{equation}
where
$\pi_\R\colon\ol\Om{\times}\MD{\times}\R{\times}\R\to\ol\Om{\times}\R{\times}\R$
is defined by $\pi_\R(x,\xi,\theta,\eta):=(x,\theta,\eta)$.

As the system $\muu$ is absolutely continuous with respect to time, 
we also have that the system
$\ol\muu^Y$ given by $\ol\muu^Y_{t_1\dots t_m}:=\ol{(\muu_{t_1\dots t_m})}^Y$ is
absolutely continuous, so that its weak$^*$ derivative $\dot{\ol\muu}{}^Y_t$ exists 
for a.e.\ $t\in[0,+\infty)$. We note that in general the weak$^*$ derivative $\dot{\ol\muu}{}^Y_t$ of 
$\ol\muu^Y$ does not coincide with the finite part $\ol{(\dot\muu_t)}{}^Y$  of $\dot\muu_t$.
Nevertheless the following identity holds
\begin{equation}\label{DC}
\ol{(\dot{\ol\muu}{}^{Y^{\vphantom{A}}}_t)}^Y=\ol{(\dot\muu_t)}{}^Y\,.
\end{equation}
To see this we observe that at every time $t$ where both $\dot\muu_t$ and $\dot{\ol\muu}{}^Y_t$ 
are defined, there also 
exists the weak$^*$ limit of the difference quotients $q_{ts}(\hat\muu^{\infty}_{ts})$ as $s\to t^+$ (see \cite[Definition~9.1]{DM-DeS-Mor-Mor-1}), which is an element of 
$M_*^+(\ol\Om{\times}\MD{\times}\R{\times}\R)$ supported on $\{\eta=0\}$ (see \cite[Definition~2.8]{DM-DeS-Mor-Mor-1}).
The equality \eqref{DC} follows now from the identity 
$$
\dot\muu_t =\dot{\ol\muu}{}^Y_t+\lim_{s\to t^+}q_{ts}(\hat\muu^{\infty}_{ts})\,.
$$

Using \eqref{stop1} and Corollary~\ref{cor-prod}
it is easy to see that $\psi(\ol\muu^Y_{t s})=\delta_0$ 
for every $s>t$, which, in turn, implies 
\begin{equation}\label{stop2}
\pi_{\R}(\dot{\ol\muu}{}_t^Y)=\delta_0\,.
\end{equation}
By \eqref{DC}, \eqref{stop2}, and \cite[Lemma~4.8]{DM-DeS-Mor-Mor-1} we deduce that $\pi_{\R}(\ol{(\dot\muu_t)}{}^Y)=\delta_0$.
It follows that
$$
\langle \{V\}(\theta,\eta), \dot\muu_t(x,\xi,\theta,\eta) \rangle=
\langle V^\infty(\theta), \dot\muu_t(x,\xi,\theta,\eta) \rangle+\V(0)\,,
$$
hence by \eqref{fundeq} we have 
\begin{equation}
\langle H(\xi,\theta)+\{V\}(\theta,\eta), \dot\muu_t(x,\xi,\theta,\eta) \rangle =
\HH_{\rm eff}(\dot \pp(t),\dot \zz(t))+ \V(0)
\end{equation}
for a.e.\ $t\in[0,+\infty)$.
Identity \eqref{mupunto} is now a consequence of Theorem~\ref{thm311} (applied
with $p_0=0$ and $z_0=0$).

We now claim that for almost every $t\in[0,+\infty)$ 
\begin{equation}\label{CS}
 \dot{\ol\muu}{}^Y_t=\delta_{(0,0)}\,.
\end{equation}
Let us fix $t\in[0,+\infty)$ such that $\dot\muu_t$ and $\dot{\ol\muu}{}^Y_t$ exist. 
By \eqref{mupunto} and \eqref{DC} there exists $
\nu_1^{\infty}\in M_*^+(\ol\Om{\times}\MD{\times}\R{\times}\R)$, with support 
contained in $\{\eta=0\}$, such that 
\begin{equation}\label{dec.}
\dot{\ol\muu}{}^Y_t=\delta_{(0,0)}+\nu_1^{\infty}\,.
\end{equation}

By \eqref{dec.}, \eqref{stop2}, and Lemma~\ref{prop-prod} we infer that
\begin{eqnarray}
& \langle f(x,\xi,\theta,\eta), \nu_1^\infty(x,\xi,\theta,\eta)\rangle
=\langle f(x,\xi,0,\eta), \nu_1^\infty(x,\xi,\theta,\eta)\rangle\,,
\label{eqprj}
\\
& \displaystyle
\label{decmuY}
\langle f,  \dot{\ol\muu}{}_t^Y \rangle =\int_\Om f(x,0,0,1)\, dx +
\langle f(x,\xi,0,\eta), \nu_1^\infty(x,\xi,\theta,\eta)\rangle
\end{eqnarray}
for every $f\in B_{\infty,1}^{hom}(\ol\Om{\times}\MD{\times}\R{\times}\R)$.
Let $\nu_2^{\infty}\in M_*^+(\ol\Om{\times}\MD{\times}\R{\times}\R)$
be the weak$^*$ limit of the difference quotients $q_{ts}(\hat\muu^{\infty}_{ts})$
as $s\to t^+$.
From \eqref{mupunto} and \eqref{decmuY} it follows that
\begin{equation}\label{stop3}
\begin{array}{c}
\displaystyle
\langle f, \nu_2^\infty\rangle =
\int_\Om\alphaa f(x,\dot\pp^a,\hat\zz,0)\, dx +
\int_\Om(1-\alphaa)f(x,\dot\pp^a, -\hat\zz,0)\,dx+{}
\smallskip\\
\displaystyle
{}+
\int_{\ol\Om}\alphaa_\lambdaa f(x,\dot\pp^\lambdaa, \hat\zz_\lambdaa,0)\,d\lambdaa +
\int_{\ol\Om}(1-\alphaa_\lambdaa)f(x,\dot\pp^\lambdaa, -\hat\zz_\lambdaa,0)\,d\lambdaa-{}
\smallskip\\
\displaystyle\vphantom{\int}
{}-\langle f(x,\xi,0,\eta), \nu_1^\infty(x,\xi,\theta,\eta)\rangle
\end{array}
\end{equation}
for every $f\in B_{\infty,1}^{hom}(\ol\Om{\times}\MD{\times}\R{\times}\R)$.
In the previous formula and in the remaining part of the proof the dependence upon time
is omitted, since $t$ is fixed.

We shall prove that $\nu_1^\infty=0$, so that claim \eqref{CS} will follow from 
the decomposition \eqref{dec.}. 
According to \cite[Remark~4.5]{DM-DeS-Mor-Mor-1}, there exist
$\pi^\infty\in M_b^+(\ol\Om)$ and a family $(\nu_1^{x,\infty})_{x\in\ol\Om}$
of probability measures on $\Sigma:=\{\xi\in\MD:\,|\xi|=1\}$ such that
\begin{equation}\label{kin}
\langle f(x,\xi,0,\eta), \nu_1^\infty(x,\xi,\theta,\eta)\rangle=
\int_{\ol\Om}\Big(\int_\Sigma f(x,\xi,0,0)\,d\nu_1^{x,\infty}(\xi)\Big)
\, d\pi^\infty(x)
\end{equation}
for every $f\in B_{\infty,1}^{hom}(\ol\Om{\times}\MD{\times}\R{\times}\R)$.
We have to prove that $\pi^\infty=0$.
Let us consider the Lebesgue decomposition $\pi^\infty=\pi^{\infty,a}+\pi^{\infty,s}$.

We first prove that $\pi^{\infty,a}=0$.
We argue by contradiction. Assume that there exists a Borel set $A$, with
$\Ln(A)>0$ and $\lambdaa(A)=0$, such that
$\pi^{\infty,a}(x)>0$ for every $x\in A$. 

For every Borel set $A'\subset A$ let $f(x,\xi,\theta,\eta):=
1_{A'}(x)|\xi|$. Since $\nu_2^\infty$ is positive,
by \eqref{stop3} and \eqref{kin} we deduce that
$$
\int_{A'}|\dot\pp^a|\, dx\geq \int_{A'}\pi^{\infty,a}\, dx
$$
for every Borel set $A'\subset A$. Therefore, $|\dot\pp^a|>0$ a.e.\ on $A$
and there exists $h\in L^\infty(A)$,
with $0< h\leq1$, such that $\pi^{\infty,a}=h|\dot\pp^a|$ on $A$.

Since $|\dot \pp^a|>0$ a.e.\ on $A$, by \eqref{eqthm418} and Lemma~\ref{lemma01}
we have that $|\hat \zz|>0$ a.e.\ on $A$. Hence, there exists $M>0$ such that
the set $A_M:=\{x\in A:\, |\dot \pp^a(x)|<M,\, |\hat\zz(x)|>\frac1M\}$ has positive Lebesgue
measure. Let us consider the function $f(x,\xi,\theta,\eta):=1_{A_M}(x)(|\xi|-M^2|\theta|)^+$.
Using \eqref{stop3}, \eqref{kin}, the fact that $\pi^{\infty,a}=h|\dot\pp^a|$
on $A$, and the positivity of $\nu_2^\infty$, we obtain
$$
\int_{A_M} (|\dot\pp^a|-M^2|\hat\zz|)^+\,dx\geq \int_{A_M} h|\dot\pp^a|\,dx\,,
$$
which gives the contradiction, since the left-hand side vanishes, while the right-hand
side is strictly positive. Therefore, $\pi^{\infty,a}=0$.

It remains to prove that $\pi^{\infty,s}=0$. We argue by contradiction.
Assume that there exists a Borel set $A$ with $\Ln(A)=0$ and $\pi^{\infty,s}(A)>0$.

For every Borel set $A'\subset A$ let $f(x,\xi,\theta,\eta):=
1_{A'}(x)|\xi|$. Since $\nu_2^\infty$ is positive,
by \eqref{stop3} and \eqref{kin} we deduce that
$$
\int_{A'}|\dot\pp^\lambdaa|\, d\lambdaa\geq \pi^{\infty,s}(A')
$$
for every Borel set $A'\subset A$. 
Therefore, there exists $h_\lambdaa\in L^\infty_\lambdaa(A)$,
with $0\leq h_\lambdaa\leq1$, such that $\pi^{\infty,s}=h_\lambdaa|\dot\pp^\lambdaa|
\lambdaa$ on $A$.
Taking $A$ smaller if needed, we may assume that 
$\lambdaa(A)>0$, $h_\lambdaa>0$
and $|\dot\pp^\lambdaa|>0$ $\lambdaa$-a.e.\ on $A$.
We can now argue as before with $\Ln$ replaced by $\lambdaa$ and
we obtain a contradiction, which implies $\pi^{\infty,s}=0$.

This concludes the proof of the fact that $\nu_1^\infty=0$ and, in turn, of \eqref{CS}
by \eqref{dec.}.

By \cite[Theorem~10.4]{DM-DeS-Mor-Mor-1} identity \eqref{CS} yields
$$
{\rm Var}(\ol\muu^Y;0,t)=\int_0^t \langle \sqrt{|\xi|^2+\theta^2},\dot{\ol\muu}{}^Y_t(x,\xi,\theta,\eta)
\rangle\,dt
=0
$$
for every $t\in[0,+\infty)$. By Corollary~\ref{cor-prod} we deduce that
$$
\begin{array}{c}
\langle|\xi_1-\eta\ol p_0(x)|+|\theta_1-\eta\ol z_0(x)|,\ol\muu^Y_t(x,\xi_1,\theta_1,\eta)\rangle=
\smallskip
\\
=\langle|\xi_1-\xi_0|+|\theta_1-\theta_0|,\ol\muu^Y_{0t}(x,\xi_0,\theta_0,\xi_1,\theta_1,\eta)\rangle
\leq {\rm Var}(\ol\muu^Y;0,t)=0\,,
\end{array}
$$
which easily implies the \eqref{weq1}.
\end{proof}

\begin{remark}\label{rm:mupunto}
We remark that in the previous proof we could not deduce \eqref{weq1} from \eqref{mupunto} 
simply by ``integration'' with respect to time. In fact,
for a system $\muu$ of generalized Young measures it is not true in general that the knowledge of $\dot \muu_t$ 
and of $\muu_0$ is enough to identify $\muu_t$, as the following example shows. 

Let $U:=(0,1)$ and for every $t\in [0,+\infty)$ let $\pp(t)\in L^1(U)$ be the characteristic function $1_{(0,t)}$. We now consider the two systems $\muu^1$, $\muu^2\in SGY([0,+\infty),U;\R)$ defined by
$$
\muu^1_{t_1\dots t_m}:=\delta_{(\pp(t_1),\dots,\pp(t_m))}\quad\text{and}\quad \muu^2_{t_1\dots t_m}:=\delta_{(0,\dots,0)}+\omega^{\Lone}_{(\pp(t_1),\dots,\pp(t_m))}
$$
for every finite sequence $0\leq t_1<t_2<\cdots<t_m$. Note that for every $0\leq t<t_1$ 
$$
\langle |\theta_1-\theta|, \muu^1_{tt_1}\rangle=\langle |\theta_1-\theta|, \muu^2_{tt_1}\rangle=
\int_0^1|\pp(t_1)-\pp(t)|\,dx=t_1-t\,,
$$
which shows that both $\muu^1$ and $\muu^2$ are absolutely continuous. Moreover, for every 
$f\in C^{hom}(U{\times}\R{\times}\R)$ and every $0\leq t<t_1$ we have
\begin{eqnarray*}
&\displaystyle \langle f(x,\tfrac{\theta_1-\theta}{t_1-t},\eta), \muu^1_{tt_1}(x,\theta,\theta_1,\eta)\rangle=
\int_0^1 f(x, \tfrac{\pp(t_1)-\pp(t)}{t_1-t}, 1)\, dx=\\
&\displaystyle=\int_0^{t} f(x, 0,1)\, dx +\tfrac{1}{t_1-t}\int_{t}^{t_1}f(x, 1, t_1-t)\, dx +
\int_{t_1}^1f(x,0,1)\, dx
\end{eqnarray*}
and 
\begin{eqnarray*}
&\displaystyle\langle f(x,\tfrac{\theta_1-\theta}{t_1-t},\eta), \muu^2_{tt_1}(x,\theta,\theta_1,\eta)\rangle=
\int_0^1 f(x, 0,1)\, dx+\int_0^1 f(x, \tfrac{\pp(t_1)-\pp(t)}{t_1-t}, 0)\, dx\\
&\displaystyle=\int_0^1 f(x, 0,1)\, dx+\tfrac{1}{t_1-t}\int_{t}^{t_1}f(x, 1, 0)\, dx\,.
\end{eqnarray*}
It follows that 
$$
\begin{array}{c}
\displaystyle
\lim_{t_1\to t^+} \langle f(x,\tfrac{\theta_1-\theta}{t_1-t},\eta), \muu^1_{tt_1}(x,\theta,\theta_1,\eta)\rangle=
\lim_{t_1\to t^+} \langle f(x,\tfrac{\theta_1-\theta}{t_1-t},\eta), \muu^2_{tt_1}(x,\theta,\theta_1,\eta)\rangle=
\smallskip\\
\displaystyle
=\int_0^1 f(x, 0,1)\, dx+ f(t, 1, 0)\,,
\end{array}
$$
which yields
$$
\dot\muu^1_t=\dot\muu^2_t=\delta_0+\omega_1^{\delta_t}
$$
for every $t\in[0,+\infty)$.
We point out that, although $\ol{(\dot\muu{}^1_t)}^Y=\delta_0$,
the finite part of $\muu^1_t$, which coincides with $\muu^1_t$ itself,
evolves in time.
\end{remark}

\end{section}

\begin{section}{An example}\label{exa}

In this section we
assume that  $\C$ is isotropic, which implies that
$$
\C\xi=2\mu\xi_D+\kappa(\tr\,\xi)I
$$ 
for some constants $\mu>0$ and $\kappa>0$.
We also assume that 
\begin{equation}\label{K-sphere}
\begin{array}{c}
K:=\{(\sigma,\zeta)\in\MD{\times}\R: |\sigma|^2+\zeta^2\le 1\}\,,
\smallskip
\\
V(\theta):=\frac12-\frac12\sqrt{1+\theta^2}\,, \qquad \Ga_0:=\partial\Om\,, 
\qquad\Ga_1:=\emptyset\,.
\end{array}
\end{equation}
Let us fix a constant $\theta_0>0$ and a $2{\times}2$ matrix $\xi_0$ with $\tr\, \xi_0 = 0$. 
We assume that the symmetric part $\xi_0^s$ of $\xi_0$ is different from $0$. We will examine the globally stable quasistatic evolution corresponding to the boundary datum
$$
\ww(t,x):=t\xi_0x\,,
$$
and to the initial conditions 
$$
u_0(x)=0\,,\quad e_0(x)=0\,, \quad p_0(x)=0\,, \quad
z_0(x)=\theta_0\,.
$$

\begin{theorem}\label{thm:ex1}
Assume that $\C$, $K$, $V$, $\Gamma_0$, $\Gamma_1$, $\theta_0$, $\xi_0$, $\xi_0^s$, $\ww$, $u_0$, $e_0$, $p_0$, and $z_0$ satisfy the conditions considered at the beginning of this section. Let
\begin{equation}\label{t0}
t_0:=\frac{\sqrt3}{4\mu|\xi_0^s|}\,,
\end{equation}
and let $\uu$, $\ee$, $\pp$, $\zz_\infty$ be defined by
$$
\begin{array}{cc}
\uu(t,x) :=  t\xi_0x\,,
\quad
&
\quad
\ee(t,x) := \alpha(t) \xi_0^s\,,
\smallskip
\\
\pp(t,x) := \beta(t)  \xi_0^s\,,
\qquad
&
\quad
 \zz_\infty(t,x) := \frac{1}{\sqrt3} |\pp(t,x)|\,,
\end{array}
$$
where
$$
\alpha(t) := 
\begin{cases} t & \text{for } t\in [0,t_0]\,,
\\
 t_0 &  \text{for } t\in [t_0,+\infty)\,,
 \end{cases}
 \qquad
\beta(t):=  \begin{cases}0& \text{for } t\in [0,t_0]\,,
 \\
 t-t_0 & \text{for } t\in [t_0,+\infty)\,,
 \end{cases}
$$
and let $\muu\in SGY([0,+\infty),\ol\Om;\MD{\times}\R)$
be the system defined by
$$
\begin{array}{c}
\muu_{t_1\dots t_m}:=\delta_{((p_0,z_0),\dots,(p_0,z_0))}
+\tfrac12 \omega^{\Ln}_{((\pp(t_1),\zz_\infty(t_1)),\dots, (\pp(t_m),\zz_\infty(t_m)))} +{}
\smallskip\\
{}+\tfrac12 \omega^{\Ln}_{((\pp(t_1),-\zz_\infty(t_1)),\dots, (\pp(t_m),-\zz_\infty(t_m)))}
\end{array}
$$
for every finite sequence $t_1,\dots,t_m$ with $0\leq t_1<\dots<t_m$.
Then $(\uu,\ee,\muu)$ is a globally stable quasistatic evolution of Young measures
with boundary datum $\ww$ and initial condition $\muu_0=\delta_{(p_0,z_0)}$.
\end{theorem}

\begin{proof}
By \eqref{Keff} and \eqref{K-sphere} we have
$$
K_{\rm eff}=\{(\sigma,\zeta)\in\MD{\times}\R :
 |\sigma|\le \tfrac{\sqrt 3}2 ,\ |\zeta|\le \sqrt{1-|\sigma|^2}-\tfrac12\}\,.
$$
By \eqref{supeff} we have
$$
H_{\rm eff}(\xi,0)= \tfrac{\sqrt 3}2|\xi| \quad \hbox{for every }\xi \in \MD\,.
$$
Thanks to \eqref{t0} the function $(\uu,\ee,\pp)$ satisfies condition (b) of \cite[Theorem~6.1]{DM-DeS-Mor}
for the dissipation function 
$\xi\mapsto H_{\rm eff}(\xi,0)$, therefore it is a quasistatic evolution for the same dissipation function  according to 
\cite[Definition~4.2]{DM-DeS-Mor}. It follows, in particular, that
\begin{equation}\label{kkk}
\QQ(\ee(t))\le \QQ(\ee(t)+\tilde e)+ \HH_{\rm eff}(\tilde p,0)
\end{equation}
for every $t\in[0,+\infty)$ and every $(\tilde u,\tilde e,\tilde p)\in A_{reg}(0)$. 
Since $H_{\rm eff}$ is convex and even, we have $\HH_{\rm eff}(\tilde p,0)\le 
\HH_{\rm eff}(\tilde p,\tilde z)$ for every $\tilde z\in L^1(\Om)$. 
{}From \eqref{coF3} and \eqref{kkk} it follows that
\begin{eqnarray*}
& \displaystyle
\QQ(\ee(t))\le \QQ(\ee(t)+\tilde e)+ \HH(\tilde p,\tilde z)+\V(\tilde z + z_0)- \V(z_0)=
\\
& \displaystyle
=\QQ(\ee(t)+\tilde e) + \HH(\tilde p,\tilde z) +
\langle \{V\}(\theta+\eta\, \tilde z(x),\eta)- \{V\}(\theta,\eta), 
\muu_t(x, \xi, \theta,\eta)\rangle\,,
\end{eqnarray*}
which gives the global stability condition (ev1) of Definition~\ref{maindef}.

Let us prove the energy balance (ev2). Since for $0\le t_1<t_2$ we have
$$
\langle\sqrt{ |\xi_2-\xi_1|^2 +  |\theta_2-\theta_1|^2} ,
\muu_{t_1t_2}(x,\xi_1,\theta_1,\xi_2,\theta_2,\eta)\rangle=
\frac{2}{\sqrt3}\|\pp(t_2)-\pp(t_1)\|_1\,,
$$
we deduce that
$$
 \D_{\!H}(\muu;0,T)=\frac{2}{\sqrt3}\int_0^T\|\dot \pp(t)\|_1\,dt\,.
$$

Therefore, if $T\in (0,t_0]$ condition (ev2) is satisfied, since
$$
\begin{array}{c}
  \displaystyle
\vphantom{ \int_0^T}
 \QQ(\ee(T))=\mu T^2|\xi_0^s|^2\Ln(\Om)\,, \qquad  \QQ(\ee(0))=0\,, \qquad  \D_{\!H}(\muu;0,T)=0\,,
 \\
    \displaystyle
\vphantom{ \int_0^T}
  \langle \{V\}(\theta,\eta), \muu_T(x, \xi, \theta, \eta)\rangle=\V(z_0)=\langle \{V\}(\theta,\eta), \muu_0(x, \xi,
\theta, \eta)\rangle\,, 
  \\
  \displaystyle
 \int_0^T \langle\sigmaa(t), E\dot \ww(t)\rangle\, dt =  \int_0^T
 2\mu t |\xi_0^s|^2\Ln(\Om)\,dt=\mu T^2|\xi_0^s|^2\Ln(\Om)\,.
 \end{array}
$$

If $T\in[t_0,+\infty)$  condition (ev2) is satisfied, since
$$
\begin{array}{c}
   \displaystyle
\vphantom{ \int_0^T}
 \QQ(\ee(T))=\mu t_0^2|\xi_0^s|^2\Ln(\Om)\,, \qquad 
  \D_{\!H}(\muu;0,T)=\frac{2}{\sqrt3}(T-t_0)|\xi_0^s|\Ln(\Om)\,,
 \\
   \displaystyle
\vphantom{ \int_0^T}
  \langle \{V\}(\theta,\eta), \muu_T(x, \xi, \theta, \eta)\rangle=\V(z_0)+
  \V^\infty(\zz_\infty(T))=\V(z_0)-\frac1{2\sqrt3}(T-t_0)|\xi_0^s|\Ln(\Om) \,,
  \\ 
     \displaystyle
\vphantom{ \int_0^T}
\QQ(\ee(0))=0   \,, \qquad   \langle \{V\}(\theta,\eta), \muu_0(x, \xi, \theta, \eta)\rangle=
\V(z_0)\,, 
  \\
  \displaystyle
 \int_0^T \langle\sigmaa(t), E\dot \ww(t)\rangle\, dt =  \int_0^{t_0}
 2\mu t |\xi_0^s|^2\Ln(\Om)\,dt+
 \int_{t_0}^T 2\mu t_0 |\xi_0^s|^2\Ln(\Om)\,dt =
 \\
    \displaystyle
\vphantom{ \int_0^T}
 =\mu t_0^2|\xi_0^s|^2\Ln(\Om)+ 2\mu(T-t_0)  t_0|\xi_0^s|^2\Ln(\Om)=
 \mu t_0^2|\xi_0^s|^2\Ln(\Om)+ \frac{\sqrt3}2(T-t_0)  |\xi_0^s|\Ln(\Om)\,, 
 \end{array}
$$
where in the last equality we use \eqref{t0}.

It remains to prove that $(\uu,\ee,\muu)\in AY([0,+\infty),\ww)$. Let us fix
a finite sequence $t_1,\dots,t_m$ 
with $0\le t_1<\dots<t_m$. By an algebraic property of $\MD$ there exist 
$a,b\in \Rn$, with $|b|=1$, and a skew symmetric $2{\times}2$-matrix
$q$ such that $\xi_0^s=a{\,\otimes\,} b+q$. Therefore there exists a skew symmetric 
$2{\times}2$-matrix $\omega$ such that $\xi_0=a{\,\odot\,} b+\omega$.

For every $k\in\N$ and $i\in\Z$ we set
\begin{eqnarray*}
&A^i_k:=\{x\in\Om: \tfrac{i}{k}<b{\,\cdot\,}x\le  \tfrac{i+1}{k}- \tfrac{1}{k^2}\}\,,
\\
&B^i_k:=\{x\in\Om:\tfrac{i+1}{k}- \tfrac{1}{k^2}<b{\,\cdot\,}x\le  
\tfrac{i+1}{k}- \tfrac{1}{2k^2}\}\,,
\\
&\displaystyle
C^i_k:=\{x\in\Om:\tfrac{i+1}{k}- \tfrac{1}{2k^2}<b{\,\cdot\,}x\le  \tfrac{i+1}{k}\}\,,
\\
&
A_k:= \bigcup_{i}A^i_k\,,\qquad B_k:= \bigcup_{i}B^i_k\,,
\qquad C_k:= \bigcup_{i}C^i_k\,,
\\
&\displaystyle
\qquad
v_k(x):=
\begin{cases}
\frac{i}{k}a+qx &
\hbox{if } x\in A^i_k\,,
\\
\big(k b{\,\cdot\,}x-i-1+\frac{i+1}{k}\big)a+qx
&
\hbox{if } x\in B^i_k\cup C^i_k\,,
\end{cases}
\\
&\displaystyle
\qquad
\theta_k(x):=
\begin{cases}
0&
\hbox{if } x\in A^i_k\,,
\\
\frac{k}{\sqrt3} |\xi_0^s|
&
\hbox{if } x\in B^i_k \,,
\\
-\frac{k}{\sqrt3} |\xi_0^s|
&
\hbox{if } x\in C^i_k\,,
\end{cases}
\end{eqnarray*}
so that $v_k\in W^{1,\infty}(\Om;\Rn)$, $|v_k(x)-\xi_0^sx|\le |a|/k$ on $\Om$,
$Ev_k(x)=0$ on $A_k$, and
$Ev_k(x)=k\xi_0^s$ on $B_k\cup C_k$.
We note that
\begin{equation}\label{convmeas}
1_{A_k}\wto 1\,,
\qquad  \qquad k1_{B_k}\wto \tfrac12\,,
\qquad  \qquad k1_{C_k}\wto \tfrac12
 \quad\hbox{weakly}^* \hbox{ in }M_b(\Om)\,.
\end{equation}
Let $\Om_k$ be an increasing sequence of open sets, with union equal to $\Om$,
such that 
$0<{\rm dist}(\Om_k,\Rn\setmeno\Om)<2/\sqrt k$, and let $\varphi_k\in C^\infty_c(\Om)$
be cut-off functions such that $\varphi_k=1$ on $\Om_k$, $0\le\varphi_k\le 1$ on 
$\Om\setmeno\Om_k$, and $|\nabla\varphi_k|\le \sqrt k$ on $\Om$.

Let us define
\begin{eqnarray*}
&\uu_k(t,x):=\varphi_k(x)\big(\alpha(t)\xi_0^s x+ \beta(t)v_k(x)+t \omega x\big)+
(1-\varphi_k(x)) t\xi_0 x\,,
\\
&\pp_k(t,x):=\varphi_k(x) \beta(t)Ev_k(x)+
(1-\varphi_k(x)) t\xi_0^s\,,
\\
&\ee_k(t,x):= E \uu_k(t,x) - \pp_k(t,x) = \varphi_k(x)\alpha(t)\xi_0^s +
 \beta(t)\nabla \varphi_k(x){\,\odot\,} (v_k(x)-\xi_0^sx)\,,
 \\
&\zz_k(t,x):=\theta_0+\beta(t)\theta_k(x)\,.
\end{eqnarray*}
Then $\uu_k(t,x)\to t\xi_0 x$ and $e_k(t,x)\to\alpha(t)\xi_0^s$ uniformly on $\Om$
for every $t\in[0,+\infty)$.
As for $p_k$ and $z_k$, for every $f\in C^{\hom}(\ol\Om{\times}(\MD{\times}\R)^m{\times}\R)$
we have
\begin{eqnarray*}
&\displaystyle
\int_\Om f(x,p_k(t_1,x), z_k(t_1,x),\dots, p_k(t_m,x), z_k(t_m,x),1)\,dx=
\\
&\displaystyle
=\int_{A_k\cap\Om_k} f(x,0, \theta_0,\dots, 0, \theta_0,1)\,dx+{}
\\
&\displaystyle
+\int_{B_k\cap\Om_k} f(x,\beta(t_1)k\xi_0^s, \theta_0+\beta(t_1)k\tfrac{|\xi_0^s|}{\sqrt3},\dots,\beta(t_m)k\xi_0^s, \theta_0+\beta(t_m)k\tfrac{|\xi_0^s|}{\sqrt3},1)\,dx+{}
\\
&\displaystyle
{}+\int_{C_k\cap\Om_k} f(x,\beta(t_1)k\xi_0^s, \theta_0-\beta(t_1)k\tfrac{|\xi_0^s|}{\sqrt3},\dots,\beta(t_m)k\xi_0^s, \theta_0-\beta(t_m)k\tfrac{|\xi_0^s|}{\sqrt3},1)\,dx + 
{\mathcal R}_k=
\\
&\displaystyle
=\int_{A_k\cap\Om_k} f(x,0, \theta_0,\dots, 0, \theta_0,1)\,dx+{}
\\
&\displaystyle
{}+k\int_{B_k\cap\Om_k} f(x,\beta(t_1)\xi_0^s, \tfrac{\theta_0}{k}+\beta(t_1)\tfrac{|\xi_0^s|}{\sqrt3},\dots, \beta(t_m)\xi_0^s, \tfrac{\theta_0}{k}+\beta(t_m)\tfrac{|\xi_0^s|}{\sqrt3},\tfrac1k)\,dx
\\
&\displaystyle
{}+k\int_{C_k\cap\Om_k} f(x,\beta(t_1)\xi_0^s, \tfrac{\theta_0}{k}-\beta(t_1)\tfrac{|\xi_0^s|}{\sqrt3},\dots, \beta(t_m)\xi_0^s, \tfrac{\theta_0}{k}-\beta(t_m)\tfrac{|\xi_0^s|}{\sqrt3},\tfrac1k)\,dx+ {\mathcal R}_k\,,
\end{eqnarray*}
where the remainder ${\mathcal R}_k$ satisfies the estimate
$$
|{\mathcal R}_k|\le 
c\|f\|_{\hom} \big( \Ln(\Om\setmeno\Om_k)
+k\Ln(B_k\cap(\Om\setmeno\Om_k))
+ k\Ln(C_k\cap(\Om\setmeno\Om_k))\big)\,.
$$
By \eqref{convmeas} it follows that
\begin{eqnarray*}
&\displaystyle
\lim_{k\to\infty}
\int_\Om f(x,p_k(t_1,x), z_k(t_1,x),\dots, p_k(t_m,x), z_k(t_m,x),1)\,dx=
\\
&\displaystyle
=\int_{\Om} f(x,0, \theta_0,\dots, 0, \theta_0,1)\,dx+{}
\\
&\displaystyle
{}+\tfrac12\int_{\Om} f(x,\beta(t_1)\xi_0^s, \beta(t_1)\tfrac{|\xi_0^s|}{\sqrt3},\dots, \beta(t_m)\xi_0^s, \beta(t_m)\tfrac{|\xi_0^s|}{\sqrt3},0)\,dx+{}
\\
&\displaystyle
{}+\tfrac12\int_{\Om} f(x,\beta(t_1)\xi_0^s, \beta(t_1)\tfrac{|\xi_0^s|}{\sqrt3},\dots, \beta(t_m)\xi_0^s, \beta(t_m)\tfrac{|\xi_0^s|}{\sqrt3},0)\,dx\,.
\end{eqnarray*}
This proves \eqref{apprmeas} and concludes the proof of the theorem.
\end{proof}

\begin{remark}
For $t>t_0$ the globally stable quasistatic evolution described in Theorem~\ref{thm:ex1} is completely different from the approximable quasistatic evolution presented in 
\cite[Section~7]{DM-DeS-Mor-Mor-2}. The globally stable quasistatic evolution
contains the terms
$$
\tfrac12 \omega^{\Ln}_{((\pp(t_1),\zz_\infty(t_1)),\dots, (\pp(t_m),\zz_\infty(t_m)))}+\tfrac12 \omega^{\Ln}_{((\pp(t_1),-\zz_\infty(t_1)),\dots, (\pp(t_m),-\zz_\infty(t_m)))}\,,
$$
which describe oscillations of $\pp$ and $\zz$ with infinite amplitude and frequency, while the approximable quasistatic evolution is given by ordinary functions, without concentration or oscillation effects. Moreover the stress $\sigmaa$ of  the globally stable quasistatic evolution satisfies
$$
\sigmaa(t)=\frac{\sqrt3}{2}\frac{\xi_0^s}{|\xi_0^s|}\qquad\hbox{for every }t> t_0\,,
$$
while the stress $\sigmaa$ of  the approximable quasistatic evolution satisfies
$$
\sigmaa(t)\to\frac{\sqrt3}{2}\frac{\xi_0^s}{|\xi_0^s|}\qquad\hbox{as }t\to+\infty\,,
$$
but $|\sigmaa(t)|>\frac{\sqrt3}{2}$ for every $t>t_0$, as shown in 
\cite[Remark~7.7]{DM-DeS-Mor-Mor-2}.
\end{remark}

\end{section}

\bigskip

\noindent {\bf Acknowledgments.} { This work is part of the Project ``Calculus of Variations" 2006, 
supported by the Italian Ministry of University and Research, and of the research project ``Mathematical Challenges in Nanomechanics" supported by INdAM.}

\bigskip
\bigskip

{
}


\frenchspacing
\begin{thebibliography}{99}


\bibitem{Anz-Gia}Anzellotti G., Giaquinta M.:
On the existence of the field of stresses and displacements for an 
elasto-perfectly plastic body in static 
equilibrium. 
{\it J. Math. Pures Appl.\/} {\bf 61} (1982), 219-244.


\bibitem{Bre}Brezis H.: 
Op\'erateurs maximaux monotones et semi-groupes de contractions dans les 
espaces de Hilbert. 
North-Holland, Amsterdam-London; American Elsevier, New York, 1973.



\bibitem{DM-DeS-Mor}Dal Maso G., DeSimone A., Mora M.G.: 
Quasistatic evolution problems for linearly elastic - perfectly plastic materials.  
{\it Arch. Ration. Mech. Anal.\/} {\bf 180} (2006), 237-291.

\bibitem{DM-DeS-Mor-Mor-1}Dal Maso G., DeSimone A., Mora M.G., Morini M.: 
Time-dependent systems of generalized Young measures.
{\it Netw. Heterog. Media\/} {\bf 2} (2007), 1-36.

\bibitem{DM-DeS-Mor-Mor-2}Dal Maso G., DeSimone A., Mora M.G., Morini M.: 
A vanishing viscosity approach to quasistatic evolution in plasticity with softening.
{\it Arch. Ration. Mech. Anal.\/}, to appear.





\bibitem{Gof-Ser}Goffman C., Serrin J.: 
Sublinear functions of measures and variational integrals. 
{\it Duke Math. J.\/} {\bf 31} (1964), 159-178.

\bibitem{Han-Red}Han W., Reddy B.D.:
Plasticity. Mathematical theory and numerical analysis.
Springer Verlag, Berlin, 1999.

\bibitem{Hill}Hill R.: The mathematical theory of plasticity. Clarendon Press, Oxford, 1950.

\bibitem{Kachanov}Kachanov L.M.: Fundamentals of the theory of plasticity, Dover, 2004.





\bibitem{Lub}Lubliner J.: 
Plasticity theory. Macmillan Publishing Company, New York, 1990.

\bibitem{Luc-Mod}Luckhaus S., Modica L.:
The Gibbs-Thompson relation within the gradient theory of phase transitions.
{\it Arch. Ration. Mech. Anal.\/}{\bf 107} (1989), 71-83.


\bibitem{Mar}Martin J.B.: 
Plasticity. Fundamentals and general results. MIT Press, Cambridge, 1975.


\bibitem{MeySer}Meyers N.G., Serrin J.:
H=W. {\it Proc. Nat. Acad. Sci. U.S.A.\/} {\bf 51} (1964), 1055-1056.


%

\bibitem{Mie-review}Mielke A.:  Evolution of rate-independent systems. In: Evolutionary equations. Vol. II. Edited by C. M. Dafermos and E. Feireisl, 461-559, Handbook of Differential Equations. Elsevier/North-Holland, Amsterdam, 2005.

\bibitem{Mie-Ort}
Mielke A., Ortiz M.; A class of minimum principles for characterizing the trajectories and the relaxation of dissipative systems. 
{\it ESAIM Control Optim. Calc. Var.\/}, to appear.

\bibitem{Mie-Rou-Ste}
Mielke A., Roub\'i\v cek T., Stefanelli U.:
$\Gamma$-limits and relaxations for rate-independent evolutionary problems.
{\it Calc. Var. Partial Differential Equations\/}, to appear. 

%
%
%
%

\bibitem{Res}Reshetnyak Yu.G.:
Weak convergence of completely additive vector functions on a set.
{\it Siberian Math. J.\/} {\bf 9} (1968), 1039-1045.

\bibitem{Roc}Rockafellar R.T.: 
Convex Analysis. Princeton University 
Press, Princeton, 1970.



\bibitem{Tem}Temam R.: 
Mathematical problems in plasticity. 
Gauthier-Villars, Paris, 1985. 
Translation of Probl\`emes math\'ematiques en plasticit\'e.
Gauthier-Villars, Paris, 1983.



\end {thebibliography}
\end{document}